\documentclass[10pt,a4paper,english]{amsart}

\usepackage{amssymb}
\usepackage{amsthm}
\usepackage[T1]{fontenc}
\usepackage{lmodern}
\usepackage{ae}
\usepackage{enumerate}
\usepackage{mathtools}

\usepackage[utf8]{inputenc}
\usepackage{hyperref}
\hypersetup{
   colorlinks=true,
   linkcolor=blue,
   citecolor=blue,
   filecolor=red,
   urlcolor=green
}
\usepackage{geometry}
\usepackage{geometry}
\usepackage{fancyhdr}

\usepackage{stmaryrd}
\usepackage{amsmath}
\theoremstyle{plain} 
\newtheorem{theorem}{Theorem}[section]

\newtheorem{proposition}[theorem]{Proposition}
\newtheorem{lemma}[theorem]{Lemma}

\newtheorem{remark}[theorem]{Remark}
\numberwithin{equation}{section}

\newcommand{\diag}{\operatorname{diag}}

\newcommand{\tr}{\operatorname{tr}}

\renewcommand{\O}{  \mathcal{O}   }

\newcommand{\s}{  \sigma   }

\renewcommand{\phi}{  \varphi  }

\newcommand{\Card}{\operatorname{Card}}
\newcommand{\be}{\begin{equation}}
\newcommand{\ee}{\end{equation}}
\newcommand{\ben}{\begin{equation*}}
\newcommand{\een}{\end{equation*}}
\newcommand{\ban}{\begin{align*}}
\newcommand{\ean}{\end{align*}}
\newcommand{\lc}{\llbracket}
\newcommand{\rc}{\rrbracket}

\begin{document}
\title[\small{KAM for the nonlinear wave equation on the circle: A normal form theorem}]{KAM for the nonlinear wave equation on the circle: A normal form theorem}
\author{Moudhaffar Bouthelja}
\address{Laboratoire Paul-Painlev\'e, Universit\'e de Lille 1, UMR CNRS 8524, Cit\'e Scientifique, 59655 Villeneuve-d'Ascq}
\email{moudhaffar.bouthelja@math.univ-lille1.fr}
\thanks{This work was supported in part by the CPER Photonics4Society and the Labex CEMPI (ANR-11-LABX-0007-01)}

	\begin{abstract}
In this paper we prove a KAM theorem in infinite dimension which treats the case of multiple eigenvalues (or frequencies) of finite order. More precisely, we consider a Hamiltonian normal form in infinite dimension:
\begin{equation} \nonumber
h(\rho)=\omega(\rho).r + \frac{1}{2} \langle \zeta,A(\rho)\zeta \rangle,
\end{equation}
where $ r \in \mathbb{R}^n $, $\zeta=((p_s,q_s)_{s \in \mathcal{L}})$ and $ \mathcal{L}$ is a subset of $\mathbb{Z}$. We assume that the infinite matrix  $A(\rho)$ satisfies $A(\rho)= D(\rho)+N(\rho)$, where $D(\rho) =\operatorname{\diag} \left\lbrace  \lambda_{i} (\rho) I_2 ,\: 1\leq i \leq m\right\rbrace$ and $N$ is a bloc diagonal matrix. We assume that the size of each bloc of $N$ is the multiplicity of the corresponding eigenvalue in $D$.

In this context, if we start from a torus, then the solution of the associated Hamiltonian system remains on that torus. Under certain conditions emitted on the frequencies, we can affirm that the trajectory of the solution fills the torus. In this context,  the starting torus is an invariant torus. Then, we perturb this integrable Hamiltonian and we want to prove that the starting torus is a persistent torus. We show that, if the perturbation is small and under certain conditions of non-resonance of the frequencies, then the starting torus is a persistent torus.
	\end{abstract}
	\maketitle
	\tableofcontents
\section{Introduction}    
Kepler's laws predict that planetary orbits describe regular ellipses. In the eighteenth century, Newton's laws helped to better understand phenomena related to gravitation. Mathematicians then realized that Kepler's laws did not take into account the perturbations due to the interactions between the planets. The question then is whether these deviations are likely to significantly modify the trajectories of the planets.

In 1889, Poincaré showed that the series used to describe these perturbations were divergent. In other words, a small perturbation could possibly have an infinite contribution. This phenomenon was interpreted as a confirmation of the hypotheses of statistical mechanics.

In 1954, the situation changed once again after Kolmogorov's works. During a presentation at the International Congress of Mathematicians in Amsterdam, he briefly presented a result, according to which the solar system is probably stable. Instability is perfectly possible, as Poincaré said, but it happens very rarely. Indeed, Kolmogorov's theorem states that, if we start from a stable dynamic system (the solar system as imagined by Kepler) and add a small perturbation, then the system obtained remains stable for most of the initial data. In \cite{Kolmogorov54}, Kolmogorov gave only the broad lines of the proof. This discovery did not attract much interest from his contemporaries, so Kolmogorov did not continue his work in this direction.

Almost ten years later, in 1963, a student of Kolmogorov, Arnold, who was interested in the stability of planetary motion, came back to this approach. He proved that for quite small perturbations, almost all the trajectories remain close to the Kepler ellipse (see \cite{arno63a, arn63b}).

Independently, the same year, Moser developed generic techniques to solve problems related to disturbances, such as those studied by Kolmogorov (see \cite{mos62}).

All this work forms the basis of the KAM theory. For more history on the KAM theory see \cite{villani} and \cite{dumas2014kam}.

In recent years, significant progress has been made in KAM theory. For PDEs, everything starts in 1987 in \cite{Kuk87,kuk89}. In the second paper the author proves the existence of quasi-periodic solutions following the perturbation of an integrable hamiltonian in infinite dimension. In the second paper, the author proves the existence of quasi-periodic solutions after perturbing an integrable infinite dimension Hamiltonian. He assumes that the spectrum of the integrable hamiltonian is in the form of $\lambda_n \sim n^d$ with $n \geq 1$ and $d>1$. He then applies this result to the Schr\"odinger equation with potential, i.e. with external parameter, in dimension 1 with Dirichlet condition. In \cite{KukWay95}, he proves with P\"oschel a similar result for the Schr\"odinger equation, without parameter, in dimension 1 and with Dirichlet condition. This also implies the simplicity of the spectrum. See also \cite{Kuk98} for the Korteweg–de Vries equation.

Still in the context of KAM theory for PDEs, Wayne proves in 1990 in \cite {Wayne90}, using KAM methods, the existence of periodic and quasi-periodic solutions for the wave equation in dimension 1, with potential and Dirichlet condition.

In 1996, P\"oschel proves, in \cite{Poschel1996}, the existence of invariant tori of finite dimension in an infinite phase space, after a small perturbation of an integrable Hamiltonian. He assumes that the spectrum of the quadratic part of the integrable Hamiltonian is simple and satisfies:$$\lambda_s= s^d + \ldots + O(s^\delta),$$ where $d \geq 1$ and $\delta <d-1$. He then applies this result, in \cite{Poschel_Wave}, to the non-linear wave equation without an external parameter and with Dirichlet condition.

The first KAM result for the non-linear wave equation with periodic boundary conditions in dimension 1 is due to Chierchia and You in \cite{Chierchia_You} in 2000. In this paper, the authors prove the existence of quasi-periodic solutions for the wave equation with potential, i.e. with external parameter. They assume also that the non-linearity does not depends on the space variable.

More recently, in 2010, Eliasson and Kuksin succeeded in applying KAM theory to a multidimensional EDP. In \cite{eliasson2010kam}, they prove the existence of quasi-periodic solutions for the Schr\"odinger equation with potential in any dimension. For the proof they use a KAM theorem in infinite dimension, and such that the quadratic part of the integrable Hamiltonian admits an infinity of eigenvalue with any multiplicity.

In 2011 Grébert and Thomann proves in \cite{GrTh11} a KAM result in infinite dimension by improving results of Kuksin and P\"oschel \cite{kuksin1996invariant, Poschel_Wave}, and this by using the recent techniques of Eliasson and Kuksin \cite{eliasson2010kam}. They prove the existence of invariant tori of finite dimension, for a small perturbation of an integrable Hamiltonian whose external frequencies are of the form $\lambda_s \sim s$. They apply this result to prove the existence of quasi-periodic solutions for the Schr\"odinger equation in dimension 1 with harmonic potential. They also prove the reductibility of the Schr\"odinger equation with a harmonic potential quasi-periodic in time.

For more details about existing results for KAM theory for PDEs we can refer to \cite{berti16}. In this paper, the author provides an overview of the state of the art of KAM theory for PDEs. He gives several examples of Hamiltonian and reversible PDEs like the nonlinear wave, Klein–Gordon and Schrödinger equations, the water waves equations for fluids and some of its approximate models like the KdV (Korteweg de Vries) equation. He also gives a classification of the existing results. He distinguishes three categories depending on what we perturb. A first class for linear PDEs with parameters. A second class for integrable PDEs and a third one for normal form, i.e. we have to perform a Birkhoff normal form before applying a KAM result.

In this paper we use recent techniques developed by Eliasson-Grébert-Kuksin in \cite{EGK_final} and by Grébert-Paturel in \cite{grepat16_final}. In \cite{EGK_final}, the authors prove a KAM theorem in infinite dimension, which they apply to the multidimensional beam equation without external parameter and with a cubic non-linearity. They prove the existence of quasi-periodic solutions of low amplitude. In \cite{grepat16_final}, the authors obtain a similar result for the multidimensional Klein Gordon equation.

In this paper we prove an abstract KAM theorem (Theorem~\ref{theoreme kam}) that we apply to  the convolutive wave equation on the circle:
\begin{equation} \label{eq_onde_pot}
u_{tt} - u_{xx} + V \star u + \varepsilon g(x,u) = 0, \quad t \in \mathbb{R},\: x \in \mathbb{S}^1,
\end{equation}
in section 6. For simplicity we assume that $\Lambda:=-\partial_{xx}+ V \star >0$. Thanks to the potential $V$, considered as a parameter, the eigenvalues $(\lambda_a:= \sqrt{a^2+\hat{V}(a)}, \: a \in \mathbb{Z})$ of $\Lambda^{1/2}$ satisfy some suitable non-resonance conditions. This allows us to prove the existence of quasi-periodic solutions for generic potential $V$.  We also use this abstract KAM theorem to prove the existence of small amplitude quasi-periodic solutions for the nonlinear wave equation on the circle without parameter (see \cite{B2}). More precisely we consider the cubic wave equation on the circle:
\begin{equation} \label{eq_onde_m}
u_{tt} - u_{xx} + m u = 4u^3+ O(u^4), \quad t \in \mathbb{R},\: x \in \mathbb{S}^1,
\end{equation}
for $m \in \left[ 1,2 \right]$. The eigenvalues $(\lambda_a:= \sqrt{a^2+m}, \: a \in \mathbb{Z})$ of $\sqrt{-\partial_{xx}+m}$ are completely resonant. In order to satisfy the KAM non-resonance conditions for this case, we have to perform a Birkhoff normal form and "extract" from the non-linearity the integrable term in order "tune" the frequencies. In the two previous wave equations, we remark that $\lambda_a=\lambda_{-a}, a \in \mathbb{Z}$. Thus, the KAM theorem that we will use must deal with the case of multiple eigenvalues with finite order.

We begin the paper by stating a KAM result for a Hamiltonian $H=h+f$ of the following form: 
\begin{equation} \nonumber
H(\rho)=\omega(\rho).r + \frac{1}{2} \langle \zeta,A(\rho)\zeta \rangle+ f(r,\theta, \zeta; \rho),
\end{equation}
where
\begin{itemize}
\item[i)] $\rho \in \mathcal{D}$ is an external parameter. $\mathcal{D}$ is a compact set of $\mathbb{R}^p$.
\item[ii)] $ \omega$ is the frequencies vector corresponding to the internal modes in action-angle variables $(r,\theta) \in \mathbb{R}^n \times \mathbb{T}^n$.
\item[iii)]  $\zeta=(\zeta_s)_{s \in \mathcal{L}}$ are the external modes,  $ \mathcal{L}$ is an infinite set of indices of $\mathbb{Z}$, $\zeta_s= (p_s,q_s) \in \mathbb{R}^2$.
\item[iv)] $A$ is a block diagonal linear operator acting on the external modes.
\item[v)] $f$ is a perturbative hamiltonian depending  on all the modes.
\end{itemize}
Before giving the main result (Theorem~\ref{theoreme kam}), we detail the structure behind these object  and the hypothesis needed for the KAM result. In Section 3 we study the Hamiltonian flows generated by Hamiltonian functions. In Section 4 we detail the resolution of the homological equation. In section 5 we give the proof of the abstract KAM theorem \ref{theoreme kam}. In section 6 we apply the KAM theorem to the wave equation with a convolutive potential.
\section{Setting and abstract KAM theorem}
For $\mathcal{L}$ a set of $\mathbb{Z}$ and $\alpha\geq 0$, we define the $\ell_2$ weighted space:
\begin{equation} \nonumber
Y_\alpha:= \lbrace \zeta= \left( \zeta_{s} =\left( p_{s}, q_{s} \right)  , s \in \mathcal{L}  \right) | \Vert \zeta \Vert_\alpha < \infty \rbrace  ,
\end{equation}
where $$ \Vert \zeta \Vert_\alpha^2 = \sum_{s \in \mathcal{L}} \vert \zeta_s \vert^2 \langle s\rangle^{2\alpha}, \quad \langle s\rangle = \max(\vert s \vert , 1 ).$$
We endow $\mathbb{C}^2$ with the euclidean norm, i.e if $\zeta_s={}^t(p_s,q_s)$ then $|\zeta_s|=\sqrt{p_s^2+q_s^2}$.\\
We define the following linear operator on $Y_\alpha$
\begin{equation} \nonumber
J: \lbrace \zeta_{s} \rbrace \mapsto \lbrace \sigma_2 \zeta_{s} \rbrace, \mbox{ where } \sigma_2 = \begin{pmatrix}
   0 & -1 \\
   1 & 0 
\end{pmatrix}.
\end{equation}
For $\beta\geq 0$ we define the $\ell_\infty$ weighted space
$$ L_\beta= \lbrace  \left( \zeta_{s} =\left( p_{s}, q_{s} \right)  , s \in \mathcal{L}  \right) | \vert \zeta \vert_\beta < \infty \rbrace  ,$$
where
$$ \vert \zeta \vert_\beta= \underset{s \in \mathcal{L}}{\sup} \vert \zeta_s \vert \langle s\rangle^{\beta} .$$
For $\beta \leq s$, we have $Y_s \subset L_\beta.$ Consider the phase space $\mathcal{P}= \mathbb{T}^n \times \mathbb{R}^n \times Y_\alpha,$
that we endow with the following symplectic form
\begin{equation} \nonumber
dr \wedge d\theta + J d\zeta \wedge d\zeta.
\end{equation}
\textbf{Infinite matrices.} Consider the orthogonal projector $\Pi$ defined on the set of square matrices by
\begin{equation}\nonumber
\Pi: \: \mathcal{M}_{2\times2}(\mathbb{C}) \rightarrow \mathbb{S},
\end{equation}
where
\begin{equation} \nonumber
\mathbb{S}=\mathbb{C}I+\mathbb{C}\sigma_2, \quad \text{with} \quad \sigma_2=\begin{pmatrix}
   0 & -1 \\
   1 & 0
\end{pmatrix}.
\end{equation}
We introduce $\mathcal{M}$ the set  of infinite symmetric matrices $A: \mathcal{L} \times \mathcal{L} \rightarrow \mathcal{M}_2 \left( \mathbb{R} \right)$, that verify, for any $ s, s' \in \mathcal{L}$,
\begin{center}
 $A_s^{s'} \in \mathcal{M}_2 \left( \mathbb{R} \right)$,
 $A_s^{s'} = {}^tA_{s'}^s$ and  $\Pi A_{s}^{s'}=A_{s}^{s'}.$
\end{center}
We also define $\mathcal{M}_\alpha$, a subset of $\mathcal{M}$, by:
\begin{equation} \nonumber
A \in \mathcal{M}_\alpha \Leftrightarrow \vert A \vert_\alpha:= \underset{s,s' \in \mathcal{L}}{\sup} \langle s\rangle^{\alpha} \langle s'\rangle^{\alpha} \Vert A_s^{s'} \Vert_\infty < \infty .
\end{equation}
Let $n \in \mathbb{N}$, $\rho > 0$ and  $B$ be a Banach space. We define:\begin{equation} \nonumber
\mathbb{T}^n_\rho= \lbrace  \theta \in \mathbb{C}^n/ 2\pi \mathbb{Z}^n \vert \: \vert Im\theta \vert < \rho\rbrace
\end{equation}
and
\begin{equation} \nonumber
\mathcal{O}_\rho \left( B \right) = \left\lbrace x \in B \vert \Vert x \Vert_B < \rho \right\rbrace .
\end{equation}
For $\sigma , \mu \in  \left] 0 , 1\right[ $, we define
\begin{equation} \nonumber
\mathcal{O}^\alpha (  \sigma , \mu )  = \mathbb{T}^n_\sigma \times \mathcal{O}_{\mu^2} ( \mathbb{C}^n ) \times \mathcal{O}_\mu ( Y_\alpha )= \lbrace ( \theta,r,\zeta) \rbrace,
\end{equation}
\begin{equation} \nonumber
\mathcal{O}^{\alpha, \mathbb{R}}(  \sigma , \mu )= \mathcal{O}^\alpha (  \sigma , \mu ) \cap \lbrace \mathbb{T}^n \times \mathbb{R}^n \times Y_\alpha^{\mathbb{R}} \rbrace,
\end{equation}
where $ Y_\alpha^{\mathbb{R}}= \left\lbrace  \zeta \in Y_\alpha \: | \: \zeta = \left( \zeta_{s} = \begin{pmatrix}
\xi_s \\
\eta_s
\end{pmatrix}, \xi_s=\bar{\eta}_s \: s \in \mathcal{L}  \right) \right\rbrace$.\\
Let us denote a point in $\mathcal{O}^\alpha (  \sigma , \mu )$ as $x= (\theta,r, \zeta)$. A function on $\O^\alpha(\s,\mu)$ is real if it has a real value for any real $x$. We define:
\begin{equation} \nonumber
\Vert (r,\theta,\zeta) \Vert_\alpha=\max(|r|,|\theta|,\Vert \zeta \Vert_\alpha).
\end{equation}

\textbf{Class of Hamiltonian functions.}
Let $\mathcal{D}$ be a compact set of $\mathbb{R}^p$, called the parameters set from now on. Let $f:\mathcal{O}^{\alpha}(  \delta , \mu ) \times \mathcal{D} \rightarrow \mathbb{C}$ be a $\mathcal{C}^1$ function, real and holomorphic in the first variable, such that for all $\rho \in \mathcal{D}$, the maps
\begin{equation} \nonumber
\mathcal{O}^{\alpha}(  \delta , \mu ) \ni x \mapsto \nabla_\zeta f(x,\rho) \in Y_\alpha \cap L_\beta
\end{equation}
and
\begin{equation} \nonumber
\mathcal{O}^{\alpha}(  \delta , \mu ) \ni x \mapsto \nabla_{\zeta}^2f(x,\rho) \in \mathcal{M}_\beta,
\end{equation}
are holomorphic.  We define:
\begin{align*}
\left\vert f(x,.) \right\vert_\mathcal{D} = \underset{\rho \in \mathcal{D}}{sup} \left\vert f(x,\rho) \right\vert, &\quad \left\Vert \frac{\partial f}{\partial \zeta} (x,.) \right \Vert_\mathcal{D} = \underset{\rho \in \mathcal{D}}{sup}\left\Vert  \nabla_\zeta f(x,\rho) \right\Vert_\alpha ,\\
\left\vert \frac{\partial f}{\partial \zeta} (x,.) \right \vert_\mathcal{D} = \underset{\rho \in \mathcal{D}}{sup}\left\vert \nabla_\zeta f(x,\rho) \right\vert_\beta , &\quad \left\vert \frac{\partial^2 f}{\partial \zeta^2} (x,.) \right\vert_\mathcal{D} = \underset{\rho \in \mathcal{D}}{sup}\left\vert \ \nabla_\zeta^2 f(x,\rho) \right\vert_\beta.
\end{align*}
We denote by $\mathcal{T}^{\alpha,\beta}(\mathcal{D},\sigma,\mu)$ the space of functions $f$ that verify, for all $x \in \mathcal{O}^\alpha(\sigma,\mu)$, the following estimates:
\begin{equation} \nonumber
\left\vert f(x,.)\right\vert_\mathcal{D} \leq C, \quad \left\Vert \frac{\partial f}{\partial \zeta}  (x,.) \right\Vert_\mathcal{D} \leq \frac{C}{\mu}, \quad \left\vert \frac{\partial f}{\partial \zeta}  (x,.) \right\vert_\mathcal{D} \leq \frac{C}{\mu},\quad \left\vert \frac{\partial^2 f}{\partial \zeta^2}  (x,.) \right\vert_\mathcal{D} \leq \frac{C}{\mu^2}.
\end{equation}

For $f\in \mathcal{T}^{\alpha,\beta}(\mathcal{D},\sigma,\mu)$, we denote by $\lc f \rc^{\alpha,\beta} _{\sigma,\mu,\mathcal{D}}$ the smallest constant $C$ that satisfies the above estimates. If $\partial_\rho^j f \in \mathcal{T}^{\alpha,\beta}(\mathcal{D},\sigma,\mu)$ for $j\in \lbrace 0,1 \rbrace$, then for $\gamma >0$ we define:
\begin{equation} \nonumber
\lc f \rc^{\alpha,\beta,\gamma} _{\sigma,\mu,\mathcal{D}}=\lc f \rc^{\alpha,\beta} _{\sigma,\mu,\mathcal{D}}+\gamma \lc \partial_\rho f \rc^{\alpha,\beta} _{\sigma,\mu,\mathcal{D}}.
\end{equation}
\\ \textbf{ Hamiltonian equations.} Consider a $C^1$-Hamiltonian function, the Hamiltonian equations are given by:

\[ \left\{ 
\begin{aligned}
 \dot{r} &=& -\nabla_\theta f(r,\theta,\zeta),
 \\ \dot{\theta} &=& \nabla_r f(r,\theta,\zeta),
 \\ \dot{\zeta} &=& J \nabla_\zeta f(r,\theta,\zeta),
  \end{aligned}
 \right.
\]
\textbf{Poisson bracket.}
 Consider $f$ and $g$ two $C^1$-Hamiltonian function in $x=(r,\theta,\zeta)$. We define the Poisson bracket by:
\begin{equation}\nonumber
\lbrace f,g \rbrace = \nabla_r f.\nabla_\theta g -\nabla_\theta f.\nabla_r g + \langle \nabla_\zeta f, J\nabla_\zeta g \rangle.
\end{equation}
\textbf{Hamiltonian and normal form.}
Consider a Hamiltonian under the following form:
\begin{equation} \label{hamiltonien}
h(\rho)=\omega(\rho).r + \frac{1}{2} \langle \zeta,A(\rho)\zeta \rangle,
\end{equation}
where $ r \in \mathbb{R}^n , \: \zeta \in Y_\alpha$ and $A(\rho) \in \mathcal{M} $. Assume that $\omega\: : \mathcal{D} \: \rightarrow \mathbb{R}^n$ is a $C^1$ vector. Assume also that $A(\rho)$ is under the following form: $A(\rho)= D(\rho)+N(\rho)$. The matrix $D$ satisfies:
\begin{equation} \label{mat diag}
D(\rho) = \diag \left\lbrace  \lambda_{s} (\rho) I_2,\: s \in \mathcal{L}  \right\rbrace, 
\end{equation}
where \begin{enumerate}
\item[-] $ \lambda_s \geq \lambda_{s'}$ for $ s , s' \in \mathcal{L}$ and $s \geq s'$,
\item[-] $\Card \lbrace s' \in \mathcal{L} |\: \lambda_s=\lambda_{s'} \rbrace \leq d < \infty$ for $s \in \mathcal{L}$.
\end{enumerate}
$N$ is a bloc diagonal matrix and belongs to $\mathcal{M}$. We assume that the size of each bloc of $N$ is the multiplicity of the corresponding eigenvalue in $D$. We assume also that the coefficients of $N$ are under the following form: 
$\begin{pmatrix}
   \alpha & -\beta \\
   \beta & \alpha 
\end{pmatrix}$.

The matrix $N$ is a normal form if $N \in \mathcal{M}$ and satisfies the previous hypothesis. We denote $\mathcal{NF}$ the set of the normal form matrices.

We consider now the following complex change of variable:
\begin{equation} \nonumber
z_{j}= \begin{pmatrix}	  
   		\xi_{j} \\
  		 \eta_{j} 
  	\end{pmatrix}
 = \frac{1}{\sqrt{2}} \begin{pmatrix}	  
   		1 & i \\
  		1 & -i 
  	\end{pmatrix}
  \begin{pmatrix}	  
   		p_{j} \\
  		 q_{j} 
  	\end{pmatrix}	
 = \begin{pmatrix}	  
   		\frac{1}{\sqrt{2}}(p_{j}+iq_{j}) \\
  		\frac{1}{\sqrt{2}}(p_{j}-iq_{j}) 
  	\end{pmatrix}.
\end{equation}
If  $A \in \mathcal{NF}$, then, using the previous change of variable, we can transform the Hamiltonian \eqref{hamiltonien} under the following form:
\begin{equation} \nonumber
h(\rho)=\omega(\rho).r + \langle \xi,Q(\rho)\eta \rangle
\end{equation}
where $Q$ is hermitian matrix with complex coefficient.

The internal frequency vector $\omega$ and the matrices $D$ and $N$ verify hypotheses that will be stated in the following paragraph. For the following we fix two parameters $0 < \delta_0 \leq \delta \leq 1$. Assume that the eigenvalues of $D$ satisfy hypotheses A1, A2 and A3 and that $N $ satisfies hypothesis B.

\textbf{Hypothesis A1: Separation condition.} Assume that, for all $\rho \in \mathcal{D}$, we have:
\begin{itemize}
\item[$\star$] there exists a constant $c_0$ such that for all $s \in \mathcal{L}$,
\begin{equation} \nonumber
 \lambda_{s} (\rho)  \geq c_0 \langle s \rangle;
 \end{equation}
\item[$\star$] there exists a constant $c_1$ such that for all $s$, $s'$ $\in \mathcal{L}$ and $|s| \neq |s'|$, we have:
\begin{equation} \nonumber
 \vert \lambda_{s} (\rho) - \lambda_{s} ( \rho) \vert \geq c_1\left| \left| s \right| - \left| s \right| \right|.
\end{equation}
\end{itemize}
\textbf{Hypothesis A2: Transversality condition.} Assume that for all $\omega' \in \mathcal{C}^1 (\mathcal{D},\mathbb{R}^n)$ that satisfies
\begin{equation} \nonumber
\vert \omega - \omega' \vert_{\mathcal{C}^1 \left( \mathcal{D} \right)  } < \delta_0,
\end{equation}
\emph{for all $k \in \mathbb{Z}^n $, there exists a unit vector $z_k \in \mathbb{R}^p$, and all $s$, $s'$ $\in \mathcal{L}$ with $|s|>|s'|$ the following holds}:
\begin{itemize}
\item[$\star$]
\begin{equation} \nonumber
\vert k\cdot\omega' (\rho) \vert  \geq \delta,\quad \forall \rho \in \mathcal{D},
\end{equation}
or
\begin{equation} \nonumber
\langle \partial_\rho ( k\cdot\omega'(\rho)),z_k \rangle \geq \delta \quad \forall \rho \in \mathcal{D};
\end{equation}
where $k \neq 0$
\item[$\star$]
\begin{equation} \nonumber
\vert k\cdot\omega' (\rho)   \pm \lambda_{s}(\rho) \vert  \geq \delta \langle s \rangle, \quad \forall \rho \in \mathcal{D},
\end{equation}
or
\begin{equation} \nonumber
\langle \partial_\rho ( k\cdot\omega'(\rho)\pm \lambda_{s}(\rho)),z_k \rangle \geq \delta \quad \forall \rho \in \mathcal{D};
\end{equation}
\item[$\star$]
\begin{equation} \nonumber
\vert k\cdot\omega' (\rho)  + \lambda_{s}(\rho) + \lambda_{s'} ( \rho)  \vert \geq \delta( \langle s \rangle + \langle s' \rangle), \quad \forall \rho \in \mathcal{D},
\end{equation}
or
\begin{equation} \nonumber
\langle \partial_\rho ( k\cdot\omega'(\rho)+\lambda_{s}(\rho)+\lambda_{s'}(\rho)),z_k \rangle \geq \delta \quad \forall \rho \in \mathcal{D};
\end{equation}
\item[$\star$]
\begin{equation} \nonumber
\vert k\cdot\omega' (\rho) + \lambda_{s}(\rho) - \lambda_{s'} ( \rho) \vert \geq \delta ( 1 + \vert  \vert s \vert - \vert s' \vert \vert ), \quad \forall \rho \in \mathcal{D},
\end{equation}
or
\begin{equation} \nonumber
\langle \partial_\rho ( k\cdot\omega'(\rho)+\lambda_{s}(\rho) - \lambda_{s'} ( \rho)),z_k \rangle \geq \delta \quad \forall \rho \in \mathcal{D};
\end{equation}
\end{itemize}

\textbf{Hypothesis A3: Second Melnikov condition.}
Assume that for all $\omega' \in \mathcal{C}^1 (\mathcal{D},\mathbb{R}^n)$ that satisfies
\begin{equation} \nonumber
\vert \omega - \omega' \vert_{\mathcal{C}^1 \left( \mathcal{D} \right)  } < \delta_0,
\end{equation}
the following holds:
\\ for each $0 < \kappa < \delta$ and $N>0$ there exists a closed set $\mathcal{D}'\subset \mathcal{D}$ that satisfies
\begin{equation} \label{Melnikov-measure}
mes ( \mathcal{D} \setminus \mathcal{D}') \leq C (\delta^{-1} \kappa)^\tau N^\iota;
\end{equation}
for some $\tau,\: \iota >0$, such that for all  $\rho \in \mathcal{D}'$, all $ 0< \vert k \vert < N$ and all $s$,$s'\in \mathcal{L}$ with $|s| \neq |s'|$ we have:
\begin{equation} \label{Melnikov-cond}
\vert \omega'( \rho) \cdot k+\lambda_{s}(\rho)-\lambda_{s'}(\rho) \vert \geq \kappa(1+  \vert \vert s \vert - \vert s' \vert \vert).
\end{equation}

\textbf{Hypothesis B:} Assume that $N \in \mathcal{NF}$ and for all $\rho \in \mathcal{D}$ we have:
\begin{equation} \label{Hypothèse B}
\vert \partial^j_\rho N(\rho) \vert_\beta \leq \frac{\delta}{8},\:\:\: j=0,1.
\end{equation}
\begin{remark} \label{melnikov}
Assume that hypothesis A2 is satisfied, then for  $0< \kappa < \delta \leq \frac{1}{2}c_0$ and $N>1$ there exists a closed set $\mathcal{D}_1 = \mathcal{D}_1(\kappa,N) \subset \mathcal{D}$ such that:
\begin{equation*}
\operatorname{mes} ( \mathcal{D} \setminus \mathcal{D}_1) \leq C \kappa \delta^{-1} N^{2(n+1)},
\end{equation*}
where the constant $C$ depends on $\vert \omega \vert_{\mathcal{C}^1 \left( \mathcal{D} \right)}$ and $c_0$. For $\rho \in \mathcal{D}_1$ and $ \vert k \vert \leq N$ the following holds:
\begin{align*}
 &\vert k \cdot \omega'  \vert  \geq \kappa, \; \mbox{ if } k\neq 0, \\
 &\vert k \cdot \omega'  \pm \lambda_{s} \vert  \geq \kappa  \langle s\rangle ,\\
 &\vert k \cdot \omega'   + \lambda_{s} + \lambda_{s'}  \vert \geq \kappa  (\langle s\rangle +  \langle s'\rangle).
\end{align*}
\end{remark}
The remark is proven in the Appendix
\smallskip
\\Now we are able to state the abstract KAM theorem

\begin{theorem}\label{theoreme kam}
Consider the following Hamiltonian:
$$h(\rho)=\omega(\rho).r + \frac{1}{2} \langle \zeta,A(\rho)\zeta \rangle.$$
Assume that $A(\rho)=D(\rho)+N(\rho)$ where $D(\rho)$ is defined as in $\eqref{mat diag}$ and satisfies with the internal vector frequency $\omega$ the hypothesis A1, A2 and A3, while $N \in \mathcal{NF}$  and satisfies hypothesis B for fixed $\delta$ and $\delta_0$ and all $\rho \in \mathcal{D}$. Fix $\alpha,\beta>0$ and $0<\sigma,\mu \leq1$. Then there exists $\varepsilon_0$ depending on $d,n,\alpha,\beta, \sigma,\mu, |\omega_0|_{\mathcal{C}^1 \left( \mathcal{D} \right) }$ and $\vert  A_0 \vert_{\beta,{\mathcal{C}^1 \left( \mathcal{D} \right) }}$ such that, if  $\partial^j_\rho f\in \mathcal{T}^{\alpha,\beta}(\mathcal{D},\sigma, \mu)$ for $j=0,1$, if $$\lc f^T \rc^{\alpha,\beta,\kappa} _{\sigma,\mu,\mathcal{D}}=\varepsilon < \min(\varepsilon_0, \frac{1}{8} \delta_0) \mbox{ and  } \: \lc f \rc^{\alpha,\beta,\kappa} _{\sigma,\mu,\mathcal{D}} 
=O(\varepsilon^\tau),$$
for $\: 0 < \tau <1$, then there is a Borel set $\mathcal{D}' \subset \mathcal{D}$ with $mes ( \mathcal{D} \setminus \mathcal{D}') \leq  c(\sigma, \delta) \varepsilon^\gamma$ such that for all $\rho \in \mathcal{D}'$:
\begin{itemize}
\item there is a real symplectic analytical change of variable
\begin{equation} \nonumber
\Phi=\Phi_\rho: \mathcal{O}^\alpha(\frac{\sigma}{2}, \frac{\mu}{2}) \rightarrow \mathcal{O}^\alpha(\sigma,\mu)
\end{equation}
\item there is a new internal frequency vector $\tilde{\omega}(\rho) \in \mathbb{R}^n$, a matrix $\tilde{A}$ and a perturbation $\tilde{f} \in \mathcal{T}^{\alpha,\beta}(\mathcal{D}',\sigma/2, \mu/2)$ such that
\begin{equation}  \nonumber
(h_\rho + f ) \circ \Phi = \tilde{\omega}(\rho) \cdot r + \frac{1}{2}  \langle \zeta,\tilde{A}(\rho)\zeta \rangle + \tilde{f} (\theta,r,\zeta; \rho),
\end{equation}
where $ A:\mathcal{L}\times\mathcal{L} \to \mathcal{M}_{2\times 2}(\mathbb{R})$ is a block diagonal symmetric infinite matrix in $\mathcal{M}_\beta$ (ie $A_{[s]}^{[s']}=0$ if $[s] \neq [s']$). Moreover $\partial_r \tilde{f}=\partial_\zeta \tilde{f}=\partial_{\zeta\zeta}^2 \tilde{f}=0$ for $r=\zeta=0$. The mapping  $\Phi = (\Phi_\theta, \Phi_r,\Phi_\zeta)$ is close to identity, and for all $x \in  \mathcal{O}^\alpha(\frac{\sigma}{2}, \frac{\mu}{2})$ and all $\rho \in \mathcal{D}'$, we have:
\begin{equation} \label{kam_id}
 \Vert \Phi - Id \Vert_\alpha \leq C \varepsilon^{4/5}.
\end{equation}
For all $\rho \in \mathcal{D}'$, the new frequencies $\tilde{\omega}$ and the matrix $\tilde{A}$ satisfy
\begin{equation} \label{freq_kam}
\left|  \tilde{A}(\rho)-A(\rho))\right|_\alpha\leq C \varepsilon, \qquad \vert \tilde{\omega} (\rho) - \omega (\rho) \vert_{\mathcal{C}^1(\mathcal{D}')} \leq C \varepsilon,
\end{equation}
where $C$ is a constant that depends on $\varepsilon_0$.
\end{itemize}
\end{theorem}

\section{Jets of functions, Poisson bracket and Hamiltonian flow}
The space $\mathcal{T}^{\alpha,\beta}(\mathcal{D},\sigma,\mu)$ is not closed under the Poisson bracket. Therefor we will introduce the new subspace $\mathcal{T}^{\alpha,\beta+}(\mathcal{D},\sigma,\mu) \subset \mathcal{T}^{\alpha,\beta}(\mathcal{D},\sigma,\mu)$. We will prove that the Poisson bracket of a function from $\mathcal{T}^{\alpha,\beta+}(\mathcal{D},\sigma,\mu)$ and a function from $\mathcal{T}^{\alpha,\beta}(\mathcal{D},\sigma,\mu)$ belongs to $\mathcal{T}^{\alpha,\beta}(\mathcal{D},\sigma,\mu)$.
\subsection{The space $\mathcal{T}^{\alpha,\beta+}(\mathcal{D},\sigma,\mu)$}
We define the two spaces  $L_{\beta+}$ and $\mathcal{M}_{\beta+}$ by
\begin{equation} \nonumber
L_{\beta+}=\lbrace \zeta= \left( \zeta_{s} =\left( p_{s}, q_{s} \right)  , s \in \mathcal{L}  \right) | \quad \vert \zeta \vert_{\beta+} < \infty \rbrace,
\end{equation}
where
\begin{equation} \nonumber
\vert \zeta \vert_{\beta+}= \underset{s \in \mathcal{L}}{\sup} \vert \zeta_s \vert \langle s\rangle^{\beta+1},
\end{equation}
and
\begin{equation} \nonumber
\mathcal{M}_{\beta+}=\lbrace A \in \mathcal{M} | \: \vert A \vert_{\beta+} < \infty \rbrace,
\end{equation}
where
\begin{equation} \nonumber
\vert A \vert_{\beta+} = \underset{s,s' \in \mathcal{L}}{\sup} (1+| \: | s | - | s'| \:  |)\langle s\rangle^{\beta} \langle s'\rangle^{\beta} \Vert A_s^{s'} \Vert_\infty.
\end{equation}

We note that $L_{\beta+} \subset L_\beta$ and  $\mathcal{M}_{\beta+} \subset \mathcal{M}_{\beta}$. We define  $\mathcal{T}^{\alpha,\beta+}(\mathcal{D},\sigma,\mu)$ in the same way as we have defined $\mathcal{T}^{\alpha,\beta}(\mathcal{D},\sigma,\mu)$ but replacing $L_{\beta}$ by $L_{\beta+}$ and $\mathcal{M}_{\beta}$ by $\mathcal{M}_{\beta+}$. Hence we obtain that $\mathcal{T}^{\alpha,\beta+}(\mathcal{D},\sigma,\mu) \subset \mathcal{T}^{\alpha,\beta}(\mathcal{D},\sigma,\mu).$

\begin{lemma} \label{norme}
Let $\beta >0$, then there exists a positive constant $C$ that depends on $\beta$ such that:
\item[1.] Let $A \in \mathcal{M}_{\beta+}$ and $B \in \mathcal{M}_{\beta}$ then $AB,\:BA\: \in \mathcal{M}_{\beta}$ and:
$$\vert AB \vert_{\beta} \leq C \vert A \vert_{\beta+} \vert B \vert_\beta, \quad \vert BA \vert_{\beta} \leq  C \vert A \vert_{\beta+} \vert B \vert_\beta. $$
\item[2.] Let $A \in \mathcal{M}_{\beta+}$ and $\zeta \in L_\beta$ then $A\zeta \in L_\beta $ and:
$$\vert A \zeta \vert_{\beta} \leq C \vert A \vert_{\beta+} \vert \zeta \vert_{\beta}.$$
\item[3.] Let $A \in \mathcal{M}_{\beta}$ and $\zeta \in L_{\beta+}$ then $A\zeta \in L_\beta $ and:
$$\vert A \zeta \vert_{\beta} \leq C  \vert A \vert_{\beta} \vert \zeta \vert_{\beta+}.$$
\item[4.] Let $A \in \mathcal{M}_{\beta+}$ and $\zeta \in L_{\beta+}$ then $A\zeta \in L_{\beta} $ and:
$$\vert A \zeta \vert_{\beta+} \leq C  \vert A \vert_{\beta+} \Vert \zeta \Vert_{\beta+}.$$
\item[5.] Let $X\in L_\beta$ and $Y \in L_\beta$ then $A=X \otimes Y\in \mathcal{M}_\beta$ and:
$$ \vert A \vert_\beta \leq 2 \vert X \vert_\beta \vert Y \vert_\beta.  $$
\item[6.] Let $X\in L_{\beta+}$ and $Y \in L_{\beta+}$ then $A=X \otimes Y\in \mathcal{M}_{\beta+}$ and:
$$ \vert A \vert_{\beta+} \leq 2 \vert X \vert_{\beta+} \vert Y \vert_{\beta+}.  $$
\end{lemma}
The lemma is proven in the Appendix.
\subsection{Jets of functions}
For any function $f \in \mathcal{T}^{\alpha,\beta}(\mathcal{D},\sigma,\mu)$ we define its jet $f^T$ as the following Taylor polynomial of $f$ at $r=0$ and $\zeta=0$:
\begin{align*}
f^T (\theta,r,\zeta;\rho) & =  f(\theta,0,0,\rho)+\nabla_r f(\theta,0,0,\rho)r +\langle\nabla_\zeta f(\theta,0,0,\rho),\zeta\rangle +\frac{1}{2} \langle \nabla^2_{\zeta}f(\theta,0,0,\rho)\zeta,\zeta\rangle
\\ & =  f_\theta(\theta)+f_r(\theta)r+\langle f_\zeta(\theta),\zeta\rangle+\frac{1}{2} \langle f_{\zeta\zeta}(\theta)\zeta,\zeta\rangle.
\end{align*}
From the definition of the norm $\lc f \rc^{\alpha,\beta} _{\sigma,\mu,\mathcal{D}}$ we obtain the following estimations: 
\begin{equation} \label{estim deriv f-moyenne}
\begin{aligned}
\vert f_\theta(\theta;.) \vert_\mathcal{D} & \leq  \lc f \rc^{\alpha} _{\sigma,\mu,\mathcal{D}}, \:\:\:\:\:\:\:\:\:\:\:\:    \vert f_r(\theta;.) \vert_\mathcal{D} \leq \mu^{-2} \lc f \rc^{\alpha} _{\sigma,\mu,\mathcal{D}}, \\
\Vert f_\zeta(\theta;.) \Vert_\mathcal{D} & \leq  \mu^{-1}\lc f \rc^{\alpha} _{\sigma,\mu,\mathcal{D}}, \:\:\:\: \vert f_\zeta(\theta;.) \vert_\mathcal{D}  \leq  \mu^{-1}\lc f \rc^{\alpha,\beta} _{\sigma,\mu,\mathcal{D}} , \\
\vert f_{\zeta \zeta}(\theta;.) \vert_\mathcal{D} & \leq \mu^{-2} \lc f \rc^{\alpha,\beta} _{\sigma,\mu,\mathcal{D}},
\end{aligned}
\end{equation}
for $\theta \in \mathbb{T}^n_\sigma$.

We denote that for any $ \theta$ we have:
\begin{equation} \nonumber
f_\theta(\theta)= f_\theta^T(\theta), \quad f_r(\theta)=f_r^T(\theta), \quad f_\zeta(\theta)=f_\zeta^T(\theta), \mbox{  et  } f_{\zeta\zeta}(\theta)=f_{\zeta\zeta}^T(\theta).
\end{equation}
Hence we obtain:
\begin{equation} \label{estim deriv f-moyenne-jet}
\begin{aligned}
\vert f_\theta(\theta;.) \vert_\mathcal{D} & \leq  \lc f^T \rc^{\alpha} _{\sigma,\mu,\mathcal{D}}, \:\:\:\:\:\:\:\:\:\:\:\:    \vert f_r(\theta;.) \vert_\mathcal{D} \leq \mu^{-2} \lc f^T \rc^{\alpha} _{\sigma,\mu,\mathcal{D}}, \\
\Vert f_\zeta(\theta;.) \Vert_\mathcal{D} & \leq  \mu^{-1}\lc f^T \rc^{\alpha} _{\sigma,\mu,\mathcal{D}}, \:\:\:\: \vert f_\zeta(\theta;.) \vert_\mathcal{D}  \leq  \mu^{-1}\lc f^T \rc^{\alpha,\beta} _{\sigma,\mu,\mathcal{D}} , \\
\vert f_{\zeta \zeta}(\theta;.) \vert_\mathcal{D} & \leq \mu^{-2} \lc f^T \rc^{\alpha,\beta} _{\sigma,\mu,\mathcal{D}},
\end{aligned}
\end{equation}
for any $\theta \in \mathbb{T}^n_\sigma$.
\begin{lemma}\label{estimation jet fonction} 
For any $f \in \mathcal{T}^{\alpha,\beta}(\mathcal{D},\sigma,\mu)$ and $0 < \mu ' < \mu\leq 1$  we have:
\begin{equation}\label{estimation jet}
\lc f^T \rc^{\alpha,\beta} _{\sigma,\mu,\mathcal{D}} \leq 3 \lc f \rc^{\alpha,\beta} _{\sigma,\mu,\mathcal{D}},
\end{equation}
\begin{equation}\label{estimation difference jet}
\lc f - f^T \rc^{\alpha,\beta} _{\sigma,\mu ',\mathcal{D}} \leq 2 \left(  \frac{\mu'}{\mu} \right)^3  \lc f \rc^{\alpha,\beta} _{\sigma,\mu,\mathcal{D}}
\end{equation}
\end{lemma}
\begin{proof}
Let $(x,\rho) \in \mathcal{O}^\alpha (\sigma, \mu') \times \mathcal{D} $.
\begin{itemize}
\item
We start by proving the second estimate. We need to prove that:
\begin{itemize}
\item[$\cdot$] $\vert  ( f-f^T ) (x,\rho) \vert \leq 2  \left(  \frac{\mu'}{\mu} \right)^3  \lc f \rc^{\alpha,\beta} _{\sigma,\mu,\mathcal{D}}$,
\item[$\cdot$] $\Vert  \nabla_\zeta ( f-f^T ) (x,\rho) \Vert_\alpha \leq 2    \frac{\mu'^2}{\mu^3}   \lc f \rc^{\alpha,\beta} _{\sigma,\mu,\mathcal{D}}$,
\item[$\cdot$] $\vert  \nabla_\zeta ( f-f^T ) (x,\rho) \vert_\beta \leq 2    \frac{\mu'^2}{\mu^3}   \lc f \rc^{\alpha} _{\sigma,\mu,\mathcal{D}}$,
\item[$\cdot$] $\vert  \nabla_{\zeta}^2 ( f-f^T ) (x,\rho) \vert_\beta \leq 2    \frac{\mu'}{\mu^3}   \lc f \rc^{\alpha,\beta} _{\sigma,\mu,\mathcal{D}}$.
\end{itemize}
The prove of the four inequalities is the same. We choose to prove that
\begin{equation} \nonumber
\vert  \nabla_\zeta ( f-f^T ) (x,\rho) \vert_\beta \leq 2    \frac{\mu'^2}{\mu^3}   \lc f \rc^{\alpha} _{\sigma,\mu,\mathcal{D}}.
\end{equation}
Let us denote $m=\frac{\mu'}{\mu}$. For $\vert z \vert \leq 1 $ we have $( \theta, (z/m)^2r,(z/m)\zeta) \in \mathcal{O}(\sigma,\mu)$. Consider the function
\begin{center}
$\begin{array}{ccccl}
g & : & \left\lbrace  \vert  z  \vert < 1 \right\rbrace & \longrightarrow & Y^c \\
 & & z & \longmapsto & \nabla_\zeta f ( \theta, (z/m)^2r,(z/m)\zeta). \\
\end{array}$
\end{center}
$g$ is a holomorphic function bounded by $\mu^{-1} \lc f \rc^{\alpha,\beta} _{\sigma,\mu,\mathcal{D}}$ and we have
\begin{equation} \nonumber
g(z) = \underset{j\geq0}{\sum} f_j z^j.
\end{equation}
By Cauchy estimate we have: $\vert f_j \vert_\beta \leq \mu^{-1} \lc f \rc^{\alpha,\beta} _{\sigma,\mu,\mathcal{D}}  $. We remark that  $\nabla_\zeta ( f-f^T ) (x,\rho)=\underset{j\geq 2}{\sum} f_j m^j$. For $\mu'\leq \frac{1}{2} \mu$ we obtain that  $m\leq1/2$. So we have:
\begin{equation} \nonumber
\
\vert \nabla_\zeta ( f-f^T ) \vert_\beta \leq \mu^{-1} \lc f \rc^{\alpha\beta} _{\sigma,\mu,\mathcal{D}} \underset{j\geq 2}{\sum} m^j \leq \mu^{-1} \lc f \rc^{\alpha,\beta} _{\sigma,\mu,\mathcal{D}} 2 (\frac{\mu'}{\mu})^2
\end{equation}
\item Now let us prove the first estimate using the second one. We remark that $f^T=f-(f-f^T)$. The function $ (f-f^T),\:  \nabla_\zeta(f-f^T)$ and$\nabla^2_{\zeta} (f-f^T)$ are analytic on $\mathcal{O}^\alpha (\sigma, \frac{1}{2} \mu)$, so $f^T,\:  \nabla_\zeta f^T,$ and $\nabla^2_{\zeta} f^T$ are analytic on $\mathcal{O}^\alpha (\sigma, \frac{1}{2} \mu)$. We obtain that:
$$ \lc f^T \rc^{\alpha,\beta} _{\sigma,\frac{1}{2}\mu,\mathcal{D}} \leq \frac{1}{4} \lc f \rc^{\alpha,\beta} _{\sigma,\mu,\mathcal{D}} + \lc f - f^T \rc^{\alpha,\beta} _{\sigma,\frac{1}{2}\mu ,\mathcal{D}} \leq \frac{1}{2}  \lc f \rc^{\alpha,\beta} _{\sigma,\mu,\mathcal{D}}. $$
Since $f^T$ is quadratic in $\zeta$, then $f^T$, $ \nabla_ \zeta f^T$ et $ \nabla^2_{\zeta}f^T$  are analytic on  $\mathcal{O}^\alpha (\sigma, \mu)$. Since $\lc f^T \rc^{\alpha,\beta} _{\sigma,\mu,\mathcal{D}} \leq 4 \lc f^T \rc^{\alpha,\beta} _{\sigma,\frac{1}{2}\mu,\mathcal{D}}$ then
$$ \lc f^T \rc^{\alpha,\beta} _{\sigma,\mu,\mathcal{D}} \leq 2 \lc f \rc^{\alpha,\beta} _{\sigma,\mu,\mathcal{D}}.$$
\item Let us return to the second estimate in the case where $\frac{\mu}{2} < \mu' < \mu$. We have:
$$\lc f - f^T \rc^{\alpha,\beta} _{\sigma,\mu ',\mathcal{D}} \leq  \lc f \rc^{\alpha,\beta} _{\sigma,\mu,\mathcal{D}} +  \lc f^T \rc^{\alpha,\beta} _{\sigma,\mu,\mathcal{D}} \leq 3  \lc f \rc^{\alpha,\beta} _{\sigma,\mu,\mathcal{D}} $$
\end{itemize}
\end{proof}
\subsection{Jets of functions and Poisson bracket.}
Recall that the Poisson bracket of two $C^1$ functions $f$ and $g$ is defined by:
\begin{equation}\nonumber
\lbrace f,g \rbrace = \nabla_r f.\nabla_\theta g -\nabla_\theta f.\nabla_r g + \langle \nabla_\zeta f, J\nabla_\zeta g \rangle.
\end{equation}
\begin{lemma} \label{estim crochet de poisson}
Consider $f\in \mathcal{T}^{\alpha,\beta+}(\mathcal{D},\sigma,\mu)$ and $g\in \mathcal{T}^{\alpha,\beta}(\mathcal{D},\sigma,\mu)$ two jet functions, then for any $0<\sigma'<\sigma$ we have $\lbrace f,g \rbrace$ belongs to $\mathcal{T}^{\alpha,\beta}(\mathcal{D},\sigma,\mu)$ and 
\begin{equation}\label{inegalite crochet de poisson}
\lc \left\lbrace   f, g \right\rbrace    \rc^{\alpha,\beta} _{\sigma',\mu,\mathcal{D}} \leq C ( \sigma - \sigma' )^{-1} \mu^{-2} \lc f \rc^{\alpha,\beta+} _{\sigma,\mu,\mathcal{D}} \lc g \rc^{\alpha,\beta} _{\sigma,\mu,\mathcal{D}}\:,
\end{equation}
where $C$ depends only on $\beta$.
\end{lemma}
\begin{proof}
To prove this lemma, we have to show that
\begin{itemize}
\item $\vert \left\lbrace  f, g \right\rbrace  \vert _{\mathcal{D}} \leq C (\sigma-\sigma')^{-1}\mu^{-2}\lc f \rc^{\alpha,\beta+} _{\sigma,\mu,\mathcal{D}} \lc g \rc^{\alpha,\beta} _{\sigma,\mu,\mathcal{D}}$,
\vspace{0.2cm}
\item $\Vert \nabla_\zeta \left\lbrace  f, g \right\rbrace  \Vert _{\mathcal{D}} \leq C (\sigma-\sigma')^{-1}\mu^{-3}\lc f \rc^{\alpha,\beta+} _{\sigma,\mu,\mathcal{D}} \lc g \rc^{\alpha,\beta} _{\sigma,\mu,\mathcal{D}}$,
\vspace{0.2cm}
\item $\vert \nabla_\zeta \left\lbrace  f, g \right\rbrace  \vert _{\mathcal{D}} \leq C (\sigma-\sigma')^{-1}\mu^{-3}\lc f \rc^{\alpha,\beta+} _{\sigma,\mu,\mathcal{D}} \lc g \rc^{\alpha,\beta} _{\sigma,\mu,\mathcal{D}}$,
\vspace{0.2cm}
\item $\vert \nabla^2_{\zeta} \left\lbrace  f, g \right\rbrace  \vert _{\mathcal{D}} \leq C (\sigma-\sigma')^{-1}\mu^{-4}\lc f \rc^{\alpha,\beta+} _{\sigma,\mu,\mathcal{D}} \lc g \rc^{\alpha,\beta} _{\sigma,\mu,\mathcal{D}}$.
\end{itemize}
\vspace{0.2cm}
\textbf{-} Let us start with with the last estimate. we have
\begin{equation} \nonumber
\nabla^2_{\zeta} \left\lbrace  f, g \right\rbrace = f_r(\theta) \nabla_\theta g_{\zeta\zeta}(\theta)-  g_r(\theta) \nabla_\theta f_{\zeta\zeta} (\theta) + f_{\zeta\zeta}(\theta) J g_{\zeta\zeta}(\theta)
\end{equation} 
Using the estimates \eqref{estim deriv f-moyenne} for the first tho terms, the first estimate from Lemma~\ref{norme}  and Cauchy estimate for the last term, we obtain:
\begin{equation} \nonumber
\vert \nabla^2_{\zeta} \left\lbrace  f, g \right\rbrace  \vert _{\mathcal{D}} \leq  (\sigma-\sigma')^{-1}\mu^{-4}\lc f \rc^{\alpha,\beta+} _{\sigma,\mu,\mathcal{D}} \lc g \rc^{\alpha,\beta} _{\sigma,\mu,\mathcal{D}}.
\end{equation}
\\ \textbf{-} For the second and third estimate we have:
\begin{align*} \nonumber
\nabla_\zeta \left\lbrace  f, g \right\rbrace & = f_r (\theta) \nabla_\theta g_\zeta (\theta) + f_r(\theta) \nabla_\theta g_{\zeta\zeta} (\theta) \zeta - g_r (\theta) \nabla_\theta f_\zeta (\theta)\\ 
& - g_r(\theta) \nabla_\theta f_{\zeta\zeta} (\theta) \zeta + g_{\zeta\zeta} (\theta) J f_\zeta (\theta) + f_{\zeta\zeta} (\theta) J g_\zeta (\theta) + g_{\zeta\zeta} (\theta) J f_{\zeta\zeta} (\theta) \zeta.  
\end{align*}
By estimates \eqref{estim deriv f-moyenne}, estimates 1, 2 and 3 from Lemma~\ref{norme} and Cauchy estimate, there exists a constant that depends on $ \beta$ such that:
\begin{align*}
\vert \nabla_\zeta \left\lbrace  f, g \right\rbrace  \vert _{\mathcal{D}} & \leq C \mu^{-2} \lc f \rc _{\sigma,\mu,\mathcal{D}} (\sigma - \sigma')^{-1} \mu^{-3} \lc g \rc _{\sigma,\mu,\mathcal{D}}
 + C \mu^{-2} \lc f \rc _{\sigma,\mu,\mathcal{D}} (\sigma - \sigma')^{-1} \mu^{-2} \lc g \rc _{\sigma,\mu,\mathcal{D}} \mu \\
& + C \mu^{-2} \lc g \rc _{\sigma,\mu,\mathcal{D}} (\sigma - \sigma')^{-1} \mu^{-1} \lc f \rc _{\sigma,\mu,\mathcal{D}}
 + C \mu^{-2} \lc g \rc _{\sigma,\mu,\mathcal{D}} (\sigma - \sigma')^{-1} \mu^{-2} \lc f \rc _{\sigma,\mu,\mathcal{D}} \mu \\
& + C \mu^{-2} \lc g \rc _{\sigma,\mu,\mathcal{D}}  \mu^{-1} \lc f \rc _{\sigma,\mu,\mathcal{D}}
 + C \mu^{-2} \lc f \rc _{\sigma,\mu,\mathcal{D}} (\sigma - \sigma')^{-1} \mu^{-2} \lc g \rc _{\sigma,\mu,\mathcal{D}} \mu \\
& \leq C (\sigma-\sigma')^{-1}\mu^{-3}\lc f \rc _{\sigma,\mu,\mathcal{D}} \lc g \rc _{\sigma,\mu,\mathcal{D}}.
\end{align*}
Similarly we prove the first and the second estimate.
\end{proof}
\subsection{Hamiltonian flow in $\mathcal{O}^\alpha(\sigma,\mu)$.} Consider a $C^1$-function $f$ on $\mathcal{O}^\alpha(\sigma,\mu)\times \mathcal{D}$. We denote by $\Phi_f^t\equiv \Phi^t$ the Hamiltonian flow of $f$ at time $t$. Assume that $f$ is a jet function:
\begin{equation} \nonumber
f=f_\theta(\theta;\rho)+f_r(\theta;\rho)r+\langle f_\zeta(\theta;\rho),\zeta\rangle+\frac{1}{2} \langle f_{\zeta\zeta}(\theta;\rho)\zeta,\zeta\rangle.
\end{equation}
The Hamiltonian system associated is
\begin{equation} \label{hameq-jet}
 \left\{ 
\begin{aligned} 
 \dot{r} &= -\nabla_\theta f(r,\theta,\zeta),
 \\ \dot{\theta} &= f_r(\theta),
 \\ \dot{\zeta} &= J (f_\zeta (\theta) + f_{\zeta\zeta}(\theta)\zeta).
\end{aligned}
\right.
\end{equation}
We denote by $V_f = (V_f^r,V_f^\theta,V_f^\zeta) \equiv (\dot{r}, \dot{\theta}, \dot{\zeta})$ the corresponding Hamiltonian vector field. It is analytic on any domain $\mathcal{O}(\sigma-2\eta, \mu-2\nu)$ where $0<2\eta<\sigma\leq 1$ and $0<2\nu<\mu \leq 1$.The flow maps $\Phi_f^t$ of $V_f$ are analytic on $\mathcal{O}(\sigma-2\eta, \mu-2\nu)$ as long as they exist. We will study them as long as they map $\mathcal{O}(\sigma-2\eta, \mu-2\nu)$ to $\mathcal{O}(\sigma, \mu)$.
\\Assume that
\begin{equation} \label{hyp-f}
\lc f \rc^{\alpha,\beta} _{\sigma,\mu,\mathcal{D}} \leq \frac{1}{2} \eta\nu^2,
\end{equation}
then for $x \in \mathcal{O}(\sigma-2\eta, \mu-2\nu)$ and by Cauchy estimate we have:
\[ \left\{ 
\begin{aligned}
 \dot{r} &= -\nabla_\theta f(r,\theta,\zeta)\text{ et donc } \vert \dot{r} \vert_{\mathbb{C}^n} \leq (2\eta)^{-1} \lc f \rc^\alpha _{\sigma,\mu,\mathcal{D}} \leq \nu^2,
 \\ \dot{\theta} &= f_r(\theta)\text{ et donc } \vert \dot{\theta} \vert_{\mathbb{C}^n} \leq (4\nu)^{-2} \lc f \rc^\alpha _{\sigma,\mu,\mathcal{D}} \leq \eta,
 \\ \dot{\zeta} &= J (f_\zeta (\theta) + f_{\zeta\zeta}(\theta)\zeta) \text{ et donc } \Vert \dot{\zeta} \Vert_{\mathbb\alpha} \leq (\mu^{-1}+\mu^{-2}\mu) \lc f \rc^\alpha _{\sigma,\mu,\mathcal{D}} \leq \nu.
\end{aligned}
\right.
\] 
Note that $r(t)=\int_0^t\dot{r}(\tau)d\tau+r(0)$, then for $0\leq t \leq 1$ we have $\vert r(t) \vert_{\mathbb{C}^n} \leq (\mu-\nu)^2$. Similarly we obtain that $\vert \theta(t) \vert_{\mathbb{C}^n}\leq \sigma-\eta$ and $\Vert  \zeta(t)\Vert_\alpha \leq \mu-\nu$ for $0\leq t \leq 1$. This proves that the flow maps
\begin{equation} \nonumber
\Phi_f^t : \: \mathcal{O}^\alpha(\sigma-2\eta, \mu-2\nu) \rightarrow \mathcal{O}^\alpha(\sigma-\eta, \mu-\nu),
\end{equation}
are well defined $0\leq t \leq 1$ and analytic.
\\ For $x=(r,\theta, \zeta) \in \mathcal{O}(\sigma-2\eta, \mu-2\nu) $ and $0 \leq t \leq 1$ we denote $\Phi_f^t(x)= (r(t), \theta(t), \zeta (t))$. Let us give some details about the Hamiltonian flow $\Phi_f^t$.
\begin{enumerate}
\item[$\clubsuit$] We remark that $V_f^\theta=\dot{\theta}=f_r(\theta)$ is independent from $r$ and $\zeta$. Then $\theta(t)=K(\theta;t)$ where $K$ is analytic in $\theta$ and $t$.
\item[$\clubsuit$] We note that  $V_f^\zeta= \dot{\theta}= J f_\zeta (\theta(t))+ J f_{\zeta \zeta}(\theta(t))\zeta$. Using Cauchy estimate, $Jf_{\zeta\zeta}(\theta(t))$ is a linear bounded operator on $Y^c_\alpha$. Since $\theta(t)=K(\theta;t)$ where $K$ is analytic in $\theta$, then $V_f^\zeta$ is also analytic in à $\theta$. Therefore $\zeta (t)= T(\theta;t)+ U(\theta;t)\zeta$ where $U$ is a linear operator bounded on $Y^c$ to $L^c_\beta$. Both $T$ and $U$ are analytic in $\theta$.
\item[$\clubsuit$] The vector $V_f^r=\dot{r}= -\nabla_\theta f(r,\theta,\zeta)$ is quadratic in $\zeta$ and linear in $r$. Then $r(t)=L(\theta, \zeta; t)+S(\theta;t)r$ where $L$ is quadratic in $\zeta$ and analytic in $\theta$ and $S$ is an $ n \times n $ matrix analytic in  $\theta$.
\end{enumerate} 
\begin{lemma} 
Let $0<2\eta<\sigma\leq1$, $0<2\nu<\mu\leq1$ and $f=f^T \in \mathcal{T}^{\alpha,\beta}(\mathcal{D}, \sigma,\mu)$ that satisfies \eqref{hyp-f}. Then for $0 \leq t < 1$, the Hamiltonian flow maps $\Phi_f^t$ of equations \eqref{hameq-jet} define an analytic symplectomorphisms from $ \mathcal{O}^\alpha(  \delta -2\eta, \mu-2\nu )$  to  $ \mathcal{O}^\alpha(  \delta -\eta, \mu-\nu )$. They are of the form:

\begin{equation} \label{flot ham}
\Phi_f^t: 
 \left( 
\begin{aligned} 
 r
 \\ \theta
 \\ \zeta
\end{aligned}
\right) \rightarrow
\left(  \begin{aligned} 
& L(\theta,\zeta;t) + S(\theta;t)r
 \\& K(\theta;t)
 \\& T(\theta;t)+U(\theta;t) \zeta
\end{aligned} \right) 
= \left(  \begin{aligned} 
 r(t)
 \\ \theta(t)
 \\ \zeta(t)
\end{aligned} \right) 
\end{equation}
where $L(\theta,\zeta;t)$ is quadratic in $\zeta$, $U(\theta;t)$ and $S(\theta,\zeta;t)$ are linear operators in corresponding spaces. All the components of $\Phi_f^t$ are bounded and analytic in $\theta$.
\end{lemma} 
\begin{proposition} \label{estim-phi}
Consider $0<2\eta<\sigma\leq1$, $0<2\nu<\mu\leq1$ and $f=f^T \in \mathcal{T}^{\alpha,\beta+}(\mathcal{D}, \sigma,\mu)$ that satisfy $\lc f \rc^{\alpha,\beta+} _{\sigma,\mu,\mathcal{D}} \leq \frac{1}{2} \eta\nu^2$. Then for $0 \leq t < 1$ we have:
\begin{itemize}
\item[1)] Mappings $T, K$ and operators $U$ and $S$ are analytic in  $\theta \in \mathbb{T}^n_{\sigma-2\eta}$. Mapping $L$ is analytic in $(\theta,\zeta) \in \mathbb{T}^n_{\sigma-2\eta} \times Y^c_\alpha$. Their norms and operator norms satisfy:
\begin{equation} \label{estim comp flot ham}
\vert S(\theta;t) \vert_{\mathcal{L}(\mathbb{C}^n,\mathbb{C}^n)},\: \Vert U(\theta;t) \Vert_{\mathcal{L}(Y_\alpha,Y_\alpha)}, \: \Vert ^tU(\theta;t) \Vert_{\mathcal{L}(Y_\alpha,Y_\alpha)}, \: \vert U(\theta;t)\vert_{\beta+} \leq 2.
\end{equation}
For any component $L^j$, of $L$ and any $x=(r,\theta, \zeta) \in \mathcal{O}(\sigma-2\eta, \mu-2\nu) $ we have:
\begin{equation} \label{estim comp L}
\begin{aligned}
\Vert \nabla_\zeta L^j (x;t) \Vert_\alpha & \leq 8 \eta^{-1} \mu^{-1} \lc f \rc^{\alpha} _{\sigma,\mu,\mathcal{D}}, \\
\vert \nabla_\zeta L^j (x;t) \vert_{\beta+} & \leq 8 \eta^{-1} \mu^{-1} \lc f \rc^{\alpha,\beta+} _{\sigma,\mu,\mathcal{D}}, \\
\vert \nabla^2_{\zeta} L^j (x;t) \vert_{\beta+} & \leq 4 \eta^{-1} \mu^{-2} \lc f \rc^{\alpha,\beta+} _{\sigma,\mu,\mathcal{D}}.
\end{aligned}   
\end{equation}
\item[2)]  The Hamiltonian flow maps $\Phi_f^t$ of equations  \eqref{hameq-jet} analytically extend from $\mathbb{C}^n \times \mathbb{T}^n_{\sigma-2\eta} \times  Y^c_\alpha$ to $\mathbb{C}^n \times \mathbb{T}^n_{\sigma} \times  Y^c_\alpha$. Furthermore they satisfy
\begin{equation} \nonumber
\begin{aligned}
 \vert r(t) - r^0  \vert_{\mathbb{C}^n} & \leq C_0 \eta^{-1}(1 + \mu^{-2}\Vert \zeta^0\Vert^2_\alpha+\mu^{-2} \vert r^0 \vert) \lc f \rc^\alpha_{\sigma,\mu,\mathcal{D}},
 \\ \vert \theta(t) - \theta^0  \vert_{\mathbb{C}^n} & \leq \mu^{-2} \lc f \rc^\alpha_{\sigma,\mu,\mathcal{D}},
 \\ \Vert \zeta(t) - \zeta^0  \Vert_\alpha & \leq  (1+ \mu^{-2} \Vert \zeta^0 \Vert_\alpha) \lc f \rc^\alpha_{\sigma,\mu,\mathcal{D}},
 \\ \Vert \zeta (t) - \zeta^0 \Vert_{\beta+} & \leq C_1 (1+ \mu^{-2} \Vert \zeta^0 \Vert_\alpha) \lc f \rc^{\alpha,\beta+}_{\sigma,\mu,\mathcal{D}},
\end{aligned}
\end{equation}
where $C_0$ is an absolute constant, while $C_1$ depends on $\beta$.
\end{itemize}
\end{proposition}
\begin{remark}
$\partial_\rho x(t)$ satisfies the same estimates as $x(t)$.
\end{remark}
\begin{proof}
\begin{itemize}
\item Let us start with the estimate on $\theta$.
From \eqref{hameq-jet} we have:
 \[ \left\{ 
\begin{aligned}
 \dot{\theta}(t) &= \nabla_r f(\theta(t)),
 \\ \theta (0) &= \theta^0 \in \mathbb{T}^n_{\sigma-2s}
\end{aligned}
\right.
\]
Consider $\bar{t}= \sup \left\lbrace t \mid \theta(u) \: \text{defined for} \: 0 \leq u \leq t ; \: \vert \theta(u) - \theta^0 \vert \leq \mu^{-2}\lc f \rc_{\sigma,\mu} \right\rbrace $ 
For any $t\leq \bar{t}$ we have:
\begin{equation} \nonumber
\theta (t) = \theta^0 + \int_0^t \nabla_rf(\theta(u))du.
\end{equation}
From \eqref{estim deriv f-moyenne} we know that $\Vert \nabla_r f \Vert \leq \mu^{-2} \lc f \rc^\alpha_{\sigma,\mu,\mathcal{D}}$, which leads to the desired estimate.
\item Let us proof the estimates on $\zeta$. From \eqref{hameq-jet} we have:
\begin{equation} \label{eq-zeta} \left\{ 
\begin{aligned}
 \dot{\zeta}(t) &= a(t) + B(t)\zeta(t),
 \\ \zeta (0) &= \zeta^0 \in \mathcal{O}_{\mu-2\nu}(Y^c_\alpha),
\end{aligned}
\right.
\end{equation} 
where $a(t):= Jf_\zeta(\theta(t))$ and $B(t) := Jf_{\zeta \zeta}(\theta(t))$. By Cauchy estimate and the fact that $\lc f \rc^{\alpha,\beta+} _{\sigma,\mu,\mathcal{D}} \leq \frac{1}{2} \eta \nu^2 $, we have: 
\begin{equation} \nonumber
\Vert a(t) \Vert_\alpha \leq \mu^{-1} \lc f \rc^\alpha_{\sigma,\mu,\mathcal{D}} \leq \nu\, \quad \Vert B(t)\Vert_{\mathcal{L}(Y_\alpha,Y_\alpha)} \leq \mu^{-2} \lc f \rc^\alpha_{\sigma,\mu,\mathcal{D}} \leq \frac{1}{2} \nu \leq \frac{1}{2}
\end{equation}
Let us rewrite \eqref{eq-zeta} in the integral form and iterating the process: 
\begin{eqnarray*}
\zeta(t) & = & \zeta^0 + \int_0^ta(t')+B(t')\zeta(t')dt'
\\ & = & \int_0^t a(t')dt'+\int_0^t \int_0^{t'}a(t'')B(t')d'' dt' +\zeta^0 + \zeta^0 \int_0^t B(t')dt\\
& + & \int_0^t\int_0^{t'} B(t')B(t'')\zeta(t'')dt''dt' \\
& = & \ldots \\
& = & a^\infty (t) + \left(  B^\infty (t) + I \right) \zeta^0,
\end{eqnarray*}
where
\begin{equation} \nonumber
a^\infty (t) = \underset{k \geq 1}{ \sum} \int_0^t \int_0^{t_1} \ldots \int_0^{t_k-1} \overset{k-1}{\underset{j=1}{\prod}} B(t_{j})a(t_k) dt_k \ldots dt_2 dt_1,
\end{equation}
and
\begin{equation} \nonumber
B^\infty (t) = \underset{k \geq 1}{ \sum} \int_0^t \int_0^{t_1} \ldots \int_0^{t_k-1} \overset{k-1}{\underset{j=1}{\prod}} B(t_{j}) dt_k \ldots dt_2 dt_1.
\end{equation}
We have:   $ \Vert \overset{k}{\underset{j=1}{\prod}} B(t_{j}) \Vert_{\mathcal{L}(Y_\alpha,Y_\alpha)}  \leq \left( \frac{1}{2} \right)^{k-1} \mu^{-2} \lc f \rc^{\alpha}_{\sigma,\mu,\mathcal{D}}$, then the operator $ B^\infty (t)$ is well defined and bounded for $t \in \lc  0\: 1 \rc $ by the convergent series $ \underset{k \geq 1}{ \sum} \frac{1}{k!} \left( \frac{1}{2} \right)^{k-1} \mu^{-2} \lc f \rc^\alpha_{\sigma,\mu,\mathcal{D}} $. Hence we have:
\begin{equation} \nonumber
\Vert B^\infty (t) \Vert_{\mathcal{L}(Y_\alpha,Y_\alpha)} \leq  \mu^{-2} \lc f \rc^\alpha_{\sigma,\mu,\mathcal{D}}.
\end{equation}
We note that $U=I+B^\infty$, it immediately  follows that $\Vert U(\theta;t)\Vert_{\mathcal{L}(Y_\alpha,Y_\alpha)},\Vert ^tU(\theta;t) \Vert_{\mathcal{L}(Y_\alpha,Y_\alpha)} \leq 2$.

Let us prove now that $a^\infty(t)$ is well defined for $t  \in \left[   0\: 1 \right]   $ . We have 
\begin{align*} \nonumber
\Vert \overset{k-1}{\underset{j=1}{\prod}} B(t_{j}) a(t_k) \Vert_\alpha & \leq  \Vert  \overset{k-1}{\underset{j=1}{\prod}} B(t_{j}) \Vert_{\mathcal{L}(Y_\alpha,Y_\alpha)} \Vert a(t_k) \Vert_\alpha \\
& \leq  \left( \frac{1}{2} \right)^{k-1}  \frac{\lc f \rc^\alpha_{\sigma,\mu,\mathcal{D}}}{\mu^2} \frac{\lc f \rc^\alpha_{\sigma,\mu,\mathcal{D}}}{\mu},
\end{align*}
then $ a^\infty (t)$ is well defined for $t \in \left[  0\: 1 \right]  $, and bounded by the convergent series $ \underset{k \geq 1}{ \sum} \frac{1}{k!} \left( \frac{1}{2} \right)^{k-1} \mu^{-3} \left(  \lc f \rc_{\sigma,\mu}^\alpha \right) ^2 $. We proved that
\begin{equation} \nonumber
\Vert a^\infty (t) \Vert_\alpha \leq \frac{(\lc f \rc^\alpha_{\sigma,\mu,\mathcal{D}})^2}{\mu^3} \leq \lc f \rc^\alpha_{\sigma,\mu,\mathcal{D}}.
\end{equation}
Using the estimates made on  $a^\infty(t)$, $B^\infty(t)$ and the hypothesis \eqref{hyp-f}, we obtain that:
\begin{equation} \label{estim-zeta}
\Vert \zeta (t) - \zeta^0 \Vert_\alpha \leq  (1+ \mu^{-2} \Vert \zeta^0 \Vert_\alpha) \lc f \rc^\alpha_{\sigma,\mu,\mathcal{D}}.
\end{equation}
Let us prove now that $\Vert \zeta (t) - \zeta^0 \Vert_{\beta+}  \leq C (1+ \mu^{-2} \Vert \zeta^0 \Vert_\alpha) \lc f \rc^{\alpha,\beta+}_{\sigma,\mu,\mathcal{D}}$. We recall that $B(t)=Jf_{\zeta\zeta}(\theta(t))$, then $B$ belongs to $\mathcal{M}^+_\beta$. Moreover
\begin{equation} \nonumber
\vert B(t) \vert_{\beta_+} \leq \mu^{-2} \lc f \rc^{\alpha,\beta+}_{\sigma,\mu,\mathcal{D}}.
\end{equation}  
Since $\mathcal{M}_{\beta+}$ is closed under multiplication, then there exists a constant that depends on $\beta$ such that
\begin{equation} \nonumber
\vert B^\infty(t) \vert_{\beta+} \leq C \mu^{-2} \lc f \rc^{\alpha,\beta+}_{\sigma,\mu,\mathcal{D}}\leq 1.
\end{equation}
From here we notice that $\vert U(\theta;t) \vert_{\beta+} \leq 2$.

Using the estimate 4 from Lemma~\ref{norme}, we obtain that:
\begin{equation} \nonumber
\vert a^\infty (t) \vert_{\beta+} \leq C \frac{(\lc f \rc^{\alpha,\beta+}_{\sigma,\mu,\mathcal{D}})^2}{\mu^3} \leq \lc f \rc^{\alpha,\beta+}_{\sigma,\mu,\mathcal{D}}.
\end{equation}
By the two previous estimates made on $\vert a^\infty(t) \vert_{\beta+}$, $\vert B^\infty(t) \vert _{\beta+}$, the estimate 4 from Lemma~\ref{norme} and the hypothesis \eqref{hyp-f}, we obtain that:
\begin{equation} \nonumber
\Vert \zeta (t) - \zeta^0 \Vert_{\beta+} \leq C (1+ \mu^{-2} \Vert \zeta^0 \Vert_\alpha) \lc f \rc^{\alpha,\beta+}_{\sigma,\mu,\mathcal{D}},
\end{equation}
where $C$ is a constant that depends only on $\beta$.
\item Let us prove now the estimate made on $r$. We have:
\begin{align*}
\dot{r}(t) & =  -\nabla_\theta f(r(t), \theta(t) , \zeta(t)) \\
& =  - \nabla_\theta f (\theta(t)) - \nabla_\theta f_r (\theta(t)) r(t) - \langle \nabla_\theta f_\zeta (\theta(t)) , \zeta(t) \rangle \\
& -\frac{1}{2} \langle \nabla_\theta f_{\zeta\zeta} (\theta(t)) \zeta(t), \zeta(t) \rangle\\
& =  -\alpha(t) - \Lambda(t) r (t).
\end{align*}
where
\begin{equation} \nonumber
\alpha(t)= \nabla_\theta f (\theta(t)) + \langle \nabla_\theta f_\zeta (\theta(t)) , \zeta(t) \rangle + \frac{1}{2} \langle \nabla_\theta f_{\zeta\zeta} (\theta(t)) \zeta(t), \zeta(t) \rangle, 
\end{equation}
and
\begin{equation} \nonumber
\Lambda(t)=\nabla_\theta f_r (\theta(t)).
\end{equation}
By Cauchy estimate we have:
\begin{equation*}
\vert \Lambda (t) \vert_{\mathcal{L}(\mathbb{C}^n,\mathbb{C}^n)} \leq \eta^{-1} \mu^{-2} \lc f \rc^\alpha_{\sigma,\mu,\mathcal{D}} \leq \frac{1}{2}.
\end{equation*}
Similarly, by Cauchy estimate, we have:
\begin{align*}
\vert \alpha(t) \vert_{\mathbb{C}^n}  & \leq  \eta^{-1} \lc f \rc^\alpha_{\sigma,\mu,\mathcal{D}} + \eta^{-1} \mu^{-1} \Vert \zeta (t) \Vert_\alpha \lc f \rc^\alpha_{\sigma,\mu,\mathcal{D}} +\frac{1}{2}\eta^{-1} \mu^{-2} \Vert \zeta (t) \Vert_\alpha^2 \lc f \rc^\alpha_{\sigma,\mu,\mathcal{D}}
\\ & \leq \eta^{-1} ( 1 + \mu^{-1} \Vert \zeta (t) \Vert_\alpha + \mu^{-2} \Vert \zeta (t) \Vert_\alpha^2) \lc f \rc^\alpha_{\sigma,\mu,\mathcal{D}}.
\end{align*}
By \eqref{estim-zeta} and \eqref{hyp-f} we note that $\Vert \zeta(t) \Vert_\alpha \leq 2 \Vert \zeta^0 \Vert_\alpha$. Then we obtain that:
\begin{equation} \nonumber
\vert \alpha (t) \vert_{\mathbb{C}^n} \leq 2 \eta^{-1} ( 1 + \mu^{-1} \Vert \zeta^0  \Vert_\alpha + \mu^{-2} \Vert \zeta^0 \Vert_\alpha^2) \lc f \rc^\alpha_{\sigma,\mu,\mathcal{D}}.
\end{equation} 
The same reasoning made for $\zeta$ gives us that:
\begin{equation} \nonumber
r(t)=-\alpha^\infty(t)+( I- \Lambda^\infty(t))r^0,
\end{equation}
where
\begin{equation} \nonumber
\alpha^\infty (t) = \underset{k \geq 1}{ \sum} \int_0^t \int_0^{t_1} \ldots \int_0^{t_k-1} \overset{k-1}{\underset{j=1}{\prod}} \Lambda(t_{j})\alpha(t_k) dt_k \ldots dt_2 dt_1,
\end{equation}
and
\begin{equation} \nonumber
\Lambda^\infty (t) = \underset{k \geq 1}{ \sum} \int_0^t \int_0^{t_1} \ldots \int_0^{t_k-1} \overset{k}{\underset{j=1}{\prod}} \Lambda(t_{j})dt_k \ldots dt_2 dt_1.
\end{equation}
We have: 
\begin{equation} \nonumber
\vert \overset{k}{\underset{j=1}{\prod}} \Lambda(t_{j}) \vert_{\mathcal{L}(\mathbb{C}^n,\mathbb{C}^n)} \leq (\frac{1}{2} )^{k-1} \eta^{-1} \mu^{-2} \lc f \rc^{\alpha}_{\sigma,\mu,\mathcal{D}},
\end{equation}
then $\Lambda^\infty (t)$ is well defined for $t \in \left[  0\: 1 \right]  $ and bounded by the convergent series $ \underset{k \geq 1}{ \sum} \frac{\left( \frac{1}{2} \right)^{k-1}}{k!}  \mu^{-2} s^{-1} \lc f \rc^\alpha_{\sigma,\mu,\mathcal{D}} $. Then
\begin{equation} \nonumber
\vert \Lambda^\infty (t) \vert_{\mathcal{L}(\mathbb{C}^n,\mathbb{C}^n)} \leq   \eta^{-1}\mu^{-2} \lc f \rc^\alpha_{\sigma,\mu,\mathcal{D}}\leq\frac{1}{2}.
\end{equation}
For the first part of the estimate \eqref{estim comp flot ham}, we remark that  $S(\theta;t)= I- \Lambda^\infty(t)$. So we obtain that $\vert S(\theta;t) \vert_{\mathcal{L}(\mathbb{C}^n,\mathbb{C}^n)} \leq 2$.

Let us now prove that $\alpha^\infty(t)$ is well defined for $t \in \left[   0\: 1 \right]  $ . We have 
\begin{align*} \nonumber
\vert \overset{k}{\underset{j=1}{\prod}} \Lambda(t_{j}) \alpha(t_k) \vert_{\mathbb{C}^n}  & \leq  \vert  \overset{k}{\underset{j=1}{\prod}} \Lambda(t_{j}) \vert_{\mathcal{L}(\mathbb{C}^n,\mathbb{C}^n)} \vert \alpha(t_k) \vert_{\mathbb{C}^n} \\
& \leq C \left(\frac{1}{2}\right)^{k-1} \eta^{-1} \left( 1+\mu^{-1}   \Vert \zeta^0 \Vert_\alpha + \mu^{-2} \Vert \zeta^0 \Vert_\alpha^2 \right)\lc f \rc^\alpha_{\sigma,\mu,\mathcal{D}}\\
& \leq C \left(\frac{1}{2}\right)^{k-2} \eta^{-1} \left( 1+ \mu^{-2} \Vert \zeta^0 \Vert_\alpha^2 \right)\lc f \rc^\alpha_{\sigma,\mu,\mathcal{D}}.   
\end{align*}
Then  $ \alpha^\infty (t)$ is well defined for $t \in \left[   0\: 1 \right]  $ and bounded by the convergent series $$  \sum_{k \geq 1 }\frac{4}{2^kk!} \eta^{-1} \left( 1+ \mu^{-2} \Vert \zeta^0 \Vert^2 \right) \lc f \rc^\alpha_{\sigma,\mu,\mathcal{D}}.$$ This leads to
\begin{equation} \nonumber
\vert \alpha^\infty (t) \vert \leq C \eta^{-1} \left( 1+ \mu^{-2} \Vert \zeta^0 \Vert^2 \right)\lc f \rc_{\sigma,\mu}
\end{equation}
Using the two previous estimates made on $\alpha^\infty(t)$ and $\Lambda^\infty(t)$, we obtain that:
\begin{equation*}
\vert r (t) - r^0 \vert_{\mathbb{C}^n} \leq C \eta^{-1} \left( 1+ \mu^{-2} \Vert \zeta^0 \Vert^2 +\mu^{-2} \vert r^0 \vert_{\mathbb{C}^n} \right)\lc f \rc^\alpha_{\sigma,\mu,\mathcal{D}},
\end{equation*}
where $C$ is an absolute constant.

It remains to prove the estimates \eqref{estim comp L}. We remark that $\Lambda^\infty(t)$ does not depends on $\zeta^0$, then $L(\theta,\zeta;t)=-\alpha^\infty(t).$

Recall that $\zeta(t)=T(\theta;t)+U(\theta,t)\zeta^0$, then $\nabla_{\zeta^0}= ^tU(\theta;t) \nabla_{\zeta(t)}$. Since
\begin{equation} \nonumber
\nabla_{\zeta(t)}\alpha(t)=\nabla_\theta f_\zeta (\theta(t))+ \nabla_\theta f_{\zeta\zeta} (\theta(t)) \zeta(t),
\end{equation}
then using Cauchy estimate, the fact that $\Vert \zeta(t) \Vert_\alpha \leq 2 \Vert \zeta^0 \Vert_\alpha$ and the estimates \eqref{estim deriv f-moyenne}, we obtain that:
\begin{equation} \nonumber
\Vert \nabla_{\zeta^0} \alpha(t) \Vert_\alpha \leq 4 \eta^{-1}\mu^{-1}(1+\mu^{-1} \Vert \zeta^0 \Vert_\alpha )\lc f \rc^\alpha_{\sigma,\mu,\mathcal{D}}.
\end{equation} 
This leads to the first estimate of \ref{estim comp L}.
\\ Using estimate 4 of Lemma~\ref{norme} we obtain
\begin{equation} \nonumber
\vert \nabla_{\zeta^0} \alpha(t) \vert_{\beta+} \leq 4 \eta^{-1}\mu^{-1}(1+\mu^{-1} \Vert \zeta^0 \Vert_\alpha )\lc f \rc^{\alpha,\beta+}_{\sigma,\mu,\mathcal{D}},
\end{equation}
this leads to the second estimate of \eqref{estim comp L}.
Similarly, we have $\nabla^2_{\zeta(t)}\alpha(t)= \nabla_\theta f_{\zeta\zeta} (\theta(t))$ and $\nabla^2_{\zeta^0}=^tU(\theta;t)\nabla^2_{\zeta(t)}U(\theta,t)$, then
\begin{equation} \nonumber
\vert \nabla_{\zeta^0}^2 \alpha(t) \vert_{\beta+} \leq 4\eta^{-1} \mu^{-2} \lc f \rc^{\alpha,\beta+}_{\sigma,\mu,\mathcal{D}}.
\end{equation}
This leads to the third estimate of \eqref{estim comp L} and finish the proof.
\end{itemize}
\end{proof}
\begin{proposition} \label{composition}
For $j=0,1$ consider $\partial_\rho^j f$ $\in \mathcal{T}^{\alpha,\beta^+} (\mathcal{D}, \sigma, \mu)$  a jet function that satisfies $\lc \partial_\rho^j f \rc^{\alpha,\beta^+} _{\sigma,\mu,\mathcal{D}} \leq \frac{1}{2} \eta\nu^2$ for $0 < 2\eta< \sigma<1$ and $0<2 \nu<\mu<1$. Let $\partial_\rho^j h \in \mathcal{T}^{\alpha,\beta} (\mathcal{D}, \sigma, \mu) $, we denote for $0 \leq t \leq 1 $ 
$$ h_t(x,\rho):= h(\Phi_{f}^t(x,\rho));\rho).$$
Then $h_t \in  \mathcal{T}^{\alpha,\beta} (\mathcal{D}, \sigma - 2\eta, \mu - 2\nu)$ and 
\begin{equation} \label{estimation composition}
\lc \partial_\rho^j h_t \rc^{\alpha,\beta} _{\sigma-2\eta,\mu-2\nu,\mathcal{D}} \leq C   \frac{\mu}{\nu} \lc \partial_\rho^j h \rc^{\alpha,\beta} _{\sigma,\mu,\mathcal{D}},
\end{equation}
where $C$ is a constant that depends only on $\beta$.
\end{proposition}
\begin{proof}
The flow $\Phi_f^t$ is analytic on  $\mathcal{O}^\alpha(\sigma-2\eta, \mu-2\nu)$, then $h_t$ is analytic on  $\mathcal{O}^\alpha(\sigma-2\eta, \mu-2\nu)$.

1) Clearly we have $$\vert \partial_\rho^j h_t(x,.)\vert_{\mathcal{D}}=\underset{\rho \in \mathcal{D}}{sup} \left\vert \partial^j_\rho h(\Phi_f^t(x,\rho),\rho) \right\vert \leq \lc h \rc^{\alpha,\beta} _{\sigma,\mu,\mathcal{D}}. $$
It remains to estimate $\Vert \nabla_{\zeta^0} h_t \Vert_\alpha$, $\vert \nabla_{\zeta^0} h_t \vert_\beta$, $\vert \nabla^2_{\zeta_0\zeta^0} h_t \vert_\beta$ and their derivatives in $\rho$.

2) For $\nabla_{\zeta^0} h_t$, since $\theta$ does not depend on $\zeta_0$, we have
$$\nabla_{\zeta^0} h_t=\sum_{k=1}^n \frac{\partial h(x(t))}{\partial r_k(t)}\frac{\partial r_k(t)}{\partial \zeta^0} + \sum_{s \in \mathcal{L}} \frac{\partial h(x(t))}{\partial \zeta_s(t)}\frac{\partial \zeta_s(t)}{\partial \zeta^0}=\Sigma_1+\Sigma_2.$$
By Cauchy estimate $\left\vert \frac{\partial h(x(t))}{\partial r_k(t)} \right\vert_{\mathbb{C}^n} \leq \nu^{-2}\lc h \rc^{\alpha} _{\sigma,\mu,\mathcal{D}}.$ Using the first estimate from  \eqref{estim comp L}, we obtain that
$$\Vert \partial_{\zeta^0}r_k(t) \Vert_\alpha = \Vert \partial_{\zeta_0}L(\theta,\zeta;t) \Vert_\alpha \leq 8 \eta^{-1} \mu^{-1}\lc f \rc^{\alpha} _{\sigma,\mu,\mathcal{D}}.$$
Combining these estimates with hypothesis \eqref{hyp-f} yields to
$$\Vert \Sigma_1 \Vert_\alpha \leq 8 \eta^{-1}\mu^{-1}\nu^{-2}\lc h \rc^{\alpha,\beta} _{\sigma,\mu,\mathcal{D}}\lc f \rc^{\alpha} _{\sigma,\mu,\mathcal{D}}\leq 8 \eta^{-1}\lc h \rc^{\alpha,} _{\sigma,\mu,\mathcal{D}}.$$
Using this time the second estimate of \eqref{estim comp L}, we have
$$\vert \partial_{\zeta^0}r_k(t) \vert_\beta = \vert \partial_{\zeta_0}L(\theta,\zeta;t) \vert_\beta \leq 8 \eta^{-1} \mu^{-1}\lc f \rc^{\alpha,\beta} _{\sigma,\mu,\mathcal{D}}.$$
Combining this estimate with the estimate on $\frac{\partial h (x (t))}{\partial r_k (t)}$ and the assumption \eqref{hyp-f} we obtain that
$$\vert \Sigma_1 \vert_\beta \leq 4 \mu^{-1}\lc h \rc^{\alpha,} _{\sigma\mu,\mathcal{D}}.$$
For $\Sigma_2$ we have
$$\sum_{s \in \mathcal{L}} \frac{\partial h(x(t))}{\partial \zeta_s(t)}\frac{\partial \zeta_s(t)}{\partial \zeta^0}=\,^tU(t) \nabla_\zeta h , $$
where $U$ defined in \eqref{flot ham}. According to the third estimate in \eqref{estim comp flot ham} we have $\Vert \,^tU(t) \Vert _{\mathcal{L}(Y_\alpha,Y_\alpha)} \leq 2 $, so 
$$\Vert \Sigma_2 \Vert_\alpha \leq 2 \Vert \nabla_\zeta h \Vert_\alpha \leq 2 \mu^{-1} \lc h \rc^{\alpha} _{\sigma,\mu,\mathcal{D}}. $$
Using estimate 2 of the Lemma~\ref{norme} and estimate 4 of \eqref{estim comp flot ham}, we have
$$\vert \Sigma_2 \vert_\beta \leq C \vert  \,^tU(t) \vert_{\beta+} \Vert \nabla_\zeta h \Vert_\beta \leq  C \vert  U(t) \vert_{\beta+} \Vert \nabla_\zeta h \Vert_\beta \leq C \mu^{-1} \lc h \rc^{\alpha,\beta} _{\sigma,\mu,\mathcal{D}}, $$
where $C$ is a constant that depends on $\beta$. Then we obtain 
\begin{align*}
\Vert \nabla_{\zeta^0} h_t \Vert_\alpha &  \leq C \mu^{-1} \lc h \rc^{\alpha} _{\sigma,\mu,\mathcal{D}},\\
\vert \nabla_{\zeta^0} h_t \vert_\beta & \leq C \mu^{-1} \lc h \rc^{\alpha,\beta} _{\sigma,\mu,\mathcal{D}}.
\end{align*}
To obtain estimates on $\Vert \partial_\rho \nabla_{\zeta^0} h_t \Vert_\alpha$ and $\vert \partial_\rho \nabla_{\zeta^0} h_t \vert_\beta$, we have 
\begin{align*}
\partial_\rho \nabla_{\zeta^0} h_t= & \sum_{k=1}^n \frac{\partial_\rho \partial h(x(t))}{\partial r_k(t)}\frac{\partial r_k(t)}{\partial \zeta^0} + \frac{\partial h(x(t))}{\partial r_k(t)}\frac{\partial_\rho \partial r_k(t)}{\partial \zeta^0} \\
& + \sum_{s \in \mathcal{L}} \frac{\partial_\rho \partial h(x(t))}{\partial \zeta_s(t)}\frac{\partial \zeta_s(t)}{\partial \zeta^0} + \frac{\partial h(x(t))}{\partial \zeta_s(t)}\frac{\partial_\rho \partial \zeta_s(t)}{\partial \zeta^0}.
\end{align*}
Using the definition of the norm $\lc \partial_\rho^j h \rc^{\alpha,\beta} _{\sigma,\mu,\mathcal{D}}$  and the fact that $\partial_\rho x(t)$ satisfies the same estimates as $x(t)$, we obtain that
\begin{align*}
\Vert \partial_\rho \nabla_{\zeta^0} h_t \Vert_\alpha &  \leq C \mu^{-1} \lc h \rc^{\alpha} _{\sigma,\mu,\mathcal{D}},\\
\vert \partial_\rho \nabla_{\zeta^0} h_t \vert_\beta & \leq C \mu^{-1} \lc h \rc^{\alpha,\beta} _{\sigma,\mu,\mathcal{D}}.
\end{align*}

3) Now we will estimate $\nabla^2_{\zeta^0}h_t$ and $\partial_\rho \nabla^2_{\zeta^0}h_t$. From \eqref{flot ham} we note that $\theta$ does not depend on $\zeta^0$ and $\zeta(t)$ is affine in $\zeta^0$, then for $s,s' \in \mathcal{L}$ we have
\begin{align*}
\frac{\partial^2 h_t(x(t))}{\partial \zeta_s^0 \partial \zeta_{s'}^0}= & \frac{\partial^2 h(x(t))}{\partial \zeta(t) \partial \zeta(t)} \frac{\partial \zeta(t)}{\partial \zeta_s^0} \frac{\partial \zeta(t)}{\partial \zeta_{s'}^0} + \frac{\partial^2 h(x(t))}{\partial r(t)^2} \frac{\partial r(t)}{\partial \zeta_{s}^0} \frac{\partial r(t)}{\partial \zeta_{s}^0} \\
& + \frac{\partial^2 h(x(t))}{\partial r(t) \partial\zeta(t)} \frac{\partial r(t)}{\partial \zeta_{s}^0} \frac{\partial \zeta(t)}{\partial \zeta_{s'}^0} + \frac{\partial h(x(t))}{\partial r(t)} \frac{\partial r(t)^2}{\partial \zeta_{s}^0\partial \zeta_{s'}^0}\\
&= \Sigma_1+\Sigma_2+\Sigma_3+\Sigma_4.
\end{align*}

i)For $\Sigma_1$, by estimates 1 and 6 from Lemma~\ref{norme}, we have:
\begin{align*}
\vert \Sigma_1 \vert_\beta & \leq C \left\vert \frac{\partial^2 h(x(t))}{\partial \zeta(t) \partial \zeta(t)} \right\vert_\beta \left\vert \nabla_{\zeta_s^0 }\zeta \otimes \nabla_{\zeta_{s'}^0} \zeta \right\vert_{\beta+}, \\
& \leq C \left\vert \frac{\partial^2 h(x(t))}{\partial \zeta(t) \partial \zeta(t)} \right\vert_\beta \left\vert \nabla_{\zeta_s^0} \zeta \right\vert_{\beta+} \left\vert \nabla_{\zeta_{s'}^0} \zeta \right\vert_{\beta+},
\end{align*}
where $C$ is a constant that depends on $\beta$. 
According to Cauchy estimatew, we have: $\left\vert \frac{\partial^2 h(x(t))}{\partial \zeta(t) \partial \zeta(t)} \right\vert_\beta \leq \mu^{-2}\lc h \rc^{\alpha,\beta} _{\sigma,\mu,\mathcal{D}} $.
\\By estimate 4 from \eqref{estim comp flot ham} we have $\left\vert \nabla_{\zeta_s^0} \zeta \right\vert_{\beta+}=\vert U(\theta,t)\vert_{\beta+}\leq2$ .
Then
$$\vert \Sigma_1 \vert_\beta \leq C \mu^{-2}\lc h \rc^{\alpha,\beta} _{\sigma,\mu,\mathcal{D}} .$$
ii) For $\Sigma_2$, by estimate 5 from Lemma~\ref{norme}, we have:
$$\vert \Sigma_2 \vert_\beta \leq C \left\vert \frac{\partial^2 h(x(t))}{\partial r(t)^2} \right\vert_{\mathbb{C}^n} \left\vert \nabla_{\zeta_s^0} r \right\vert_\beta \left\vert \nabla_{\zeta_{s'}^0} r \right\vert_\beta.  $$
According to Cauchy estimate we have
\begin{equation} \nonumber
\left\vert \frac{\partial^2 h(x(t))}{\partial r(t)^2}\right\vert_{\mathbb{C}^n} \leq \nu^{-4}\lc h \rc^{\alpha,\beta} _{\sigma,\mu,\mathcal{D}}.
\end{equation}
By the second estimate from \eqref{estim comp L} we have 
\begin{equation} \nonumber
\left\vert \nabla_{\zeta_s^0} r \right\vert_\beta=\left\vert \nabla_{\zeta_s^0} L \right\vert_\beta \leq C \eta ^{-1} \mu^{-1}\lc f \rc^{\alpha,\beta+} _{\sigma,\mu,\mathcal{D}}.
\end{equation}
Thus, by combining these estimates with the assumption \eqref{hyp-f}, we obtain 
$$\vert \Sigma_2 \vert_\beta \leq C \mu^{-2} \lc h \rc^{\alpha,\beta} _{\sigma,\mu,\mathcal{D}}.$$
iii) For $\Sigma_3$ we have
$$\Sigma_3= \sum_{j=1}^n \frac{\partial\zeta}{\partial\zeta^0} \otimes \left( \frac{\partial^2 h}{\partial r_j \partial \zeta} \frac{\partial r_j }{\partial \zeta^0} \right).  $$
By estimate 4 from \eqref{estim comp flot ham} we have  $\vert \frac{\partial\zeta}{\partial\zeta^0} \vert_\beta \leq 2$. 
\\According to Cauchy estimate we have $\vert \frac{\partial^2 h}{\partial r_j \partial \zeta} \vert_\beta \leq C \nu^{-3} \lc h \rc^{\alpha,\beta} _{\sigma,\mu,\mathcal{D}} $.
\\Using the second estimate from \eqref{estim comp L} we have $\vert \frac{\partial r_j }{\partial \zeta^0} \vert_\beta \leq  \eta^{-1} \mu^{-1} \lc f \rc^{\alpha,\beta} _{\sigma,\mu,\mathcal{D}}. $
\\So by combining these results with the estimate 5 from Lemma~\ref{norme} and the hypothesis \eqref{hyp-f} we get:
$$\vert \Sigma_3 \vert_\beta \leq 8 \nu^{-1} \mu^{-1} \lc h \rc^{\alpha,\beta} _{\sigma,\mu,\mathcal{D}}, $$
where $C$ is a constant that depends only on $\beta$.
\\iv) Finally for $\Sigma_4$ we have:
$$ \vert  \Sigma_4 \vert_\beta \leq \left\vert \frac{\partial h(x(t))}{\partial r(t)} \right\vert_{\mathbb{C}^n} \left\vert \frac{\partial r(t)^2}{\partial \zeta_{s}^0\partial \zeta_{s'}^0} \right\vert_\beta .$$
By Cauchy estimate we have $\left\vert \frac{\partial h(x(t))}{\partial r(t)} \right\vert_{\mathbb{C}^n} \leq \nu^{-2} \lc h \rc^{\alpha} _{\sigma,\mu,\mathcal{D}} $. Using estimate from \eqref{estim comp L} we get 
\begin{equation} \nonumber
\left\vert \frac{\partial r(t)^2}{\partial \zeta_{s}^0\partial \zeta_{s'}^0} \right\vert_\beta = \vert \nabla^2_\zeta L(\theta, \zeta;t)\vert_{\beta} \leq C \eta^{-1} \mu^{-2} \lc f \rc^{\alpha,\beta+} _{\sigma,\mu,\mathcal{D}} .
\end{equation}
So by combining these estimates with the assumption \eqref{hyp-f}, we obtain
$$\vert \Sigma_4 \vert_\beta \leq C \mu^{-2}\lc h \rc^{\alpha} _{\sigma,\mu,\mathcal{D}}, $$
where $C$ is a constant that depends only on $\beta$.
\\The $\rho$-derivative of the hessian leads to similar estimates. 
\end{proof}
\section{Homological equation}
Let us first recall the general KAM strategy. Consider the following normal form 
\begin{equation} \nonumber
h_0(\rho)=  \omega_0(\rho) \cdot r + \frac{1}{2} \langle \zeta, A_0 (\rho) \zeta \rangle,
\end{equation}
that satisfies hypotheses A and B. We consider a small perturbation $f$ of $h_0$. If the jet of $f$ were zero, then  $\mathbb{T}^n \times \left\lbrace 0 \right\rbrace  \times \left\lbrace 0 \right\rbrace $ would be an invariant $n$-dimensional torus by the flow generated by the perturbed Hamiltonian $h_0+f$. Assume that the perturbation $f$ is small, let us say $f= \mathcal{O}( \varepsilon)$, then $f^T = \mathcal{O}(\varepsilon)$. In order to decrease the size of the jet of the perturbation term, we search for a symplectic change of variable $ \phi_S$, that transforms $h_0+f$ into a new normal formal close to the initial one and such that the jet of the  new perturbation term is much smaller than $f^T$. More precisely, we are searching for a Hamiltonian jet $S = S^T = \mathcal{O}(\varepsilon)$ such that its time one flow $\phi_S^1=\phi_S$ transforms the Hamiltonian $h_0+f$ into:
\begin{equation} \nonumber
(h_0+f)\circ \phi_S = h^++f^+, \quad \mbox{ where } \: (f^+)^T=\mathcal{O}(\varepsilon^2),
\end{equation}
and $ h^+ $ is the new Hamiltonian normal form close to $ h $ ( i.e. $ \vert h^+ - h \vert \sim \mathcal{O}( \varepsilon ) $). The Hamiltonian $h^+$ will be in the following form:
\begin{equation}\nonumber
 h^+=h_0+\hat{h}, \quad \hat{h}=C(\rho)+\chi(\rho)r+\frac{1}{2} \langle \zeta,\hat{K}(\rho)\zeta\rangle.
\end{equation}
Using Taylor expansion and the Hamiltonian structure , the Hamiltonian jet $S$ will solves the following nonlinear homological equation:
\begin{equation} \label{non-linear-hom-eq}
\lbrace h_0,S \rbrace + \lbrace f-f^T, S \rbrace^T + f^T = \hat{h} +\mathcal{O}( \varepsilon^2 ).
\end{equation}
We repeat the previous procedure with $h^+$ instead of $h_0$ and $f^+$ instead of $f$. Therefor the nonlinear homological equation will be solved for 
\begin{equation} \nonumber
h(\rho)=  \omega(\rho) \cdot r + \frac{1}{2} \langle \zeta, A(\rho) \zeta \rangle
\end{equation} 
for $\omega$ close to $\omega_0$ and $A$ close $A_0$.

In several proof of KAM theorems, the authors solve the following linear homological equation instead of \eqref{non-linear-hom-eq}:
\begin{equation} \label{eq homo}
\lbrace h_0,S \rbrace = \hat{h} - f^T + \mathcal{O}(\varepsilon^2).
\end{equation} 
In the KAM procedure it is very important to precisely control the jet of the new perturbation. If we use the linear equation, we would have to control $(f-f^T)\circ  \phi_S^1$ (where $f$ is the perturbation of the corresponding step). This term is difficult to control. Therefore, we solve equation \eqref{non-linear-hom-eq} and at each step we obtain a new perturbation whose jet is easier to control than the one we would obtain by solving equation \eqref{eq homo}. For more detail see Remark~\ref{utilite eq hom NL} and the elementary step.  However, we note that if we decompose $S=S_0+S_1+S_2$ where
\begin{equation*}
S_0(\theta):= S_\theta(\theta); \quad S_1(\theta,r):= S_r(\theta)r+\langle S_\zeta(\theta),\zeta \rangle;\quad S_2(\theta,\zeta):=\frac{1}{2} \langle S_{\zeta\zeta}(\theta)\zeta,\zeta \rangle,
\end{equation*}
and we do the same with $\hat{h}$ and we replace in the nonlinear homological equation, we obtain three equations of the form of the linear one. That's why we solve the linear homological equation and after we solve the nonlinear one.

Equation \eqref{eq homo} is linear because the solution $S$ is linearly dependent on the nonlinearity $f$. However, in equation \eqref{non-linear-hom-eq} the solution does not linearly depend on the nonlinearity $f$.

\subsection{Linear homological equation} 
Let $h$ be a Hamiltonian normal form 
\begin{equation} \nonumber
h(\rho)=  \omega(\rho) \cdot r + \frac{1}{2} \langle \zeta, A(\rho) \zeta \rangle.
\end{equation} 
In this part we solve the linear homological equation \eqref{eq homo}. The unknowns are $S$ and $\hat{h}$. It is sufficient to take $S$ as a jet-function:
\begin{equation} \nonumber
S(\theta,r,\zeta)= S_\theta(\theta)+ S_r(\theta)r+\langle S_\zeta(\theta),\zeta \rangle+ \frac{1}{2} \langle S_{\zeta\zeta}(\theta)\zeta,\zeta \rangle.
\end{equation}
After computing the Poisson bracket of $h$ and $S$, the equation \eqref{eq homo} becomes:
\begin{align*} 
\nabla_\theta S_\theta . \omega & + \nabla_\theta S_r . \omega. r  + \langle \nabla_\theta S_\zeta . \omega, \zeta \rangle + \frac{1}{2} \langle \nabla_\theta S_{\zeta\zeta}.\omega \zeta, \zeta \rangle + \langle A\zeta, JS_\zeta \rangle + \langle A \zeta , JS_{\zeta\zeta}\zeta \rangle \\ & = C - f_\theta + \chi .r  - f_r . r  - \langle f_\zeta , \zeta \rangle + \frac{1}{2} \langle \zeta ,\hat{K} \zeta \rangle - \frac{1}{2} \langle f_{\zeta \zeta} \zeta, \zeta \rangle + \mathcal{O}(\varepsilon^2)
\end{align*}
This gives four equation to solve. Recall that $f \in \mathcal{T}^{\alpha,\beta} (\mathcal{D}, \sigma, \mu) $, the solution $S$ will belongs to $\mathcal{T}^{\alpha,\beta+} (\mathcal{D}, \sigma, \mu) $.
\\ \textbf{The first two equations} The first two equations are
\begin{align} 
\nabla_\theta S_\theta  (\theta) . \omega & =  C - f_\theta (\theta) + \mathcal{O} (\varepsilon^2), \label{eq homo 1} \\
\nabla_\theta S_r(\theta) . \omega & =  \chi - f_r (\theta) + \mathcal{O} (\varepsilon^2) \label{eq homo 2}.
\end{align}
To solve these two equations, we impose that:
\begin{equation} \nonumber
C(\rho)= \int_{\mathbb{T}^n} f(\theta,0,0,\rho)d\theta,\quad \chi(\rho)= \int_{\mathbb{T}^n} f_r(\theta,0,0,\rho)d\theta.
\end{equation}
\textbf{The third equation} We have
\begin{equation} \nonumber
\langle \nabla_\theta S_\zeta . \omega , \zeta \rangle + \langle A \zeta , J S_\zeta (\theta)\rangle = - \langle f_\zeta , \zeta \rangle + \mathcal{O} (\varepsilon^2).
\end{equation}
The matrix $A$ is symmetric, so 
\begin{equation} \nonumber
\langle \nabla_\theta S_\zeta . \omega , \zeta \rangle + \langle A J S_\zeta (\theta), \zeta ,\rangle =  \langle -f_\zeta , \zeta \rangle + \mathcal{O} (\varepsilon^2).
\end{equation}
Matrices $A$ and $J$ commute since each $A$ block is of the form $A_{i}^j= \mathbb{C}I_2+\mathbb{C}\sigma_2$. Then the third equation becomes:
\begin{equation} \nonumber
\nabla_\theta S_\zeta (\theta).\omega + JAS_\zeta(\theta)=- f_\zeta(\theta) + \mathcal{O} (\varepsilon^2).
\end{equation} 
\textbf{The fourth equation} We have:
\begin{equation}\nonumber
\frac{1}{2} \langle \nabla_\theta S _{\zeta \zeta}(\theta).\omega \zeta, \zeta \rangle + \langle A \zeta , JS _{\zeta \zeta}(\theta) \zeta \rangle = \frac{1}{2} \langle \zeta , B \zeta \rangle - \frac{1}{2}  \langle f_{\zeta \zeta} ( \theta) \zeta, \zeta \rangle.
\end{equation}
Using the fact that $A$ is symmetric, $A$ and $J$ commute, and that $(AJS_{\zeta \zeta})^*=- S_{\zeta \zeta}JA$, we obtain:
\begin{equation} \nonumber
\langle \zeta , \nabla_\theta S_{\zeta \zeta}(\theta) . \omega \zeta \rangle + \langle \zeta , AJS_{\zeta \zeta}(\theta)\zeta \rangle - \langle \zeta , S_{\zeta \zeta}(\theta)AJ\zeta \rangle =  \langle \zeta , B \zeta \rangle -  \langle  \zeta, f_{\zeta \zeta} ( \theta) \zeta \rangle.
\end{equation}
Then we obtain the fourth equation
\begin{equation} \nonumber
\nabla_\theta S_{\zeta \zeta} ( \theta).\omega +AJS_{\zeta \zeta} ( \theta) -S_{\zeta \zeta} ( \theta) JA=-f_{\zeta \zeta} ( \theta) + B+\mathcal{O} (\varepsilon^2).
\end{equation}
\subsubsection{The first two equations.}
The first two equations \eqref{eq homo 1} and \eqref{eq homo 2} are of the form:
\begin{equation} \label{prototype}
\nabla_\theta \varphi (\theta).\omega = \psi ( \theta)
\end{equation}
with $\int_{\mathbb{T}^n}\psi(\theta)d\theta=0$. In the first equation $\varphi = S_\theta$ and  $\psi= C- f_\theta$. In the second equation $\varphi =S_r$ and  $\psi=\chi-f_r$. We start by expanding $\varphi$ and $\psi$ in Fourier series:
\begin{equation} \nonumber
\varphi (\theta) = \underset{k \in \mathbb{Z}^n \setminus \left\lbrace 0 \right\rbrace }{ \sum} \hat{\varphi}(k) e^{-ik\theta},\:\:\:\:\: \psi (\theta) = \underset{k \in \mathbb{Z}^n \setminus \left\lbrace 0 \right\rbrace }{ \sum} \hat{\psi}(k) e^{-ik\theta}
\end{equation} 
where $ \hat{\varphi}(k) = \int_{\mathbb{T}^n} \varphi (\theta) e^{-ik\theta}d\theta$ and $ \hat{\psi}(k) = \int_{\mathbb{T}^n} \psi (\theta) e^{-ik\theta}d\theta$.
We solve \eqref{prototype} by choosing
\begin{equation} \nonumber
\hat{\varphi}(k)= -i \frac{\hat{\psi}(k)}{k \cdot \omega}, \quad k\in \mathbb{Z}^n \setminus \left\lbrace 0 \right\rbrace ; \quad \hat{\varphi}(0)=0.
\end{equation}
To control the Fourier coefficients of $\varphi$, we need to truncate in $k$. For any $N \in \mathbb{N}^*$ we have: 
\begin{equation} \nonumber
-i \underset{0 < \vert k \vert \leq N}{ \sum} k \cdot \omega \hat{\varphi}(k)e^{-ik\theta}=\underset{k \in \mathbb{Z}^n}{ \sum}  \hat{\psi}(k)e^{-ik\theta} - \underset{\vert k \vert > N}{ \sum}  \hat{\psi}(k)e^{-ik\theta}
\end{equation} 
Using hypothesis A2, we have:
\begin{equation} \nonumber
\vert \omega(\rho) \cdot k \vert \geq \delta \geq \kappa, \:\:\:\:\forall \rho \in \mathcal{D}
\end{equation}
or there exists a unit vector $z_k \in \mathbb{R}^p$ such that
\begin{equation} \nonumber
\langle \partial_\rho ( k \cdot \omega(\rho)),z_k \rangle \geq \delta.
\end{equation}
The second case involves, according to Proposition~ \ref{melnikov}, that for $0<  \kappa < \delta$ and $N>1$ there exists a closed subset whose Lebesgue measure verifies:
\begin{equation} \nonumber
\operatorname{mes} ( \mathcal{D} \setminus \mathcal{D}_1) \leq C \kappa \delta^{-1} N^{2(n+1)},
\end{equation} 
such that for all $0 < \vert k \vert \leq N$ and $\rho \in \mathcal{D}_1 $ we have:
\begin{equation} \nonumber
\vert k \cdot \omega(\rho) \vert  \geq \kappa.
\end{equation}
Hence for $\rho \in \mathcal{D}_1$ and all $ 0 < \vert k \vert \leq N$ we have:
\begin{equation} \label{controle}
\vert \hat{\varphi}(k) \vert \leq \frac{\vert \hat{\psi}(k)}{\kappa},\:\:\:\:\: 0 < \vert k \vert \leq N.
\end{equation}
So we solved 
\begin{equation} \nonumber
\nabla_\theta \varphi (\theta) . \omega = \psi (\theta) + R(\theta),
\end{equation}
where $R(\theta) = - \underset{\vert k \vert > N}{ \sum}  \hat{\psi}(k)e^{-ik\theta}$.
\begin{lemma} \label{controle coef de fourier}
Let $f: \mathbb{R}^n \: \rightarrow \mathbb{C}$ a periodic holomorphic function on  $\vert Im \theta \vert < \sigma$ and continuous on $\vert Im \theta \vert \leq \sigma$. Then the Fourier coefficients of $f$ satisfy
\begin{equation} \nonumber
\vert \hat{f}(k) \vert \leq C e^{-\vert k \vert \sigma} \underset{\vert Im \theta \vert < \sigma}{ \sup} \vert f(\theta) \vert,
\end{equation}
where $C$ depends only on $n$.
\end{lemma}
\begin{proof} Recall that the Fourier coefficients of $f$ are given by:
\begin{equation*}
\hat{f}(k)  = \int_{\mathbb{T}^n} f (\theta) e^{-ik \theta} d\theta.
\end{equation*}
We can also consider the torus $\mathbb{T}^n-i(\sigma-\varepsilon)\frac{k}{\vert k \vert}$ instead of $ \mathbb{T}^n$ for all $0<\varepsilon<\sigma$. On this torus we have:
\begin{equation} \nonumber
\vert e^{ik\theta} \vert \leq e^{-\vert k \vert (\sigma-\varepsilon)}.
\end{equation}
Then we have
\begin{equation*}
\vert \hat{f} (k) \vert  \leq \int_{\mathbb{T}^n-i(\sigma-\varepsilon)\frac{k}{\vert k \vert}} \vert f(\theta) \vert e^{-\vert k \vert (\sigma-\varepsilon)} d\theta
\leq C  e^{-\vert k \vert (\sigma-\varepsilon)} \underset{\vert Im \theta \vert < \sigma}{ \sup} \vert f(\theta) \vert.
\end{equation*}
We obtain the desired inequality by continuity.
\end{proof}
\begin{lemma} \label{estim serie-integ}
Let $k\in \mathbb{Z}^n$, $a>0$ and $N\in \mathbb{N}^*$, then we have:
\begin{equation} \label{estim serie-integ 1}
\underset{\vert k \vert \leq N}{\sum} e^{- a \vert k \vert } \leq \frac{2^n}{a^n},
\end{equation}
\begin{equation} \label{estim serie-integ 3}
\underset{\vert k \vert \leq N}{\sum} \vert k \vert e^{- a \vert k \vert } \leq \frac{2^n}{a^{2n}},
\end{equation}
\begin{equation} \label{estim serie-integ 2}
\underset{\vert k \vert > N}{\sum} e^{- a \vert k \vert } \leq C(n) \frac{e^{-\frac{aN}{2}}}{a^n}.
\end{equation}
\end{lemma}
\begin{proof}

Lets us start with \eqref{estim serie-integ 1}. We have:
\begin{equation} \nonumber
\underset{\vert k \vert \leq N}{\sum} e^{- a \vert k \vert } \leq \left( \underset{\vert p \vert \leq N}{\sum} e^{- a \vert p \vert } \right)^n \leq \left( 2 \int_0^N e^{-xa}dx \right)^n \leq \frac{2^n}{a^n}. 
\end{equation}
Similarly we prove \eqref{estim serie-integ 3}. For the last estimation \eqref{estim serie-integ 2}, we have:
\begin{equation} \nonumber
\underset{\vert k \vert > N}{\sum} e^{- a \vert k \vert } \leq \int_{\vert x \vert > N}  e^{- a \vert x \vert }dx.
\end{equation}
Substituting $ \vert x \vert$ by $ y/a$ we obtain
\begin{align*}
\int_{\vert x \vert > N}  e^{- a \vert x \vert }dx & = \int_{y> aN} e^{-y} \left( \frac{y}{a} \right) ^{n-1} \frac{dy}{a}\\
& \leq a^{-n} e^{-\frac{aN}{2}} \int_{y>0} e^{-y/2}y^{n-1}dy \\
& \leq C(n) \frac{e^{-\frac{aN}{2}}}{a^n}.
\end{align*}

\end{proof}
\begin{lemma} \label{estim eq 1 et 2}
Let  $\kappa >0$, $N>1$, $0< \sigma' < \sigma \leq 1$, and $1\geq\mu>0$. We consider $\varphi$ and $\psi:\: \mathbb{R} \rightarrow \mathbb{C}$ two holomorphic functions on $\vert Im \theta \vert < \sigma$ that verify the equation \eqref{prototype}. Assume that $\rho \in \mathcal{D}_1 (\kappa, N)$, then we have:
\begin{align*}
\underset{\vert Im \theta \vert < \sigma'}{ \sup} \vert \varphi(\theta) \vert & \leq \frac{C}{\kappa(\sigma-\sigma')^{n}} \underset{\vert Im \theta \vert < \sigma}{ \sup} \vert \psi (\theta) \vert, \\
\underset{\vert Im \theta \vert < \sigma'}{ \sup} \vert R(\theta) \vert & \leq \frac{Ce^{-(\sigma-\sigma')N/2}}{(\sigma-\sigma')^{n}} \underset{\vert Im \theta \vert < \sigma}{ \sup} \vert \psi (\theta) \vert.
\end{align*} 
where $C$ depends only on $n$. 
\end{lemma}
\begin{proof} Using estimation \eqref{controle}, Lemma~\ref{controle coef de fourier} then estimation \eqref{estim serie-integ 1} from Lemma~\ref{estim serie-integ} we obtain:
\begin{align*}
\underset{\vert Im \theta \vert < \sigma'}{ \sup} \vert \varphi(\theta) \vert & \leq \underset{\vert k \vert \leq N}{ \sum} \frac{ \vert \hat{\psi}(k) \vert }{\kappa} e^{\vert k \vert \sigma'},\\
& \leq C  \underset{\vert k \vert \leq N}{ \sum} \frac{1}{\kappa} e^{-\vert k \vert \sigma} e^{\vert k \vert \sigma'} \underset{\vert Im \theta \vert < \sigma}{ \sup} \vert \psi(\theta) \vert, \\
& \leq \frac{C}{\kappa} \underset{\vert Im \theta \vert < \sigma}{ \sup} \vert \psi(\theta) \vert  \underset{\vert k \vert \leq N}{ \sum} e^{-\vert k \vert (\sigma- \sigma')}, \\
& \leq  \frac{C}{\kappa(\sigma- \sigma')^{n}} \underset{\vert Im \theta \vert < \sigma}{ \sup} \vert \psi(\theta) \vert.
\end{align*}
For the second estimation, by  Lemma~\ref{controle coef de fourier} and estimation \eqref{estim serie-integ 2} from Lemma~\ref{estim serie-integ}, we have:
\begin{align*}
\underset{\vert Im \theta \vert < \sigma'}{ \sup} \vert R(\theta) \vert & = \underset{\vert Im \theta \vert < \sigma'}{ \sup} \vert \underset{\vert k \vert > N}{ \sum} \hat{\psi}(k) e^{ik\theta} \vert\\
& \leq  \underset{\vert k \vert > N}{ \sum} \vert \hat{\psi}(k) \vert e^{\vert k \vert \sigma'},\\
& \leq C  \underset{\vert Im \theta \vert < \sigma}{ \sup} \vert \psi(\theta) \vert \underset{\vert k \vert >N}{ \sum} e^{-\vert k \vert (\sigma- \sigma')},\\
& \leq  \frac{Ce^{-(\sigma- \sigma')N/2}}{(\sigma- \sigma')^{n}} \underset{\vert Im \theta \vert < \sigma}{ \sup} \vert \psi(\theta) \vert.
\end{align*}
\end{proof}
\begin{lemma} \label{estim deriv eq 1 et 2}
Let $\kappa >0$, $N>1$, $0< \sigma' < \sigma \leq 1$, and $1\geq\mu>0$. Consider $\varphi$ and $\psi:\: \mathbb{R}^n \rightarrow \mathbb{C}$ two holomorphic functions on $\vert Im \theta \vert < \sigma$ that verify the equation \eqref{prototype}. Assume that $\rho \in \mathcal{D}_1 (\kappa, N)$, then we have:
\begin{align*}
\underset{\vert Im \theta \vert < \sigma'}{ \sup} \vert \partial_\rho \varphi(\theta) \vert & \leq \frac{C}{\kappa(\sigma-\sigma')^{n}} \underset{\vert Im \theta \vert < \sigma}{ \sup} \vert \partial_\rho \psi (\theta) \vert + \frac{C}{\kappa^2(\sigma-\sigma')^{2n}} \underset{\vert Im \theta \vert < \sigma}{ \sup} \vert \psi (\theta) \vert , \\
\underset{\vert Im \theta \vert < \sigma'}{ \sup} \vert \partial_\rho R(\theta) \vert & \leq \frac{Ce^{-(\sigma-\sigma')N/2}}{(\sigma-\sigma')^{n}} \underset{\vert Im \theta \vert < \sigma}{ \sup} \vert \partial_\rho \psi (\theta) \vert,
\end{align*}  
where $C$ depends on $n$ and $|\omega|_{\mathcal{C}^1 \left( \mathcal{D} \right) }$.
\end{lemma}
\begin{proof} Differentiating the equation \eqref{prototype} in $\rho$ gives
\begin{equation} \nonumber
\partial_\rho \hat{\varphi}(k)=i \frac{\partial_\rho \hat{\psi}(k)}{k \cdot \omega}+ i \frac{\hat{\psi}(k)}{(k \cdot \omega)^2} (k.\partial_\rho \omega(\rho)), \quad \mbox{for } \: 0<\vert k \vert \leq N,
\end{equation}
and
\begin{equation} \nonumber
\partial_\rho \hat{R}(k)= \partial_\rho \hat{\psi} (k), \quad \mbox{for } \: \vert k \vert > N.
\end{equation}
By applying the same arguments as in the proof of the Lemma~\ref{estim eq 1 et 2}, we obtain:
\begin{align*}
\underset{\vert Im \theta \vert < \sigma'}{ \sup} \vert \partial_\rho \varphi(\theta) \vert & \leq \frac{C}{\kappa(\sigma-\sigma')^{n}} \underset{\vert Im \theta \vert < \sigma}{ \sup} \vert \partial_\rho \psi (\theta) \vert + \frac{C}{\kappa^2(\sigma-\sigma')^{2n}} \underset{\vert Im \theta \vert < \sigma}{ \sup} \vert \psi (\theta) \vert , \\
\underset{\vert Im \theta \vert < \sigma'}{ \sup} \vert \partial_\rho R(\theta) \vert & \leq \frac{Ce^{-(\sigma-\sigma')N/2}}{(\sigma-\sigma')^{n}} \underset{\vert Im \theta \vert < \sigma}{ \sup} \vert \partial_\rho \psi (\theta) \vert.
\end{align*} 
\end{proof}
Now we apply the lemmas \ref{estim eq 1 et 2} and \ref{estim deriv eq 1 et 2} to the first two equations and we obtain the following proposition:
\begin{proposition}
Let $0<\kappa<\delta$, $N>1$, $0< \sigma' < \sigma$, $\mu>0$  and let $\omega : \mathcal{D} \to \mathbb{R}^n$ be $\mathcal{C}^1$ verifying $ \vert \omega - \omega_0 \vert_{\mathcal{C}^1(\mathcal{D})} \leq \delta_0$. Assume that $\partial^j_\rho f \in \mathcal{T}^{\alpha,\beta}(\sigma,\mu,\mathcal{D})$ for $j=0,1$, then there exists a closed subset whose Lebesgue measure satisfies:
\begin{equation} \nonumber
\operatorname{mes} ( \mathcal{D} \setminus \mathcal{D}_1) \leq C_0 \kappa \delta^{-1} N^{2(n+1)},
\end{equation} 
such that for $0 < \vert k \vert \leq N$ and $\rho \in \mathcal{D}_1 $ we have: 
\begin{enumerate}
\item There exist two analytic functions $S_\theta (.;\rho)$ and $R_\theta (.;\rho)$ on $\mathbb{T}^n_{\sigma'}$ such that:
\begin{equation} \nonumber
\nabla_\theta S_\theta(\theta,\rho).\omega= -f(\theta,0,\rho)+ \int_{\mathbb{T}^n} f(\theta,0,0,\rho)d\theta + R_\theta (\theta, \rho),
\end{equation}
where
\begin{align*}
\underset{\vert Im \theta \vert < \sigma'}{ \sup} \vert S_\theta(\theta,\rho) \vert & \leq \frac{C}{\kappa(\sigma-\sigma')^{n}}  \lc f^T \rc^{\alpha,\beta} _{\sigma,\mu,\mathcal{D}}, \\
\underset{\vert Im \theta \vert < \sigma'}{ \sup} \vert \partial_\rho S_\theta(\theta,\rho) \vert & \leq \frac{C}{\kappa(\sigma-\sigma')^{n}} \lc \partial_\rho f^T \rc^{\alpha,\beta} _{\sigma,\mu,\mathcal{D}} + \frac{C}{\kappa^2(\sigma-\sigma')^{2n}} \lc f^T \rc^{\alpha,\beta} _{\sigma,\mu,\mathcal{D}} , \\
\underset{\vert Im \theta \vert < \sigma'}{ \sup} \vert \partial_\rho^j R_\theta(\theta,\rho) \vert & \leq \frac{Ce^{-(\sigma-\sigma')N/2}}{(\sigma-\sigma')^{n}} \lc f^T \rc^{\alpha,\beta} _{\sigma,\mu,\mathcal{D}}\mbox{ avec } j=0,1.
\end{align*}
\item There exist two analytic functions  $S_r (.;\rho)$ and $R_r (.;\rho)$ on $\mathbb{T}^n_{\sigma'}$ such that:
\begin{equation} \nonumber
\nabla_\theta S_r(\theta,\rho).\omega= -\nabla_r f(\theta,0,\rho)+ \int_{\mathbb{T}^n} \nabla_r f(\theta,0,0,\rho)d\theta + R_r (\theta, \rho),
\end{equation}
where
\begin{align*}
\underset{\vert Im \theta \vert < \sigma'}{ \sup} \vert S_r(\theta,\rho) \vert & \leq \frac{C}{\kappa \mu^2 (\sigma-\sigma')^{n}}  \lc f^T \rc^{\alpha,\beta} _{\sigma,\mu,\mathcal{D}},\\
\underset{\vert Im \theta \vert < \sigma'}{ \sup} \vert \partial_\rho S_r(\theta,\rho) \vert & \leq \frac{C}{\kappa\mu^2(\sigma-\sigma')^{n}} \lc \partial_\rho f^T \rc _{\sigma,\mu,\mathcal{D}}^{\alpha,\beta} + \frac{C}{\kappa^2 \mu^2(\sigma-\sigma')^{2n}} \lc f^T \rc^{\alpha,\beta} _{\sigma,\mu,\mathcal{D}} , \\
\underset{\vert Im \theta \vert < \sigma'}{ \sup} \vert \partial_\rho^j R_r(\theta,\rho) \vert & \leq \frac{Ce^{-(\sigma-\sigma')N/2}}{(\sigma-\sigma')^{n}} \lc f^T \rc^{\alpha,\beta} _{\sigma,\mu,\mathcal{D}} \mbox{ avec } j=0,1.
\end{align*}
\end{enumerate} 
The constant $C_0$ depends on $ \vert \omega_0 \vert_{\mathcal{C}^1\left( \mathcal{D}\right) }$ while the constant $C$ depends on $n$ in addition.
\end{proposition}
\subsubsection{The third equation.}
We begin by stating the following result proved in the appendix of \cite{eliasson2010kam}.
\begin{lemma} \label{controle inverse}
Let $A(t)$ a real square diagonal $N$-matrix  with diagonal components $a_j(t)$ which are $\mathcal{C}^1$ on $I=]-1,1[$ satisfying for all $1 \leq j \leq N$ and all $t\in I$
\begin{equation} \nonumber
a_j'(t) \geq \delta. 
\end{equation}
Let $B(t)$ be a hermitiab square $N$-matrix of class $\mathcal{C}^1$ on $I$ such that 
\begin{equation} \nonumber
\Vert B'(t) \Vert \leq \frac{\delta}{2},
\end{equation}
for all $t \in I$. Then 
\begin{equation} \nonumber
\operatorname{mes} \lbrace t \in I \mid \Vert (A(t)+B(t))^{-1} \Vert > \kappa^{-1}  \rbrace \leq 4 N \kappa \delta_0^{-1}.
\end{equation}
\end{lemma}
\begin{proposition}
Let $0<\kappa \leq \frac{\delta}{2} \leq \frac{c_0}{4}$ , $N>1$, $0<\sigma'<\sigma$, $\mu>0$ and let $\omega : \mathcal{D} \to \mathbb{R}^n$ be $\mathcal{C}^1$ and verifying $\vert \omega - \omega_0 \vert_{\mathcal{C}^1(\mathcal{D})} \leq \delta_0$. Let $A : \mathcal{D} \to \mathcal{NF} \cap \mathcal{M}_\beta$ be $\mathcal{C}^1$ and satisfying $
\vert A - A_0  \vert_{\beta, \mathcal{C}^1(\mathcal{D})} \leq \delta_0$.
Then there exists a closed subset $\mathcal{D}_2 \subset \mathcal{D}$ whose Lebesgue measure satisfy
\begin{equation} \nonumber
\operatorname{mes}(\mathcal{D} \setminus \mathcal{D}_2) \leq \tilde{C}d \kappa \delta^{-1}N^{n+2},
\end{equation}
such that for all $\rho \in \mathcal{D}_2$ there exist two real analytic functions $S_\zeta (.;\rho)$ and $R_\zeta (.;\rho)$ on $\mathbb{T}^n_{\sigma'}$ satisfying:
\begin{equation} \nonumber
\nabla_\theta S_\zeta (\theta).\omega + JAS_\zeta(\theta)=- f_\zeta(\theta) + R_\zeta (\theta, \rho),
\end{equation}
with
\begin{align*}
\underset{\vert Im \theta \vert < \sigma'}{ \sup} \Vert S_\zeta(\theta) \Vert_{\alpha+1} + \underset{\vert Im \theta \vert < \sigma'}{ \sup} \vert S_\zeta(\theta) \vert_{\beta+} & \leq \frac{C}{\kappa \mu(\sigma-\sigma')^{n}}\lc f^T \rc^{\alpha,\beta}_{\sigma,\mu,\mathcal{D}} , \\
\underset{\vert Im \theta \vert < \sigma'}{ \sup} \Vert R_\xi(\theta) \Vert_\alpha + \underset{\vert Im \theta \vert < \sigma'}{ \sup} \vert R_\xi(\theta) \vert_\beta & \leq \frac{Ce^{-(\sigma-\sigma')N/2}}{\mu(\sigma-\sigma')^{n}} \lc f^T \rc^{\alpha,\beta}_{\sigma,\mu,\mathcal{D}},\\
\underset{\vert Im \theta \vert < \sigma'}{ \sup} \Vert \partial_\rho S_\zeta(\theta,\rho) \Vert_{\alpha+1} + \underset{\vert Im \theta \vert < \sigma'}{ \sup} \vert \partial_\rho S_\zeta(\theta,\rho) \vert_{\beta+} & \leq \frac{C}{\kappa^2 \mu (\sigma-\sigma')^{2n}} \lc f^T \rc^{\alpha,\beta} _{\sigma,\mu,\mathcal{D}}\\ & +\frac{C}{\kappa \mu (\sigma-\sigma')^{n}} \lc \partial_\rho f^T \rc^{\alpha,\beta} _{\sigma,\mu,\mathcal{D}},\\
\underset{\vert Im \theta \vert < \sigma'}{ \sup} \Vert \partial_\rho R_\xi(\theta) \Vert_\alpha + \underset{\vert Im \theta \vert < \sigma'}{ \sup} \vert \partial_\rho R_\xi(\theta) \vert_\beta & \leq \frac{Ce^{-(\sigma-\sigma')N/2}}{\mu(\sigma-\sigma')^{n}} \lc f^T \rc^{\alpha,\beta} _{\sigma,\mu,\mathcal{D}}.
\end{align*}
The constant $C$ depends on $n$, $ |\omega_0|_{\mathcal{C}^1 \left( \mathcal{D} \right) }$ and $\vert  A_0 \vert_{\beta,{\mathcal{C}^1 \left( \mathcal{D} \right) }}$ while $\tilde{C}$ depends on $ |\omega_0|_{\mathcal{C}^1 \left( \mathcal{D} \right) }$ and $c_0$.
\end{proposition}
\begin{proof} Consider the following complex variables:
\begin{equation} \nonumber
z_{s}= \begin{pmatrix}	  
   		\xi_{s} \\
  		 \eta_{s} 
  	\end{pmatrix}
 = \frac{1}{\sqrt{2}} \begin{pmatrix}	  
   		1 & i \\
  		1 & -i 
  	\end{pmatrix}
 	\zeta_s
 = \begin{pmatrix}	  
   		\frac{1}{\sqrt{2}}(p_{s}+iq_{s}) \\
  		\frac{1}{\sqrt{2}}(p_{s}-iq_{s}) 
  	\end{pmatrix}; \quad s \in \mathcal{L}
\end{equation}
In these new variables the Hamiltonian normal form becomes
\begin{equation} \nonumber
h(\rho) = \omega(\rho)\cdot r + \langle \xi,Q(\rho)\eta \rangle,
\end{equation}
where $Q$ is a Hermitian matrix. The jet of a function $g$ is given by:
\begin{align*}
g^T= &  g_\theta(\theta)+g_r(\theta).r + \langle g_\xi(\theta),\xi \rangle + \langle g_\eta(\theta) ,\eta \rangle + \frac{1}{2} \langle g_{\xi\xi} (\theta)\xi,\xi \rangle  + \frac{1}{2} \langle g_{\eta\eta} (\theta)\eta,\eta \rangle+ \langle g_{\xi\eta}(\theta) \xi,\eta \rangle.
\end{align*}
The Poisson bracket becomes
\begin{equation} \nonumber
\left\lbrace h,S \right\rbrace = \nabla_r h \nabla_\theta S - \nabla_\theta h \nabla_r S -i ( \langle \nabla_\xi h , \nabla_\eta S \rangle - \langle \nabla_\eta h , \nabla_\xi S \rangle)
\end{equation}
In the complex variables, the jet-function $S$ is given by 
\begin{align*} \nonumber
S(\theta) & = S_\theta(\theta)+S_r(\theta).r + \langle S_\xi(\theta),\xi \rangle + \langle S_\eta (\theta),\eta \rangle + \frac{1}{2} \langle S_{\xi\xi}(\theta) \xi,\xi \rangle \\
& + \frac{1}{2} \langle S_{\eta\eta}(\theta) \eta,\eta \rangle+ \langle S_{\xi\eta}(\theta) \xi,\eta \rangle.
\end{align*}
So we can decouple the third homological equation into two equations:
\begin{align}
\begin{array}{c} \label{troisieme eq dec}
\nabla_\theta S_\xi (\theta) . \omega +i Q S_\xi (\theta) = -f_\xi (\theta) + R_ \xi( \theta),\\
\nabla_\theta S_\eta (\theta) . \omega - i {}^tQ S_\eta (\theta) = - f_\eta (\theta)+R_\eta(\theta),
\end{array}
\end{align}
where $Q$ is Hermitian and block diagonal matrix of the following form:
\begin{equation*}
Q=diag \left\lbrace \lambda_{s},\:\: s \in \mathcal{L} \right\rbrace + H, 
\end{equation*}
where $H$ is Hermitian and block diagonal matrix.
Expending in Fourier series, equations \eqref{troisieme eq dec} becomes:
\begin{align}
\begin{array}{cc} \label{troisieme eq fourier}
i ( k \cdot \omega I + Q )\hat{S}_\xi (k) = - \hat{f}_\xi (k) + \hat{R}_ \xi(k), & k \in \mathbb{Z}^n,\\
i ( k \cdot \omega I - {}^t Q )\hat{S}_\eta (k) = - \hat{f}_\eta (k) + \hat{R}_ \eta(k),  & k \in \mathbb{Z}^n.
\end{array}
\end{align}
To solve them we need to control the operators $( k \cdot \omega I + Q )$ and $( k \cdot \omega I - {}^t Q )$. The two equations are similar. We will consider the first equation .

Consider $Q_{[s]}$ the restriction of $Q$ to the block $s \times s$. We start by decomposing the equation over the block $[s]$
\begin{equation} \label{3eme eq hom bloc}
i ( k \cdot \omega I_{[s]} + Q_{[s]} )\hat{S}_{[s]} (k) = - \hat{f}_{[s]} (k) + \hat{R}_{[s]}(k),  \quad k \in \mathbb{Z}^n,
\end{equation}
where $\hat{S}_{[s]} (k)$, $\hat{f}_{[s]} (k)$ and $\hat{R}_{[s]}(k)$ are respectively the restriction of  $\hat{S}_{\xi} (k)$, $\hat{f}_{\xi} (k)$ and $\hat{R}_{\xi}(k)$ over the block $[s]$.

$Q_{[s]}$ is a Hermitian matrix. Let $\alpha_{\ell}$ denotes an eigenvalue of $Q_{[s]}$ and let us recall that $Q_{[s]}=D_{[s]}+H_{[s]}$ where $D_{[s]}=diag \left\lbrace \lambda_{\ell},\; \ell \in [s] \right\rbrace$ and $H_{[s]}$ is a matrix of dimension at most $d$. Thus $\alpha_{\ell}-\lambda_{\ell}$ is an eigenvalue of $H_{[s]}$.

To control the operator $ k \cdot \omega I_{[s]} + Q_{[s]}$  we need to control $\vert k \cdot \omega + \alpha_{\ell} \vert$ for $0 \leq \vert k \vert \leq N$. We remark that if $\vert \partial_\rho^j N(\rho) \vert_\beta < \frac{\delta}{2}$, then for any $ \ell \in [s]$
\begin{equation} \nonumber
\vert \alpha_{\ell}-\lambda_{\ell} \vert \leq \Vert H_{[s]} \Vert \leq \Vert H_{[s]} \Vert_{\infty} \leq \frac{1}{\langle s \rangle^{2\beta}} \vert H(\rho) \vert_\beta \leq  \frac{\delta}{2\langle s\rangle ^{2\beta}} \leq \frac{\delta}{2} \leq \frac{c_0}{4}.
\end{equation}
So we obtain that
\begin{equation} \nonumber
\vert \alpha_\ell \vert \geq \lambda_\ell  - \vert \alpha_{\ell}-\lambda_{\ell} \vert \geq \frac{3}{4}c_0\langle s\rangle .
\end{equation}
For $k=0$, equation \eqref{3eme eq hom bloc} is solved. For $k \neq 0$, using hypothesis A2, we have either
\begin{equation} \nonumber
\vert k \cdot \omega + \alpha_{\ell} \vert \geq \vert k \cdot \omega + \lambda_{\ell} \vert - \vert \alpha_{\ell} - \lambda_{\ell} \vert \geq \delta  \langle s\rangle - \frac{\delta}{2} \geq \frac{\delta}{2}  \langle s \rangle \geq \kappa  \langle s \rangle,
\end{equation}
or there exists a unit vector $z_k \in \mathbb{R}^p$ such that
\begin{equation} \nonumber
\langle \partial_\rho (k \cdot \omega+\lambda_{\ell}), z_k \rangle \geq \delta.
\end{equation}
Let us consider the second case. We note that for this unit vector $z_k$ we have:
\begin{equation} \nonumber
\Vert (\partial_\rho \cdot z_k) H_{[s]} \Vert \leq \Vert (\partial_\rho \cdot z_k) H_{[s]} \Vert_{\infty} \leq \Vert \partial_\rho H_{[s]} \Vert_{\infty} \leq \frac{\delta}{2}.
\end{equation}
Consider
\begin{equation} \nonumber
J(k,s) = \lbrace \rho \in \mathcal{D} \mid \Vert (k \cdot \omega I_{[s]} + Q_{[s]})^{-1} \Vert > (\kappa \langle s \rangle)^{-1}   \rbrace
\end{equation}
Applying the Lemma~\ref{controle inverse} to $ \operatorname{diag} \lbrace k \cdot \omega + \lambda_\ell , \; \ell \in [s]  \rbrace$ and the Hermitian matrix $Q_{[s]}$ we obtain that 
\begin{equation} \nonumber 
\operatorname{mes} J(k,s) \leq d \delta^{-1} \kappa \langle s \rangle \leq 4 C c_0^{-1} d \kappa \delta^{-1} N ,
\end{equation}
for $ \vert s \vert \leq  4 C c_0^{-1}N$ and $ C:= \vert \omega \vert_{\mathcal{C}^1 \left( \mathcal{D} \right)}+ \delta_0 \leq \vert \omega \vert_{\mathcal{C}^1 \left( \mathcal{D} \right)} +1 $. Consider the set  $\mathcal{B} = \lbrace (k,s) \in \mathbb{Z}^n \times \mathcal{L} \mid \vert k \vert \leq N ,\; \vert s \vert \leq 4 C c_0^{-1} \kappa \delta^{-1}N  \rbrace$. This set contains at most $16C c_0^{-1}N^{n+1}$ points. Then
\begin{equation} \nonumber
\operatorname{mes} \bigcup_{(k,s) \in  \mathcal{B}} J(k,s) \leq 64C^2 c_0^{-2}d \kappa \delta^{-1}N^{n+2} = \tilde{C}d \kappa \delta^{-1}N^{n+2}. 
\end{equation}
For $ \vert s \vert > 4 C c_0^{-1}N$, we have
\begin{align*}
\vert k \cdot \omega + \alpha_\ell \vert & \geq  \lambda_\ell - \vert \alpha_\ell - \lambda_\ell \vert - \vert k \cdot \omega \vert 
 \geq c_0 \langle s \rangle - \frac{1}{4} c_0 \langle s \rangle - \frac{1}{4} c_0 \langle s \rangle 
 \geq \kappa \langle s \rangle.
\end{align*}
We define $ \displaystyle \mathcal{D}_2 := \mathcal{D} \setminus \bigcup_{(k,s) \in  \mathcal{B}} J(k,s)$. The third equation is solved by posing for all $\rho \in \mathcal{D}_2$  :
\begin{align*}
\begin{array}{cc} 
(\hat{S}_\xi)_{[s]} (k) = i \left( (k \cdot \omega I_{[s]}+Q_{[s]})^{-1}( \hat{f}_\xi)_{[s] }(k)\right)_ , & 0 \leq \vert k \vert \leq N,\\
\end{array}
\end{align*}
and
\begin{equation} \nonumber
(R_\xi)_{[s] \ell}=- \underset{ \vert k \vert > N }{ \sum}( \hat{f}_\xi)_{[s] \ell}(k) e^{ik\theta}.
\end{equation}
We denote by $(\hat{S}_\xi)_{[s] \ell}$ the $\ell$-th component of $(\hat{S}_\xi)_{[s]}$ and similarly we define $(\hat{f}_\xi)_{[s] \ell}$ and $(\hat{R}_\xi)_{[s] \ell}$. Using the same argument as in the first equation, we obtain for any $0 < \sigma < \sigma'$ and all $\rho \in \mathcal{D}_2$  that
\begin{align*}
\underset{\vert Im \theta \vert < \sigma'}{ \sup} \vert (S_\xi)_{ [s]\ell}(\theta) \vert & \leq \frac{C}{\kappa \langle s \rangle (\sigma-\sigma')^{n}} \underset{\vert Im \theta \vert < \sigma}{ \sup} \vert (f_\xi)_{[s]\ell} (\theta) \vert, \\
\underset{\vert Im \theta \vert < \sigma'}{ \sup} \vert (R_\xi)_{[s]\ell}(\theta) \vert & \leq \frac{Ce^{-(\sigma-\sigma')N/2}}{(\sigma-\sigma')^{n}} \underset{\vert Im \theta \vert < \sigma}{ \sup} \vert (f_\xi)_{[s]\ell} (\theta) \vert.
\end{align*}
The estimates imply that
\begin{align*}
\underset{\vert Im \theta \vert < \sigma'}{ \sup} \Vert S_\xi(\theta) \Vert_{\alpha+1} + \underset{\vert Im \theta \vert < \sigma'}{ \sup} \vert S_\xi(\theta) \vert_{\beta+} & \leq \frac{C}{\kappa \mu(\sigma-\sigma')^{n}}\lc f^T \rc^{\alpha,\beta}_{\sigma,\mu,\mathcal{D}} , \\
\underset{\vert Im \theta \vert < \sigma'}{ \sup} \Vert R_\xi(\theta) \Vert_\alpha + \underset{\vert Im \theta \vert < \sigma'}{ \sup} \vert R_\xi(\theta) \vert_\beta & \leq \frac{Ce^{-(\sigma-\sigma')N/2}}{\mu(\sigma-\sigma')^{n}} \lc f^T \rc_{\sigma,\mu,\mathcal{D}},
\end{align*}
where $C$ depends on $n$, $ |\omega_0|_{\mathcal{C}^1 \left( \mathcal{D} \right) }$ and $\vert  A_0 \vert_{\beta,{\mathcal{C}^1 \left( \mathcal{D} \right) }}$.\\
To obtain the estimates of the derivatives of $S_\xi$ and $R_\xi$ with respect to $\rho$, we differentiate the equation \eqref{troisieme eq fourier}. So we obtain
\begin{equation} \nonumber
i ( k \cdot \omega I + Q ) \partial_\rho \hat{S}_\xi (k) = -i ( \partial_\rho k \cdot\omega I + \partial_\rho Q) \hat{S}_\xi (k) - \partial \hat{f}_\xi (k) , \quad   0 \leq \vert k \vert \leq N,
\end{equation}
and
\begin{equation} \nonumber
\partial_\rho R_\xi ( \theta) = - \underset{ \vert k \vert \geq N }{ \sum} \partial_\rho ( \hat{f}_\xi)(k) e^{ik\theta}.
\end{equation}
The same process described above gives us
\begin{align*} 
\underset{\vert Im \theta \vert < \sigma'}{ \sup} \Vert \partial_\rho S_\xi(\theta,\rho) \Vert_{\alpha+1} & \leq \frac{C}{\kappa^2(\sigma-\sigma')^{2n}}\underset{\vert Im \theta \vert < \sigma}{ \sup} \Vert f_\xi(\theta,\rho) \Vert_\alpha + \frac{C}{\kappa(\sigma-\sigma')^{n}} \underset{\vert Im \theta \vert < \sigma'}{ \sup} \Vert \partial_\rho f_\xi(\theta,\rho) \Vert_\alpha,
\end{align*}
and
\begin{align*} \nonumber
\underset{\vert Im \theta \vert < \sigma'}{ \sup} \vert \partial_\rho S_\xi(\theta,\rho) \vert_{\beta+} & \leq \frac{C}{\kappa^2(\sigma-\sigma')^{2n}}\underset{\vert Im \theta \vert < \sigma}{ \sup} \vert f_\xi(\theta,\rho) \vert_\beta + \frac{C}{\kappa(\sigma-\sigma')^{n}} \underset{\vert Im \theta \vert < \sigma'}{ \sup} \vert \partial_\rho f_\xi(\theta,\rho) \Vert_\beta,
\end{align*}
This leads to
\begin{align*} \nonumber
\underset{\vert Im \theta \vert < \sigma'}{ \sup} \Vert \partial_\rho S_\xi(\theta,\rho) \Vert_{\alpha+1} + \underset{\vert Im \theta \vert < \sigma'}{ \sup} \vert \partial_\rho S_\xi(\theta,\rho) \vert_{\beta+} & \leq \frac{C}{\kappa^2 \mu (\sigma-\sigma')^{2n}} \lc f^T \rc^{\alpha,\beta} _{\sigma,\mu,\mathcal{D}}\\
& + \frac{C}{\kappa \mu (\sigma-\sigma')^{n}} \lc \partial_\rho f^T \rc^{\alpha,\beta} _{\sigma,\mu,\mathcal{D}} .
\end{align*}
Similarly, for the $\rho$-derivative of $R_\xi(\theta)$ we obtain:
\begin{equation} \nonumber
\underset{\vert Im \theta \vert < \sigma'}{ \sup} \Vert \partial_\rho R_\xi(\theta) \Vert_\alpha  \leq \frac{Ce^{-(\sigma-\sigma')N/2}}{(\sigma-\sigma')^{n}} \underset{\vert Im \theta \vert < \sigma}{ \sup} \Vert \partial_\rho f_\xi (\theta) \Vert_\alpha,
\end{equation}
and
\begin{equation} \nonumber
\underset{\vert Im \theta \vert < \sigma'}{ \sup} \vert \partial_\rho R_\xi(\theta) \vert_\beta  \leq \frac{Ce^{-(\sigma-\sigma')N/2}}{(\sigma-\sigma')^{n}} \underset{\vert Im \theta \vert < \sigma}{ \sup} \vert \partial_\rho f_\xi (\theta) \vert_\beta.
\end{equation}
Therefore we get
\begin{equation} \nonumber
\underset{\vert Im \theta \vert < \sigma'}{ \sup} \Vert \partial_\rho R_\xi(\theta) \Vert_\alpha + \underset{\vert Im \theta \vert < \sigma'}{ \sup} \vert \partial_\rho R_\xi(\theta) \vert_\beta  \leq \frac{Ce^{-(\sigma-\sigma')N/2}}{\mu(\sigma-\sigma')^{n}} \lc \partial_\rho f^T \rc^{\alpha,\beta} _{\sigma,\mu,\mathcal{D}},
\end{equation}
where $C$ depends on $n$, $ |\omega_0|_{\mathcal{C}^1 \left( \mathcal{D} \right) }$ and $\vert  A \vert_{\beta,{\mathcal{C}^1 \left( \mathcal{D} \right) }}$.

The functions $f_{\zeta}$ and $ R_{\zeta}$ are complex, so the constructed solution $S_\zeta$ may also be complex. Instead of proving that it is real, we replace $S_\zeta$, $\theta \in \mathbb{T}^n$, by its real part and then analytically extend it to $\mathbb{T}^n_{\sigma'}$, using the relation $\mathfrak{R}S_\zeta (\theta, \rho) := \frac{1}{2} S_\zeta (\theta, \rho) + \bar{S}_\zeta ( \bar{\theta}, \rho))$. Thus we obtain a real solution which obeys the same estimates.
\end{proof}
\begin{remark}
The key to solve the third equation is to control the eigenvalues of the Hermitian block diagonal matrix $k \cdot \omega I + Q$ which is of class $\mathcal{C}^1$ on $\mathcal{D}$. To do this, we consider a given block and we solve the equation block by block. Specifically, we consider the Hermitian matrix $k \cdot \omega I_{[s]} + Q_{[s]}$ which is the restriction of $k \cdot \omega I + Q$ to the block $[s] \times [s]$, where $Q_{[s]}(\rho)=diag \left\lbrace \lambda_{\ell}(\rho),\; \ell \in [s] \right\rbrace + H_{[s]}(\rho)$ with $H_{[s]}(\rho)$ a Hermitian  $\mathcal{C}^1$ matrix. Let $\alpha_\ell$ be an eigenvalue of $Q_{[s]}$, then $\alpha_\ell - \lambda_\ell$ is an eigenvalue of $H_{[s]}$. One of the key points to solve the third equation is to control $\partial_\rho (\alpha_\ell - \lambda_\ell)$. The hypothesis that $k\cdot \omega I_{[s]} + Q_{[s]}$ is Hermitian is essential. Indeed, if the matrix is just diagonalisable, then we cannot hope to find an orthonormal basis in which the matrix is diagonal. In this case, we cannot control $\vert \partial_\rho( \alpha_\ell -\lambda_\ell) \vert$ by $\Vert H_{[s]} \Vert$. The other difficulty is that the matrix is not analytical. Recall that the eigenvalue of a $\mathcal{C}^1$ matrix are not necessarily $\mathcal{C}^1$. If we further assume that the matrix $H$ is analytic in $\rho$ (which is equivalent to assume that $A'$ is analytic), then by using the assumption $\vert \partial_\rho^j H \vert_\beta \leq \frac{\delta}{2}$ for $j=0,1$, we have
\begin{equation} \nonumber
\vert \partial_\rho^j (\alpha_\ell - \lambda_\ell) \vert \leq  \Vert \partial_ \rho^j H_{[s]} \Vert \leq \Vert \partial_\rho^j H_{[s]} \Vert_{\infty} \leq \frac{1}{\langle s \rangle^{2\beta}} \vert \partial_\rho^j H(\rho) \vert_\beta \leq  \frac{\delta}{2\langle s\rangle ^{2\beta}} \leq \frac{\delta}{2}.
\end{equation}
Then, by using the hypothesis $A2$, we have either 
\begin{equation} \nonumber
\vert k \cdot \omega + \alpha_\ell \vert \geq \kappa \langle s \rangle,
\end{equation}
or there exists a unit vector $z_k \in \mathbb{R}^p$ such that
\begin{equation} \nonumber
\vert \langle \partial_\rho ( k \cdot \omega + \alpha_\ell ) , z_k \rangle \vert \geq \frac{\delta}{2},
\end{equation}
and we conclude by using the same type of arguments as in the Proposition~\ref{melnikov}.
\end{remark}
\subsubsection{The fourth equation.}
\begin{proposition}
Let $0<\kappa < \frac{\delta}{4} \leq \frac{1}{8} \min ( c_0,c_1)$ , $N>1$, $0<\sigma'<\sigma$, $\mu>0$ and $\omega : \mathcal{D} \to \mathbb{R}^n$ be $\mathcal{C}^1$ and verifying $\vert \omega - \omega_0 \vert_{\mathcal{C}^1(\mathcal{D})} \leq \delta_0$. Let $A : \mathcal{D} \to \mathcal{NF} \cap \mathcal{M}_\beta$ be $\mathcal{C}^1$ and satisfying $ \vert A - A_0  \vert_{\beta, \mathcal{C}^1(\mathcal{D}} \leq \delta_0$.
Then there exists a closed subset $\mathcal{D}_3 \subset \mathcal{D}$ whose Lebesgue measure satisfy
\begin{equation} \nonumber
{mes}(\mathcal{D} \setminus \mathcal{D}_3) \leq \tilde C d (\delta^{-1}\kappa)^\iota N^\upsilon,
\end{equation}
where $\iota , \upsilon >0$ and such that for all $\rho \in \mathcal{D}_3$, there exist real $ \mathcal{C}^1$ functions $\hat{K} : \mathcal{D}_3 \rightarrow \mathcal{M}_\beta	\cap \mathcal{NF} $, $S_{\zeta\zeta} (.;\rho)$ and $R_{\zeta\zeta} (.;\rho)$ : $\mathbb{T}^n_{\sigma'}\times \mathcal{D}_3 \rightarrow \mathcal{M}_\beta	$  analytic in $\theta$ such that:
\begin{align*}
\nabla_\theta S_{\zeta \zeta} ( \theta,\rho).\omega(\rho) & +A(\rho)JS_{\zeta \zeta} ( \theta,\rho) -S_{\zeta \zeta} ( \theta,\rho) JA(\rho)  =-f_{\zeta \zeta} ( \theta,\rho) + \hat{K}(\rho)(\rho)+R_{\zeta \zeta} ( \theta,\rho),
\end{align*}
and for all $(\theta,\rho) \in \mathbb{T}^n_{\sigma'}\times \mathcal{D}_3$ we have: 
\begin{align*}
\underset{\vert Im \theta \vert < \sigma'}{ \sup} \vert  S_{\zeta\zeta}(\theta,\rho) \vert_{\beta+} & \leq \frac{C}{\kappa \mu^2 (\sigma-\sigma')^{n}} \lc f^T \rc^{\alpha,\beta} _{\sigma,\mu,\mathcal{D}},\\
\underset{\vert Im \theta \vert < \sigma'}{ \sup} \vert \partial_\rho S_{\zeta\zeta}(\theta,\rho) \vert_{\beta+} & \leq \frac{C}{\kappa^2 \mu^2 (\sigma-\sigma')^{2n}} \lc f^T \rc^{\alpha,\beta} _{\sigma,\mu,\mathcal{D}} \\ &+\frac{C}{\kappa \mu^2 (\sigma-\sigma')^{n}} \lc \partial_\rho f^T \rc^{\alpha,\beta} _{\sigma,\mu,\mathcal{D}},\\
\underset{\vert Im \theta \vert < \sigma'}{ \sup} \vert \partial^j_\rho R_{\zeta\zeta}(\theta) \vert_\beta & \leq \frac{Ce^{-(\sigma-\sigma')N/2}}{\mu^2(\sigma-\sigma')^{n}} \lc f^T \rc^{\alpha,\beta} _{\sigma,\mu,\mathcal{D}}, \quad j=0,1,\\
\vert \partial^j_\rho \hat{K}(\rho) \vert_\beta & \leq \frac{\lc \partial_\rho^j f^T \rc^{\alpha,\beta} _{\sigma,\mu,\mathcal{D}}}{\mu^{2}},\quad j=0,1.
\end{align*}
The constant $C$ depends on   $n$, $ |\omega_0|_{\mathcal{C}^1 \left( \mathcal{D} \right) }$ and $\vert  A_0 \vert_{\beta,{\mathcal{C}^1 \left( \mathcal{D} \right) }}$, while $\tilde{C}$ depends on $ |\omega_0|_{\mathcal{C}^1 \left( \mathcal{D} \right) }$, $c_0$ and $c_1$.
\end{proposition}
\begin{proof} 
Using the same notations and complex variables as in the third equation, we can decouple the fourth equation into three equations:
\begin{align}
\begin{array}{c} \label{quatrieme eq dec}
\nabla_\theta S_{\xi\xi} (\theta) \cdot \omega +i  Q S_{\xi\xi} (\theta) + i S_{\xi\xi} (\theta) {}^tQ = -f_{\xi\xi} (\theta) + R_ {\xi\xi}( \theta),\\
\nabla_\theta S_{\eta\eta} (\theta) \cdot \omega  - i S_{\eta\eta} (\theta) Q - i {}^tQ S_{\eta\eta} (\theta) = -f_{\eta\eta} (\theta) + R_ {\eta\eta}( \theta),\\
\nabla_\theta S_{\xi\eta} (\theta) \cdot \omega +i Q S_{\xi\eta} (\theta) - i S_{\xi\eta} (\theta) Q   = K - f_{\xi\eta} (\theta)+R_{\xi\eta}(\theta).
\end{array}
\end{align}
We start by expanding $S_{\xi\xi}, S_{\xi\eta}, S_{\xi\eta},f_{\xi\xi}, f_{\xi\eta}, f_{\xi\eta}, R_{\xi\xi}, R_{\xi\eta}$ and  $R_{\xi\eta}$ in Fourier series. then for any $k \in \mathbb{Z}^n$ we have:
\begin{align}
\label{homo4-1} ( k\cdot\omega I + Q )\hat{S}_{\xi\xi} (k) +\hat{S}_{\xi\xi} (k){}^tQ & = i( \hat{f}_{\xi\xi} (k) - \hat{R}_{\xi\xi}(k)), \\
\label{homo4-2} ( k \cdot \omega I - {}^tQ )\hat{S}_{\eta\eta} (k) - \hat{S}_{\eta\eta} (k) Q & = i( \hat{f}_{\eta\eta} (k) - \hat{R}_{\eta\eta}(k)), \\
\label{homo4-3} ( k \cdot \omega I + Q )\hat{S}_{\xi\eta} (k) - \hat{S}_{\xi\eta} (k)Q & = -i( \delta_{k,0}\hat{K} -\hat{f}_{\xi\eta} (k) - \hat{R}_{\xi\eta}(k)),
\end{align}
where $\delta_{k,j}$ is the Kronecker symbol.

Equation \eqref{homo4-1} and \eqref{homo4-2} are of the same form, So to end the resolution of the homological equation, it is enough to solve \eqref{homo4-1} and \eqref{homo4-3}.

\textbf{The equation \eqref{homo4-1}. } Consider a matrix $M \in \mathcal{M}_\beta$, we denote by $M_{[s ],[s']}$ the restriction of $M$ to the block $[s] \times [s']$ and we define $M_{[s]}:=M_{[s],[s]}$. We recall that $Q=D+H$ where $D=\diag \left\lbrace \lambda_{s}(\rho),\; s \in \mathcal{L} \right\rbrace$ and $H$ is Hermitian and block diagonal matrix. We start by decomposing the equation \eqref{homo4-1} over the block $[s] \times [s']$:
\begin{equation} \label{homologique 4-1 dec}
k \cdot\omega (\hat{S}_{\xi\xi})_{[s],[s']} (k)  +    Q_{[s]} (\hat{S}_{\xi\xi})_{[s],[s']} (k) +(\hat{S}_{\xi\xi})_{[s],[s']} (k){}^tQ_{[s']}=
  i( (\hat{f}_{\xi\xi})_{[s],[s']} (k) - (\hat{R}_{\xi\xi})_{[s],[s']}(k)).
\end{equation}
$Q_{[s]}$ is Hermitian, so we can diagonalize it in an orthonormal basis: $\tilde{Q}_{[s]}={}^tP_{[s]} Q_{[s]} P_{[s]}$ (similarly for ${}^tQ_{[s']}$). 
We denote $(\hat{S}_{\xi\xi})'_{[s],[s']}(k)  ={}^tP_{[s]} (\hat{S}_{\xi\xi})_{[s],[s']}(k) P_{[s']}$, $(\hat{f}_{\xi\xi})'_{[s],[s']}(k)  ={}^tP_{[s]} (\hat{f}_{\xi\xi})_{[s],[s']}(k) P_{[s']}$ and $(\hat{R}_{\xi\xi})'_{[s],[s']}(k) = {}^tP_{[s]} (\hat{R}_{\xi\xi})_{[s],[s']}(k) P_{[s']}$.
By multiplying equation \eqref{homologique 4-1 dec} by ${}^tP_{[s]}$ on the left and $P_{[s']}$ on the right, we get:
\begin{equation} \label{eqhmo4diag}
(k \cdot \omega I_{[s]} +  \tilde{Q}_{[s]}) (\hat{S}_{\xi\xi})'_{[s],[s']} (k)  +  (\hat{S}_{\xi\xi})'_{[s],[s']} (k) \tilde{Q}_{[s']}  = i (\hat{f}_{\xi\xi})'_{[s],[s']} (k) - i (\hat{R}_{\xi\xi})'_{[s],[s']}(k).
\end{equation}
Let $\alpha_{\ell}$ denotes an eigenvalue of $\tilde{Q}_{[s]}$ and $\alpha_{\ell'}$ an eigenvalue of $\tilde{Q}_{[s']}$. To solve \eqref{homo4-1}), we need to find a lower bound of the modulus of 
\begin{equation} \nonumber
k \cdot \omega + \alpha_{\ell} + \alpha_{\ell'},
\end{equation}
for $ \ell \in [s]$ and $ \ell' \in [s']$. The hypothesis \ref{Hypothèse B} implies that $\vert  H \vert_\beta \leq \frac{\delta}{4}$ and we obtain:
\begin{equation} \nonumber
\vert \alpha_{\ell}-\lambda_{\ell} \vert \leq \Vert H_{[s]} \Vert \leq \Vert H_{[s]} \Vert_{\infty} \leq \langle s \rangle^{-2\beta} \vert H(\rho) \vert_\beta \leq  \frac{\delta}{4} \langle s\rangle ^{-2\beta} \leq \frac{\delta}{4} \leq \frac{c_0}{8}.
\end{equation}
So for $\ell \in [s]$ and $\ell' \in [s']$  we obtain:
\begin{align*}
\vert \alpha_\ell  + \alpha_{\ell'} \vert & \geq \lambda_{\ell} -  \vert \alpha_{\ell}-\lambda_{\ell} \vert + \lambda_{\ell'} - \vert \alpha_{\ell'}-\lambda_{\ell'} \vert \\
& \geq c_0 \langle s\rangle - \frac{c_0}{8} \langle s\rangle + c_0 \langle s'\rangle - \frac{c_0}{8} \langle s' \rangle \\
& \geq \kappa (\langle s \rangle + \langle s' \rangle).
\end{align*}
Thus we obtain a lower bound when $k=0$. Assume now that $k \neq 0 $, by hypothesis A2 we have either
\begin{equation} \nonumber
\vert k \cdot \omega + \alpha_{\ell'} + \alpha_{\ell}  \vert \geq \delta(\langle s \rangle + \langle s' \rangle) - \frac{\delta}{2} \geq \kappa(\langle s \rangle + \langle s' \rangle),
\end{equation}
or there exists a unit vector $z_k \in \mathbb{R}^p$ such that
\begin{equation*}
\vert \langle \partial_\rho (k\cdot\omega+\lambda_{\ell'}+\lambda_{\ell}),z_k \rangle \vert  \geq \delta.
\end{equation*}
Let us consider the second case. We note that for this unit vector $z_k \in \mathbb{R}^p$, we have:
\begin{equation} \nonumber
 \Vert (\partial_\rho \cdot z_k) H_{[s]} \Vert \leq  \Vert (\partial_\rho \cdot z_k) H_{[s]} \Vert_{\infty} \leq  \Vert \partial_\rho H_{[s]} \Vert_{\infty} \leq \frac{\delta}{4} \langle s \rangle^{-2\beta} \leq \frac{\delta}{4}.
\end{equation}
By simplifying the notations, we can rewrite the equation \eqref{eqhmo4diag} in the form:
\begin{equation} \label{eqhmo4diag-op}
L(k,s,s') S_{[s],[s']} = i F_{[s],[s']},
\end{equation}
where
\begin{equation} \nonumber
L(k,s,s') S = (k \cdot \omega I_[s] + \tilde{Q}_{[s]} ) S_{[s],[s']} + S_{[s],[s']}  \tilde{Q}_{[s']},
\end{equation}
and
\begin{equation} \nonumber
F_{[s],[s']}= f_{[s],[s']} + R_{[s],[s']}.
\end{equation}
For $A, B \in \mathcal{M}_d(\mathbb{C})$, we consider the scalar product $\langle A,B \rangle = \tr(AB^*)$. We remark that $L$ is linear Hermitian operator. Consider
\begin{equation} \nonumber
J(k,s,s') = \lbrace \rho \in \mathcal{D} \mid \Vert (L(k,s,s'))^{-1} \Vert > (\kappa (\langle s \rangle + \langle s' \rangle))^{-1} \rbrace
\end{equation}
Following the same procedure to prove the Lemme~\ref{controle inverse} in \cite{eliasson2010kam}, we obtain that:
\begin{equation} \nonumber 
\operatorname{mes} J(k,s,s') \leq 4 d \delta^{-1} \kappa(\langle s \rangle + \langle s' \rangle) \leq 16 C c_0^{-1} d \kappa \delta^{-1} N ,
\end{equation}
for $ \max (\vert s \vert, \vert s' \vert ) \leq  4 C c_0^{-1}N$ and $ C:= \vert \omega \vert_{\mathcal{C}^1 \left( \mathcal{D} \right)}+ \delta_0 \leq \vert \omega \vert_{\mathcal{C}^1 \left( \mathcal{D} \right)} +1 $.
Consider the set
\begin{equation} \nonumber
\mathcal{B}_0 = \lbrace (k,s,s') \in \mathbb{Z}^n \times \mathcal{L} \times \mathcal{L} \mid \vert k \vert \leq N ,\; \max (\vert s \vert, \vert s' \vert ) \leq 4 C c_0^{-1} N  \rbrace
\end{equation}
This set contains at most $2^4C^2 c_0^{-2}N^{n+2}$ points. Then
\begin{equation} \nonumber
\operatorname{mes} \bigcup_{(k,s,s') \in  \mathcal{B}_0} J(k,s,s') \leq 2^8C^3 c_0^{-3}d \kappa \delta^{-1}N^{n+3} = \tilde{C}d \kappa \delta^{-1}N^{n+3}. 
\end{equation}
Let $ \displaystyle \mathcal{D}_1 = \mathcal{D} \setminus \bigcup_{(k,s,s') \in  \mathcal{B}_0} J(k,s,s')$, then for $ \rho \in \mathcal{D}_1$ we have
\begin{equation} \nonumber
\vert k \cdot \omega + \alpha_\ell + \alpha_{\ell'} \vert \geq \kappa ( \langle s \rangle + \langle s' \rangle).
\end{equation}
if $ \max (\vert s \vert, \vert s' \vert ) > 4 C c_0^{-1}N$, then we have:
\begin{align*}
\vert k \cdot \omega + \alpha_\ell + \alpha_{\ell'} \vert & \geq  \lambda_\ell + \lambda_{\ell'}  -  \vert \alpha_\ell - \lambda_\ell \vert - \vert \alpha_{\ell'} - \lambda_{\ell '}\vert - \vert k \cdot \omega \vert \\
& \geq  c_0 ( \langle s \rangle + \langle s' \rangle ) - \frac{c_0}{8} - \frac{c_0}{8} - \frac{c_0}{4} ( \langle s \rangle + \langle s' \rangle ) \\
& \geq \kappa ( \langle s \rangle + \langle s' \rangle ).
\end{align*}
For $ \ell \in [s]$ and $ \ell' \in [s']$, we can solve \eqref{eqhmo4diag} term by term:
\begin{equation} \nonumber
^{\ell}_{\ell'} (\hat{S}_{\xi\xi})'_{[s],[s']} (k)=\frac{i}{ k \cdot \omega + \alpha_{\ell} + \alpha_{\ell'}} \: ^{\ell}_{\ell'} (\hat{f}_{\xi\xi})'_{[s],[s']} (k) \quad \text{for}\: \vert k \vert \leq N,
\end{equation}
\begin{equation} \nonumber
^{\ell}_{\ell'}(\hat{R}_{\xi\xi})'_{[s],[s']} (k) =\: ^{\ell}_{\ell'}(\hat{f}_{\xi\xi})'_{[s],[s']}(k) ,\quad \text{for}\: \vert k \vert > N.
\end{equation}
Using the same arguments as in the first three equations, we obtain for any $0 < \sigma < \sigma'$ that
\begin{align*}
 \Vert (\hat{S}_{\xi\xi})'_{[s],[s']} (k) \Vert_\infty & \leq \frac{C}{\kappa(\langle s \rangle + \langle s' \rangle)(\sigma-\sigma')^{n}} \Vert (\hat{f}_{\xi\xi})'_{[s],[s']} (k) \Vert_\infty,\\
\Vert (\hat{R}_{\xi\xi})'_{[s],[s']} (k) \Vert_\infty & \leq \frac{Ce^{-(\sigma-\sigma')N/2}}{(\sigma-\sigma')^{n}}  \Vert (\hat{f}_{\xi\xi})'_{[s],[s']} (k) \Vert_\infty.
\end{align*}
Recall that the matrices $P_{[s']}$ and $P_{[s]}$ are unitary, so we obtain
\begin{align}
\underset{\vert Im \theta \vert < \sigma'}{ \sup} \vert (S_{\xi\xi})(\theta) \vert_{\beta+} & \leq \frac{C}{ \kappa(\sigma-\sigma')^{n}} \underset{\vert Im \theta \vert < \sigma}{ \sup} \vert (f_{\xi\xi}) (\theta) \vert_\beta, \label{estim S eq 4-1} \\
\underset{\vert Im \theta \vert < \sigma'}{ \sup} \vert (R_{\xi\xi})(\theta) \vert_\beta & \leq \frac{Ce^{-(\sigma-\sigma')N/2}}{(\sigma-\sigma')^{n}} \underset{\vert Im \theta \vert < \sigma}{ \sup} \vert (f_{\xi\xi}) (\theta) \vert_\beta. \label{estim R eq 4-1}
\end{align}
These estimates leads to
\begin{equation} \nonumber
\underset{\vert Im \theta \vert < \sigma'}{ \sup} \vert  S_{\xi\xi}(\theta,\rho) \vert_{\beta+} \leq \frac{C}{\kappa \mu^2 (\sigma-\sigma')^{n}} \lc f^T \rc^{\alpha,\beta} _{\sigma,\mu,\mathcal{D}},
\end{equation}
and
\begin{equation} \nonumber
\underset{\vert Im \theta \vert < \sigma'}{ \sup} \vert  R_{\xi\xi}(\theta) \vert_\beta  \leq \frac{Ce^{-(\sigma-\sigma')N/2}}{\mu^2(\sigma-\sigma')^{n}} \lc f^T \rc^{\alpha,\beta} _{\sigma,\mu,\mathcal{D}},
\end{equation}
where $C$ depends on $n$ and $\beta$.

To obtain the estimates of the derivative of  $S_{\xi\xi}$ and $R_{\xi\xi}$ with respect to $\rho$, we differentiate the equation \eqref{homo4-1}). So we obtain
\begin{equation} \nonumber 
( k \cdot \omega I +{}^t Q )\partial_\rho \hat{S}_{\xi\xi} (k) +\partial_\rho \hat{S}_{\xi\xi} (k)Q= - \partial_\rho ( k \cdot \omega I +{}^t Q ) \hat{S}_{\xi\xi} (k) - \hat{S}_{\xi\xi} (k) \partial_\rho Q  + i \hat{f}_{\xi\xi} (k),
\end{equation}
for $\vert k \vert \leq N$, and
\begin{equation} \nonumber
\partial_\rho R_{\xi\xi} ( \theta) =  \underset{ \vert k \vert \geq N }{ \sum} \partial_\rho ( \hat{f}_{\xi\xi})(k) e^{ik\theta}.
\end{equation}
Using the same method as to prove the estimates \eqref{estim S eq 4-1} and  \eqref{estim R eq 4-1} we obtain:
\begin{align*} \nonumber
\underset{\vert Im \theta \vert < \sigma'}{ \sup} \vert \partial_\rho S_{\xi\xi}(\theta,\rho) \vert_{\beta+}  \leq \frac{C}{\kappa^2(\sigma-\sigma')^{2n}}\underset{\vert Im \theta \vert < \sigma}{ \sup} \vert f_{\xi\xi}(\theta,\rho) \vert_\beta + \frac{C}{\kappa(\sigma-\sigma')^{n}} \underset{\vert Im \theta \vert < \sigma'}{ \sup} \vert \partial_\rho f_{\xi\xi}(\theta,\rho) \vert_\beta.
\end{align*}
and
\begin{equation} \nonumber
\underset{\vert Im \theta \vert < \sigma'}{ \sup} \vert \partial_\rho R_{\xi\xi}(\theta) \vert_\beta  \leq \frac{Ce^{-(\sigma-\sigma')N/2}}{(\sigma-\sigma')^{n}} \underset{\vert Im \theta \vert < \sigma}{ \sup} \vert \partial_\rho f_{\xi\xi} (\theta) \vert_\beta,
\end{equation}
These estimates leads to
\begin{align*} \nonumber
\underset{\vert Im \theta \vert < \sigma'}{ \sup} \vert \partial_\rho S_{\xi\xi}(\theta,\rho) \vert_{\beta+} \leq  \frac{C}{\kappa^2 \mu^2 (\sigma-\sigma')^{2n}} \lc f^T \rc^{\alpha,\beta} _{\sigma,\mu,\mathcal{D}} +\frac{C}{\kappa \mu^2 (\sigma-\sigma')^{n}} \lc \partial_\rho f^T \rc^{\alpha,\beta} _{\sigma,\mu,\mathcal{D}}.
\end{align*}
and
\begin{equation} \nonumber
\underset{\vert Im \theta \vert < \sigma'}{ \sup} \vert \partial_\rho R_{\xi\xi}(\theta) \vert_\beta  \leq \frac{Ce^{-(\sigma-\sigma')N/2}}{\mu^2(\sigma-\sigma')^{n}} \lc \partial_\rho f^T \rc^{\alpha,\beta} _{\sigma,\mu,\mathcal{D}}.
\end{equation}
The constant $C$ depends on $n$, $ |\omega_0|_{\mathcal{C}^1 \left( \mathcal{D} \right) }$ and $\vert  A_0 \vert_{\beta,{\mathcal{C}^1 \left( \mathcal{D} \right) }}$.
\\To get the estimates on $ S_{\eta\eta}(\theta)$, $\partial_\rho S_{\eta\eta}(\theta)$ , $ R_{\eta\eta}(\theta)$ and $\partial_\rho R_{\eta\eta}(\theta)$ we follow the same procedure.

\textbf{The equation \eqref{homo4-3}.} It remains to solve the equation \eqref{homo4-3}. We start by decomposing the equation over the block $[s] \times [s']$
 \begin{align*}
k \cdot \omega (\hat{S}_{\xi\eta})_{[s],[s']} (k) - Q_{[s]} (\hat{S}_{\xi\eta})_{[s],[s']} (k) & +(\hat{S}_{\xi\eta})_{[s],[s']} (k)Q_{[s']}
\\ & = -i( \delta_{k,0}\hat{K}_{[s],[s']} - (\hat{f}_{\xi\eta})_{[s],[s']} (k) + (\hat{R}_{\xi\eta})_{[s],[s']}(k)).
\end{align*}
Using the same notation as in \eqref{eqhmo4diag} we obtain the following equation:
\begin{align}
 \label{eqhmo4diag-2}
k \cdot \omega (\hat{S}_{\xi\eta})'_{[s],[s']} (k) - \tilde{Q}_{[s]} (\hat{S}_{\xi\eta})'_{[s],[s']} (k) & +(\hat{S}_{\xi\eta})'_{[s],[s']} (k) \tilde{Q}_{[s']}\\ \nonumber
& = - i(\delta_{k,0}\hat{K}_{[s],[s']}' - (\hat{f}_{\xi\eta})'_{[s],[s']} (k) + (\hat{R}_{\xi\eta})'_{[s],[s']}(k)).
\end{align}
Let $\alpha_{\ell}$ denotes an eigenvalue of $\tilde{Q}_{[s]}$ and $\alpha_{\ell'}$ an eigenvalue of $\tilde{Q}_{[s']}$. To solve \eqref{homo4-3}, we need to find a lower bound of the modulus of
\begin{equation} \nonumber
k \cdot \omega + \alpha_{\ell} - \alpha_{\ell'}.
\end{equation} 
We will distinguish two cases: $k=0$ or $k \neq 0 $. Assume that $k =0$. We will distinguish two cases:
\begin{itemize}
\item If  $[s]=[s']$, then $k \cdot \omega + \alpha_{\ell} - \alpha_{\ell'}=0$. We solve the equation by posing
\begin{equation} \nonumber
(\hat{S}_{\xi\eta})'_{[s],[s]} (0)=0,\quad \hat{K}'_{[s],[s]}= (\hat{f}_{\xi\eta})'_{[s],[s]} (0).
\end{equation}
\item If $[s] \neq [s']$, then for $\ell \in [s]$ and $\ell' \in [s']$ and by hypothesis A1 we have 
\begin{align*}\nonumber
\vert \alpha_{\ell} - \alpha_{\ell'} \vert  \geq  c_1 (\vert | s| -|s'|  \vert) - \frac{\delta}{2} \geq \frac{c_1}{2}(\vert | s| -|s'|  \vert)  \geq \kappa (  1 + \vert  \vert s \vert - \vert s' \vert  \vert).
\end{align*}
In this case we solve the equation \eqref{eqhmo4diag-2} by posing
\begin{equation} \nonumber
(\hat{S}_{\xi\eta})'_{[s'],[s]} (0) = i \frac{(\hat{f}_{\xi\eta})'_{[s'],[s]}(0)}{ \alpha_{s'} - \alpha_{s}},\quad K'_{[s'],[s]}=0.
\end{equation}
\end{itemize}
Assume now that $k \neq 0$. Using the second Melnikov condition, we have: for $\vert k \vert \leq N$ there exists a closed subset $\mathcal{D}_2 \subset \mathcal{D}$ whose Lebesgue measure verifies:
\begin{equation} \nonumber
\operatorname{mes} ( \mathcal{D} \setminus \mathcal{D}_2) \leq C (\delta^{-1} \tilde \kappa)^\tau N^\iota,
\end{equation}
for $\tau, \iota >0$, such that for $\rho \in \mathcal{D}_2$ we have:
\begin{equation} \nonumber
\vert k \cdot \omega +\lambda_{\ell}-\lambda_{\ell'} \vert \geq 2 \tilde \kappa(1+ \vert \vert s \vert - \vert s' \vert \vert ).
\end{equation}
At the end of the resolution of the homological equation, we will fix $ \tilde \kappa$ such that $ \tilde{\kappa} < \kappa$. For $ \rho \in \mathcal{D}_2$  we have:
\begin{align*}
\vert k \cdot \omega + \alpha_{\ell'} - \alpha_{\ell}  \vert & \geq  2 \tilde{\kappa}(1+ \vert \vert s \vert - \vert s' \vert \vert ) - \frac{\delta}{4} \langle s \rangle^{-2 \beta} - \frac{\delta}{4}  \langle s' \rangle^{-2 \beta} \geq \tilde{\kappa}(1+ \vert \vert s \vert - \vert s' \vert \vert ),
\end{align*}
for $ \min ( \vert s \vert , \vert s' \vert ) \geq \left( \frac{\delta}{2 \tilde{\kappa}} \right)^{1/2\beta}$. Assume now that $ \min ( \vert s \vert , \vert s' \vert ) < \left( \frac{\delta}{2 \tilde{\kappa}} \right)^{1/2\beta}$. Without losing generality, suppose for example that $ \vert s \vert < \left( \frac{\delta}{2 \tilde{\kappa}} \right)^{1/2\beta}$. By the hypothesis $A2$, we have either 
\begin{equation} \nonumber
\vert k \cdot \omega+\lambda_{\ell}-\lambda_{\ell'} \vert \geq  \delta (1+ \vert s \vert - \vert s' \vert \vert ),
\end{equation}
and this implies that
\begin{align*}
\vert k \cdot\omega + \alpha_{\ell'} - \alpha_{\ell}  \vert  \geq \delta (1+ \vert \vert s \vert - \vert s' \vert \vert ) - \frac{\delta}{4}  - \frac{\delta}{4} \geq \kappa (1+ \vert \vert s \vert - \vert s' \vert \vert ),
\end{align*}
or there exists a unit vector $z_k \in \mathbb{R}^p$ such that:
\begin{equation} \label{drv_4_eq_hom}
 \langle \partial_\rho (  k \cdot \omega + \lambda_{\ell} - \lambda_{\ell'} ) , z_k     \rangle \geq \delta.
\end{equation}
Let us consider the second case. We Recall that $ \vert s \vert \leq \left( \frac{\delta}{2 \tilde{\kappa}} \right)^{1/2\beta}$.  Assume that $ \vert \vert s \vert - \vert s' \vert \vert \geq 2 C c_1^{-1}N$ where  $ C:= \vert \omega \vert_{\mathcal{C}^1 \left( \mathcal{D} \right)}+ \delta_0 \leq \vert \omega \vert_{\mathcal{C}^1 \left( \mathcal{D} \right)} +1 $. Then we have: 
\begin{align*}
\vert k \cdot\omega + \alpha_{\ell} - \alpha_{\ell'}  \vert  \geq   c_1 ( \vert \vert s \vert - \vert s' \vert \vert ) - \frac{c_1}{2} ( \vert \vert s \vert - \vert s' \vert \vert )  - \frac{\delta}{2}    \geq \kappa (1+ \vert \vert s \vert - \vert s' \vert \vert ).
\end{align*}
It remains a finite number of cases, i.e when $ \vert s \vert \leq \left( \frac{\delta}{2 \tilde{\kappa}} \right)^{1/2\beta}$ and $ \vert \vert s \vert - \vert s' \vert \vert < 2 C c_1^{-1}N$. We note that this implies that $ \vert s ' \vert \leq 2 C c_1^{-1}N + \left( \frac{\delta}{2 \tilde{\kappa}} \right)^{1/2\beta}$. Consider the following set
\begin{equation}  \nonumber
I(k,s,s') = \lbrace \rho \in \mathcal{D} \mid \vert k \cdot \omega +  \alpha_{\ell} - \alpha_{\ell'} \vert < \kappa ( 1+ \vert \vert s \vert - \vert s' \vert \vert)  \rbrace.
\end{equation}
Suppose that the eigenvalues of the hermitian matrix $H$ are analytic, then we obtain
\begin{equation} \nonumber
\vert \partial_\rho (\alpha_\ell - \lambda_\ell) \vert \leq  \Vert \partial_ \rho H_{[s]} \Vert \leq \Vert \partial_\rho H_{[s]} \Vert_{\infty} \leq \frac{1}{\langle s \rangle^{2\beta}} \vert \partial_\rho H(\rho) \vert_\beta \leq  \frac{\delta}{4\langle s\rangle ^{2\beta}} \leq \frac{\delta}{4}.
\end{equation}
By \eqref{drv_4_eq_hom}, the Lebesgue measure of $I(k,s,s')$ satisfies:
\begin{equation} \nonumber 
\operatorname{mes} I(k,s,s') \leq 8  \delta^{-1} \kappa \left( \left( \frac{\delta}{2 \tilde{\kappa}} \right)^{1/2\beta} + C c_1^{-1}N  \right).
\end{equation}
If the eigenvalues of $H$ are not analytic, we obtain the same estimates by density of the analytic functions in the space of continuous functions. Consider the following set
\begin{equation*}
 \mathcal{B}_1 = \Bigg\{ (k,s,s') \in \mathbb{Z}^n \times \mathcal{L}\times \mathcal{L} \mid \vert k \vert \leq N , \; \vert s \vert \leq {\left(\frac{\delta}{2 \tilde{\kappa}}
\right)}^{1/2\beta}, \vert s ' \vert \leq 2 C c_1^{-1}N + {\left(
\frac{\delta}{2 \tilde{\kappa}}\right)}^{1/2\beta}  \Bigg\}.
\end{equation*}

This set contains at most $N^n \left( \frac{\delta}{2 \tilde{\kappa}} \right)^{\frac{1}{2 \beta}} (  2 C c_1^{-1}N + \left( \frac{\delta}{2 \tilde{\kappa}} \right)^{\frac{1}{2 \beta}})  $ points. We define $ \displaystyle \mathcal{D}_2' = \mathcal{D} \setminus \bigcup_{(k,s,s') \in  \mathcal{B}_1} I(k,s,s')$. The set $ \mathcal{D}_2'$ satisfies
\begin{align*} \nonumber 
\operatorname{mes} \mathcal{D} \setminus \mathcal{D}_2' & \leq N^n \left( \frac{\delta}{2 \tilde{\kappa}} \right)^{1/2\beta} \left(  2 C c_1^{-1}N + \left( \frac{\delta}{2 \tilde{\kappa}} \right)^{1/2\beta} \right) \delta^{-1} \kappa \times \left( \left( \frac{\delta}{2 \tilde{\kappa}} \right)^{1/2\beta} + C c_1^{-1}N  \right)\\
& \leq N^{n+2 } \left( \frac{\delta}{ \tilde{\kappa}} \right)^{3/2\beta} C_1^{-2}C^2 \kappa \delta^{-1}.
\end{align*}
Let us fix $ \tilde{\kappa} = \delta \left( \frac{\kappa}{\delta} \right)^{\frac{6 \beta}{9 + 2 \beta}}$. For this choice, we have $\tilde{\kappa} < \kappa$, and
\begin{equation} \nonumber
\operatorname{mes} \mathcal{D} \setminus \mathcal{D}_2' \leq \tilde{C}  N^{n+2}  \left( \frac{\kappa}{\delta} \right)^{\frac{2 \beta}{9 + 2 \beta}}.
\end{equation}
By construction, for $\rho \in \mathcal{D}_2 \cap \mathcal{D}_2'$, we have
\begin{equation*}
\vert k \cdot \omega + \alpha_{\ell} - \alpha_{\ell'} \vert    \geq \kappa (1 + \vert  \vert s \vert - \vert s' \vert \vert ).
\end{equation*} 
For $\rho \in \mathcal{D}_2 \cap \mathcal{D}_2'$, $ \ell \in [s]$ and $ \ell' \in [s']$, we solve the equation  \eqref{eqhmo4diag-2} by posing 
\begin{equation} \nonumber
^{\ell}_{\ell'} (\hat{S}_{\xi\eta})'_{[s],[s']} (k)=\frac{i}{ k \cdot \omega + \alpha_{\ell} - \alpha_{\ell'}} \: ^{\ell}_{\ell'} (\hat{f}_{\xi\eta})'_{[s],[s']} (k) \quad \text{for}\: \vert k \vert \leq N,
\end{equation}
\begin{equation} \nonumber
^{\ell}_{\ell'}(\hat{R}_{\xi\eta})'_{[s],[s']} (k) =\: ^{\ell}_{\ell'}(\hat{f}_{\xi\eta})'_{[s],[s']}(k) ,\quad \text{for}\: \vert k \vert > N.
\end{equation}
The same reasoning as in the equation \eqref{homo4-1} gives us that:
\begin{align*}
\underset{\vert Im \theta \vert < \sigma'}{ \sup} \vert (S_{\xi\eta})(\theta) \vert_{\beta+} & \leq \frac{C}{ \kappa \mu^2 (\sigma-\sigma')^{n}}  \lc f^T \rc^{\alpha,\beta}_{\sigma,\mu,\mathcal{D}}, \\
\underset{\vert Im \theta \vert < \sigma'}{ \sup} \vert (R_{\xi\eta})(\theta) \vert_\beta & \leq \frac{Ce^{-(\sigma-\sigma')N/2}}{(\sigma-\sigma')^{n}}  \lc f^T \rc^{\alpha,\beta}_{\sigma,\mu,\mathcal{D}},\\
\underset{\vert Im \theta \vert < \sigma'}{ \sup} \vert \partial_\rho S_{\xi\eta}(\theta,\rho) \vert_{\beta+} & \leq \frac{C}{\kappa^2 \mu^2 (\sigma-\sigma')^{2n}} \lc f^T \rc^{\alpha,\beta}_{\sigma,\mu,\mathcal{D}}  +\frac{C}{\kappa \mu^2 (\sigma-\sigma')^{n}} \lc \partial_\rho f \rc^{\alpha,\beta}_{\sigma,\mu,\mathcal{D}},\\
\underset{\vert Im \theta \vert < \sigma'}{ \sup} \vert \partial_\rho R_{\xi\eta}(\theta) \vert_\beta & \leq \frac{Ce^{-(\sigma-\sigma')N/2}}{\mu^2(\sigma-\sigma')^{n}} \lc \partial_\rho f^T \rc^{\alpha,\beta} _{\sigma,\mu,\mathcal{D}},\\
\vert \partial^j_\rho \hat{K}(\rho) \vert_\beta & \leq \frac{\lc \partial_\rho^j f^T \rc^{\alpha,\beta} _{\sigma,\mu,\mathcal{D}}}{\mu^{2}},\quad j=0,1.
\end{align*}
the constant $C$ depends on  $n$, $ |\omega_0|_{\mathcal{C}^1 \left( \mathcal{D} \right) }$ and $\vert  A_0 \vert_{\beta,{\mathcal{C}^1 \left( \mathcal{D} \right) }}$.
\\This completes the resolution of the linear homological equation.

In this way we have constructed a solution $S_{\zeta\zeta}, R_{\zeta\zeta}$ , $\hat{K}$  of the fourth component of the linear homological equation which satisfies all required estimates. To guarantee that it is real, as at the of the resolution of the third equation, we replace $S_{\zeta\zeta}, R_{\zeta\zeta}$ , $\hat{K}$ by their real parts and extend it analytically to $\mathbb{T}^n_{\sigma'}$ (i.e. replace $S_{\zeta\zeta}(\theta, \rho)$ by $\frac{1}{2}(S_{\zeta\zeta}(\theta, \rho) + \bar{S}_{\zeta\zeta}(\bar{\theta}, \rho)))$.
\end{proof}
The following theorem summarizes the results obtained in the resolution of the linear homological equation.
\begin{theorem}\label{theoreme homo}
Let $0<\kappa < \frac{\delta}{4} \leq \frac{1}{8} \min (c_0, c_1)$ , $N>1$, $0<\sigma'<\sigma$, $\mu >0$ and $\omega : \mathcal{D} \to \mathbb{R}^n$ be $\mathcal{C}^1$ and verifying $\vert \omega - \omega_0 \vert_{\mathcal{C}^1(\mathcal{D})} \leq \delta_0$. Let $A : \mathcal{D} \to \mathcal{NF} \cap \mathcal{M}_\beta$ be $\mathcal{C}^1$ and satisfying $\vert A - A_0  \vert_{\beta, \mathcal{C}^1(\mathcal{D}} \leq \delta_0$.
Assume that $\partial^j_\rho f \in \mathcal{T}^{\alpha,\beta}(\sigma,\mu,\mathcal{D})$ for $j=0,1$. Then there exists a closed subset $\mathcal{D'} \equiv \mathcal{D'}(\kappa, N) \in \mathcal{D}$ such that
\begin{equation} \nonumber
\operatorname{mes} (\mathcal{D} \setminus \mathcal{D'}) \leq \tilde C d (\kappa \delta^{-1})^\iota N^\upsilon,
\end{equation}
For $\rho \in \mathcal{D'}$, there exist two real jet-functions $S=S^T$ and  $R=R^T$ such that $\partial_\rho^j S \in \mathcal{T}^{\alpha,\beta+}(\sigma',\mu,\mathcal{D'})$ and $\partial_\rho^j R \in \mathcal{T}^{\alpha,\beta+}(\sigma',\mu,\mathcal{D'})$  and a normal form 
\begin{equation} \nonumber
\hat{h} (\rho)=\int_{\mathbb{T}^n} f(\theta,0,0,\rho)d\theta+\int_{\mathbb{T}^n} f_r(\theta,0,0,\rho)d\theta .r+\frac{1}{2} \langle \zeta,\hat{K}(\rho)\zeta\rangle,
\end{equation}
satisfying
\begin{equation} \nonumber
\lbrace h, S \rbrace + f^T = \hat{h} + R.
\end{equation}
Furthermore, for $ \rho \in \mathcal{D'} $ we have:
\begin{equation} \label{estima B}
\vert \partial_\rho^j \hat{K}(\rho) \vert_\beta  \leq \frac{\llbracket \partial_\rho^j f^T \rrbracket^{\alpha,\beta} _{\sigma,\mu,\mathcal{D}}}{\mu^{2}}, \quad j=0,1
\end{equation}
\begin{equation} \label{estim sol eq homo bis}
\llbracket S \rrbracket^{\alpha,\beta+,\kappa} _{\sigma',\mu,\mathcal{D}'} \leq \frac{C}{\kappa(\sigma-\sigma')^{2n}} \llbracket f^T \rrbracket^{\alpha,\beta,\kappa} _{\sigma,\mu,\mathcal{D}}, 
\end{equation}
\begin{equation} \label{estim sol eq homo bis 2}
\llbracket R \rrbracket^{\alpha,\beta,\kappa} _{\sigma',\mu,\mathcal{D}'} \leq \frac{C e^{-(\sigma-\sigma')N/2}}{(\sigma-\sigma')^{n}} \llbracket f^T \rrbracket^{\alpha,\beta,\kappa} _{\sigma,\mu,\mathcal{D}}, 
\end{equation}
The two exponent $ \iota$ and $\upsilon$ are positive, the constant $\tilde C$ depends on $ |\omega_0|_{\mathcal{C}^1 \left( \mathcal{D} \right) }$, $c_0$ and $c_1$ while $C$ depends on $n$, $\beta$, $ |\omega_0|_{\mathcal{C}^1 \left( \mathcal{D} \right) }$ and $\vert  A_0 \vert_{\beta,{\mathcal{C}^1 \left( \mathcal{D} \right) }}$.
\end{theorem}
\begin{remark} 
By \eqref{estimation jet}, all previous estimates remain valid if we replace $f^T$ by $f$.
\end{remark}
\subsection{nonlinear homological equation}
In this section we will solve the nonlinear homological equation \eqref{non-linear-hom-eq} using the linear one. We start by stating the main result of this section
\begin{proposition}\label{proposition homo NL}
Let $0<\kappa < \frac{\delta}{4} \leq \frac{1}{8} \min (c_0, c_1)$ , $N>1$, $0<\sigma'<\sigma$, $\mu >0$ and $\omega : \mathcal{D} \to \mathbb{R}^n$ be $\mathcal{C}^1$ and satisfying $\vert \omega - \omega_0 \vert_{\mathcal{C}^1(\mathcal{D})} \leq \delta_0$. Let $A : \mathcal{D} \to \mathcal{NF} \cap \mathcal{M}_\beta$ be $\mathcal{C}^1$ and verifying $ \vert A - A_0  \vert_{\beta, \mathcal{C}^1(\mathcal{D}} \leq \delta_0$.
Assume that $\partial^j_\rho f \in \mathcal{T}^{\alpha,\beta}(\sigma,\mu,\mathcal{D})$ for $j=0,1$. Then there exists a closet subset  $\mathcal{D'} \equiv \mathcal{D'}(\kappa, N) \in \mathcal{D}$ such that
\begin{equation} \nonumber
\operatorname{mes} (\mathcal{D} \setminus \mathcal{D'}) \leq \tilde C d (\kappa \delta^{-1})^\iota N^\upsilon.
\end{equation}
For $\rho \in \mathcal{D'}$ there exist two real jet-functions $S=S^T$ and  $R=R^T$ such that $\partial_\rho^j S \in \mathcal{T}^{\alpha,\beta+}(\sigma',\mu,\mathcal{D'})$ and $\partial_\rho^j R \in \mathcal{T}^{\alpha,\beta+}(\sigma',\mu,\mathcal{D'})$  and a normal form $\hat{h}$ 
\begin{equation} \nonumber
\hat{h}= \hat{\omega} (\rho) \cdot r + \frac{1}{2} \langle \zeta , \hat K(\rho) \zeta \rangle,
\end{equation}
(up to a constant) such that
\begin{equation} \nonumber
\lbrace h, S \rbrace + \lbrace f-f^T ,S \rbrace^T + f^T = \hat{h} + R.
\end{equation}
Furthermore, for $\rho \in \mathcal{D'} $ we have:
\begin{equation} \label{estim new ham NL}
\llbracket \hat{h} \rrbracket^{\alpha,\beta,\kappa} _{\sigma',\mu',\mathcal{D}'} \leq C \left(1 + \frac{ \mu' }{\Pi \kappa \mu^3} \Xi + \frac{1}{\Pi \kappa^2 \mu^3 \mu'} \Xi^2 \right) \varepsilon,
\end{equation}
\begin{equation} \label{estim sol eq hom NL}
\llbracket R \rrbracket^{\alpha,\beta,\kappa} _{\sigma',\mu',\mathcal{D}'} \leq C \frac{\Delta}{\Pi}\left(1 + \frac{ \mu'}{ \kappa \mu^3}  \Xi  + \frac{1}{ \kappa^2 \mu^3 \mu'} \Xi^2 \right) \varepsilon,
\end{equation}
\begin{equation} \label{estim reste sol hom NL}
\llbracket S \rrbracket^{\alpha,\beta+,\kappa} _{\sigma',\mu',\mathcal{D}'} \leq \frac{C}{\Pi\kappa}\left(1 + \frac{ \mu' }{ \kappa \mu^3} \Xi  + \frac{1}{ \kappa^2 \mu^3 \mu'} \Xi^2 \right) \varepsilon,
\end{equation}
\begin{equation} \label{fréquence eq homo NL}
\vert \partial^j_\rho \hat{\omega} (\rho)  \vert \leq  C \left[  1+ \frac{\Xi}{\Pi \kappa (\mu^2-\mu'^2)} \left( 1 + \frac{\mu'}{\kappa \mu^3} \Xi + \frac{\mu'^2}{\kappa \mu^4} \Xi \right) \right] \frac{\varepsilon}{\mu^2}, \quad j=0,1,
\end{equation}
\begin{equation} \label{partie quad eq hom NL}
\vert \partial^j_\rho \hat K (\rho) \vert_\beta \leq C \left( 1 + \frac{\mu'}{\Pi \kappa \mu^3} \Xi  + \frac{\mu'^2}{\Pi \mu^6 \kappa^2} \Xi^2   \right) \frac{\varepsilon}{\mu^2}, \quad j=0,1,
\end{equation}
where we used the following notations
\begin{align*}
\begin{array}{ll}
\Delta  = e^{-( \sigma-\sigma')N/10}, \quad & \Pi = (\sigma-\sigma')^{6n+2},\\
\:\:\varepsilon  = \llbracket f^T \rrbracket^{\alpha,\beta,\kappa} _{\sigma,\mu,\mathcal{D}}, \quad & \Xi = \llbracket f \rrbracket^{\alpha,\beta,\kappa} _{\sigma,\mu,\mathcal{D}}.
\end{array}
\end{align*}
The two exponent $ \iota$ and $\upsilon$ are positive, the constant $\tilde C$ depends on $ |\omega_0|_{\mathcal{C}^1 \left( \mathcal{D} \right) }$, $c_0$ and $c_1$ while $C$ depends on $n$, $\beta$, $ |\omega_0|_{\mathcal{C}^1 \left( \mathcal{D} \right) }$ and $\vert  A_0 \vert_{\beta,{\mathcal{C}^1 \left( \mathcal{D} \right) }}$.
\end{proposition}
\begin{proof}
Consider $\sigma':=\sigma_5 <\sigma_4<\sigma_3<\sigma_2<\sigma_1<\sigma_0=:\sigma$ an arithmetic progression, and $\mu':=\mu_5<\mu_4<\mu_3<\mu_2<\mu_1<\mu_0=:\mu$ a geometric progression. We will construct two jet-functions $S$ and $R$ and a normal form $ \hat{h}$ verifying the following nonlinear homological equation
\begin{equation} \label{eq-hom-NL}
\lbrace h,S \rbrace + \lbrace f-f^T,S\rbrace^T + f^T = \hat{h} + R,
\end{equation}
We decompose $S$ as follows $=S_0+S_1+S_2$, where
\begin{equation} \label{decomp jet NL}
S_0(\theta):= S_\theta(\theta); \quad S_1(\theta,r):= S_r(\theta)r+\langle S_\zeta(\theta),\zeta \rangle;\quad S_2(\theta,\zeta):=\frac{1}{2} \langle S_{\zeta\zeta}(\theta)\zeta,\zeta \rangle.
\end{equation}
Similarly we decompose $h^+$ and $R$ in three parts:
\begin{equation} \nonumber
\hat{h}=\hat{h}^0+\hat{h}^1+\hat{h}^2, \quad R=R^0+R^1+R^2.
\end{equation}
We remark that $\lbrace f-f^T,S_3\rbrace^T =0$, so we can decompose \eqref{eq-hom-NL} in three equations
\begin{subequations}
\begin{equation} \label{eq-hom-NL1}
\lbrace h,S_0 \rbrace + f^T = \hat{h}^0+ R^0,
\end{equation}
\begin{equation} \label{eq-hom-NL2}
\lbrace h,S_1 \rbrace + f^T_1 = \hat{h}^1+ R^1 \mbox{ where } f_1= \lbrace f-f^T,S_0 \rbrace
\end{equation}
\begin{equation} \label{eq-hom-NL3}
\lbrace h,S_2 \rbrace + f^T_2 = \hat{h}^2+ R^2 \mbox{ where } f_2= \lbrace f-f^T,S_1 \rbrace
\end{equation}
\end{subequations} 
To solve \eqref{eq-hom-NL}, it suffices to solve successively the three previous homological equations. Moreover, each of these equations is of the same type as the linear homological equation solved in the previous section.

To ease notations, we define
\begin{equation} \nonumber
\Delta  = e^{-( \sigma-\sigma')N/10}, \quad  \Pi = (\sigma-\sigma')^{6n+2}, \quad \varepsilon  = \llbracket f^T \rrbracket^{\alpha,\beta,\kappa} _{\sigma,\mu,\mathcal{D}}, \quad  \Xi = \llbracket f \rrbracket^{\alpha,\beta,\kappa} _{\sigma,\mu,\mathcal{D}}.
\end{equation}
Thanks to Theorèm~\ref{theoreme homo}, we have
\begin{equation} \nonumber
\llbracket \hat{h}^0 \rrbracket^{\alpha,\beta,\kappa} _{\sigma_1,\mu,\mathcal{D}'}  \leq \varepsilon \; , \quad \llbracket R^0 \rrbracket^{\alpha,\beta,\kappa} _{\sigma_1,\mu,\mathcal{D}'} \leq \frac{C e^{-(\sigma-\sigma')N/2}}{(\sigma-\sigma')^{n}} \varepsilon, 
\end{equation}
\begin{equation} \label{S-0 eq hom NL}
\llbracket S_0 \rrbracket^{\alpha,\beta+,\kappa} _{\sigma_1,\mu,\mathcal{D}'} \leq \frac{C}{\kappa(\sigma-\sigma')^{2n}} \varepsilon.
\end{equation}
Let us now consider the equation \eqref{eq-hom-NL2}. By Lemma~\ref{estim crochet de poisson} and Lemma~\ref{estimation jet fonction},  $f_1$ satisfy
\begin{align*}
\llbracket f_1 \rrbracket^{\alpha,\beta,\kappa} _{\sigma_2,\mu_2,\mathcal{D}'}  & \leq \frac{C}{(\sigma-\sigma') \mu'^2} \llbracket f-f^T \rrbracket^{\alpha,\beta,\kappa} _{\sigma_1,\mu_1,\mathcal{D}'} \llbracket S_0 \rrbracket^{\alpha,\beta+,\kappa} _{\sigma_1,\mu_1,\mathcal{D}'}\leq \frac{C \mu'}{(\sigma-\sigma')^{2n+1} \kappa \mu^3} \varepsilon \Xi.
\end{align*}
Thus, thanks to the Theorem~\ref{theoreme homo}, we obtain:
\begin{align*} 
\llbracket \hat{h}^1 \rrbracket^{\alpha,\beta,\kappa} _{\sigma_2,\mu_2,\mathcal{D}'}  & \leq \llbracket f^T_1 \rrbracket^{\alpha,\beta,\kappa} _{\sigma_2,\mu_2,\mathcal{D}} \leq \frac{C \mu'}{(\sigma-\sigma')^{2n+1} \kappa \mu^3} \varepsilon \Xi.
\end{align*}
\begin{align*} \nonumber
 \llbracket R^1 \rrbracket^{\alpha,\beta,\kappa} _{\sigma_3,\mu_2,\mathcal{D}'} & \leq \frac{C\Delta}{(\sigma-\sigma')^{n}} \llbracket f^T_1 \rrbracket^{\alpha,\beta,\kappa} _{\sigma_2,\mu_2,\mathcal{D}} \leq \frac{C \Delta\mu'}{(\sigma-\sigma')^{3n+1}\kappa \mu^3}\varepsilon \Xi.
\end{align*}
\begin{equation} \label{S-1 eq homo NL}
\llbracket S_1 \rrbracket^{\alpha,\beta+,\kappa} _{\sigma_3,\mu_2,\mathcal{D}'}  \leq \frac{C}{(\sigma-\sigma')^{2n}\kappa} \llbracket f^T_1 \rrbracket^{\alpha,\beta,\kappa} _{\sigma_2,\mu_2,\mathcal{D}} \leq \frac{C \mu'}{(\sigma- \sigma')^{4n+1} \kappa^2 \mu^3} \varepsilon \Xi.
\end{equation}
Consider now the third equation \eqref{eq-hom-NL3}. By Lemma~\ref{estim crochet de poisson} and Lemma~\ref{estimation jet fonction},  $f_2$ satisfy
\begin{align*}
\llbracket f_2 \rrbracket^{\alpha,\beta,\kappa} _{\sigma_4,\mu_4,\mathcal{D}'}  & \leq \frac{C}{(\sigma-\sigma') \mu'^2} \llbracket f-f^T \rrbracket^{\alpha,\beta,\kappa} _{\sigma_1,\mu_1,\mathcal{D}'} \llbracket S_1 \rrbracket^{\alpha,\beta+,\kappa} _{\sigma_3,\mu_2,\mathcal{D}'} \leq \frac{C}{(\sigma-\sigma')^{4n+2} \kappa^2 \mu^3 \mu'} \varepsilon \Xi^2.
\end{align*}
So, by Theorem~\ref{theoreme homo}, we have
\begin{align*} 
\llbracket \hat{h}^1 \rrbracket^{\alpha,\beta,\kappa} _{\sigma_2,\mu_2,\mathcal{D}'}  & \leq \llbracket f^T_2 \rrbracket^{\alpha,\beta,\kappa} _{\sigma_4,\mu_4,\mathcal{D}} \leq \frac{C}{(\sigma-\sigma')^{4n+2} \kappa^2 \mu^3 \mu'} \varepsilon \Xi^2.
\end{align*}
\begin{align*} \nonumber
 \llbracket R^2 \rrbracket^{\alpha,\beta,\kappa} _{\sigma_5,\mu_4,\mathcal{D}'}  \leq \frac{C \Delta}{(\sigma-\sigma')^{n}} \llbracket f^T_2 \rrbracket^{\alpha,\beta,\kappa} _{\sigma_4,\mu_4,\mathcal{D}} \leq \frac{C \Delta}{(\sigma-\sigma')^{5n+2}\kappa^2 \mu^3\mu'}\varepsilon \Xi^2.
\end{align*}
\begin{align*} \nonumber
\llbracket S_2 \rrbracket^{\alpha,\beta+,\kappa} _{\sigma_5,\mu_4,\mathcal{D}'}  \leq \frac{C}{(\sigma-\sigma')^{2n} \kappa} \llbracket f^T_2 \rrbracket^{\alpha,\beta,\kappa} _{\sigma_4,\mu_4,\mathcal{D}} \leq \frac{C}{(\sigma- \sigma')^{6n+2} \kappa^2 \mu^3 \mu'} \varepsilon \Xi^2.
\end{align*}
Thus the unknowns of the nonlinear homological equation \eqref{eq-hom-NL} satisfy
\begin{equation} \nonumber
\llbracket \hat{h} \rrbracket^{\alpha,\beta,\kappa} _{\sigma',\mu',\mathcal{D}'} \leq C \left(1 + \frac{ \mu' \Xi}{(\sigma-\sigma')^{2n+1} \kappa \mu^3}  + \frac{\Xi^2}{(\sigma-\sigma')^{4n+2} \kappa^2 \mu^3 \mu'} \right) \varepsilon.
\end{equation}
\begin{equation} \nonumber
\llbracket R \rrbracket^{\alpha,\beta,\kappa} _{\sigma',\mu',\mathcal{D}'} \leq \frac{C \Delta}{(\sigma-\sigma')^{5n+2}}\left(1 + \frac{ \mu' \Xi}{ \kappa \mu^3}  + \frac{\Xi^2}{ \kappa^2 \mu^3 \mu'} \right) \varepsilon.
\end{equation}
\begin{equation} \nonumber
\llbracket S \rrbracket^{\alpha,\beta+,\kappa} _{\sigma',\mu',\mathcal{D}'} \leq \frac{C}{(\sigma-\sigma')^{6n+2}\kappa}\left(1 + \frac{ \mu' \Xi}{ \kappa \mu^3}  + \frac{\Xi^2}{ \kappa^2 \mu^3 \mu'} \right) \varepsilon.
\end{equation}
It remains to prove the estimates \eqref{fréquence eq homo NL}-\eqref{partie quad eq hom NL}. Let us start with the new frequency $\hat{\omega}$. Thanks to Theorem~\ref{theoreme homo} the new new frequency $\hat{\omega}$ is given by
\begin{equation} \nonumber
\hat\omega= \int_{\mathbb{T}^n} \nabla_r f(\theta,0,0,\rho)d\theta + \int_{\mathbb{T}^n} \nabla_r f_1(\theta,0,0,\rho)d\theta +\int_{\mathbb{T}^n} \nabla_r f_2 (\theta,0,0,\rho)d\theta.
\end{equation}
A simple computation gives
\begin{equation} \nonumber
\hat\omega= \hat{\omega}_1 + \hat{\omega}_2 + \hat{\omega}_3 + \hat{\omega}_4,
\end{equation}
where
\begin{align*}
\hat{\omega}_1 & := \int_{\mathbb{T}^n}   \nabla_r f(\theta,0,0,\rho)d\theta & \hat{\omega}_2 & :=  \int_{\mathbb{T}^n} \nabla_{rr} f(\theta,0,0,\rho) \nabla_\theta S_0 (\theta,0,0,\rho) d\theta \\
\hat{\omega}_3& :=\int_{\mathbb{T}^n} \nabla_{rr} f(\theta,0,0,\rho) \nabla_\theta S_1 (\theta,0,0,\rho) d\theta & \hat{\omega}_4 & := \int_{\mathbb{T}^n} \nabla_{r} \langle \nabla_\zeta(f-f^T),J \nabla_\zeta S_1 \rangle(\theta,0,0,\rho) d\theta 
\end{align*}
For $\hat{\omega}_1$ we have:
\begin{equation} \nonumber
\vert \hat{\omega}_1 \vert \leq \frac{\varepsilon}{\mu^2}.
\end{equation}
For  $\hat{\omega}_2$, by Cauchy estimate and the estimate \eqref{S-0 eq hom NL} , we have
\begin{align*}
\vert \hat{\omega}_2 \vert& \leq \Vert \nabla^2_\zeta f(\theta,0,0,\rho) \Vert_{\mathcal{L}(\mathbb{R}^n, \mathbb{R}^n)} \Vert  \nabla_{\theta}  S_0 (\theta,0,0,\rho) \Vert_{\mathbb{R}^n}\\
& \leq \frac{1}{\mu^2-\mu'^2}\sup_{x  \in  \mathcal{O}_\mu(\mathbb{R}^n)} \left(\Vert \nabla_r f(x) \Vert_{\mathbb{R}^n} \right) \Vert \nabla_{\theta}  S_0 (\theta,0,0,\rho) \Vert_{\mathbb{R}^n} \\
& \leq \frac{C}{(\sigma-\sigma')^{2n} \kappa \mu^2 (\mu^2-\mu'^2) } \varepsilon \Xi.
\end{align*}
Similarly, for $\hat{\omega}_3$, by Cauchy estimate and \eqref{S-1 eq homo NL}, we have
\begin{equation} \nonumber
\vert \hat{\omega}_3 \vert \leq \frac{C \mu' }{( \sigma - \sigma')^{4n+1} \kappa^2 \mu^5 ( \mu^2 - \mu'^2)} \varepsilon  \Xi^2.
\end{equation}
Finally, by Cauchy estimate, the inequality \eqref{estimation difference jet} from Lemma~\ref{estimation jet fonction} and the estimate \eqref{S-1 eq homo NL}, $\hat{\omega}_4$ verify
\begin{equation} \nonumber
\vert \hat{\omega}_4 \vert \leq \frac{C \mu'^2}{(\sigma-\sigma')^{4n+1} \kappa^2 \mu^6 (\mu^2- \mu'^2)}  \varepsilon  \Xi^2.
\end{equation}
Similarly we prove that $\partial_\rho \hat{\omega}$ satisfy the same estimate as  $\hat{\omega}$ for $\rho \in \mathcal{D}'$.

Let us now prove the estimate \eqref{partie quad eq hom NL}. We can decompose $K$ as follows:
\begin{equation} \nonumber
\hat K=\hat K_1+\hat K_2+\hat K_3,
\end{equation}
where $K_1$ comes from the resolution of equation \eqref{eq-hom-NL1}, $K_2$ from \eqref{eq-hom-NL2}, and $K_3$ from \eqref{eq-hom-NL3}. By estimate \eqref{estima B} from Theorem~\ref{theoreme homo} and for  $\rho \in \mathcal{D}'$ we have:
\begin{equation} \nonumber
\vert \partial_\rho^j \hat K(\rho) \vert_\beta  \leq \frac{\llbracket \partial_\rho^j f^T \rrbracket^{\alpha,\beta} _{\sigma,\mu,\mathcal{D}}}{\mu^{2}}, \quad \vert \partial_\rho^j \hat K_i(\rho) \vert_\beta  \leq \frac{\llbracket \partial_\rho^j f^T_i \rrbracket^{\alpha,\beta} _{\sigma,\mu,\mathcal{D}}}{\mu^{2}},
\end{equation}
for  $i \in \lbrace 1,2 \rbrace$ and $j \in \lbrace 0,1 \rbrace$. We conclude by using the estimates obtained for $f_1$ and $f_2$.
\end{proof}
\section{Proof of the KAM theorem.}
In this section we proof the Theorem~\ref{theoreme kam}. As mentioned in the previous section, the theorem will be proved with an iterative procedure. We start by describing the general step.
\subsection{Elementary step.}
Let $h$ be a Hamiltonian normal form
\begin{equation} \nonumber
h(\rho)= \omega(\rho).r +\frac{1}{2} \langle \zeta , A(\rho) \zeta  \rangle,
\end{equation}
where $A=D+N$, with $D$ given by \eqref{mat diag} and $ N \in \mathcal{NF} \cap \mathcal{M}_\beta$ for $\beta >0$. We assume that the internal frequency $\omega$ and the matrix $D$ verify the hypotheses A1, A2 and A3 and $N$ satisfies the hypothesis B. We consider a small perturbation $f$ (let us say $f= \mathcal{O}( \varepsilon)$) and satisfies $\partial_\rho^j f \in \mathcal{T}^{\alpha,\beta}(\mathcal{D},\sigma,\mu)$ for $j=0,1$. We search a jet-function $S$ such that its time one flow $\Phi_S^1$ is near the identity and satisfies:
\begin{equation} \nonumber
(h+f)\circ \Phi_S = h_1 +f_1, \quad \mbox{ whith } \: (f_1)^T\simeq\mathcal{O}(\varepsilon^\gamma), \quad \gamma>1
\end{equation}
and $ h_1 $ is the new normal form close to $ h $ ( i.e. $ \vert h_1 - h \vert \simeq \mathcal{O}( \varepsilon ) $). The Hamiltonian $h_1$ will be in the following form:
\begin{equation}\nonumber
h_1=h+\hat{h}.
\end{equation}
According to the previous section, the jet-function $S$ verifies the following nonlinear homological equation:
\begin{equation} \nonumber
\left\lbrace  h , S \right\rbrace + \left\lbrace  f-f^T , S \right\rbrace + f^T  = \hat{h} + R.
\end{equation} 
Using Taylor expansion and the Hamiltonian structure, we obtain
\begin{align*}
(h+f) \circ \Phi_S^1 & = h+ f + \left\lbrace  h + f, S \right\rbrace  + \int_0^1 (1-t) \left\lbrace  \left\lbrace  h+f, S \right\rbrace  , S \right\rbrace  \circ \Phi_S^t dt \\
&=h+ \hat{h} + \lbrace h,S \rbrace - \hat{h} + f + \lbrace f , S \rbrace + \int_0^1 (1-t) \left\lbrace  \lbrace  h+f, S \rbrace  , S \right\rbrace  \circ \Phi_S^t dt\\
& = h_1 + f_1,
\end{align*}
where 
\begin{equation} \label{perturbation-first-step}
f_1=  R+(f-f^T) + \lbrace f,S \rbrace - \lbrace  f-f^T , S \rbrace^T + \int_0^1 (1-t) \left\lbrace  \left\lbrace  h+f, S \right\rbrace  , S \right\rbrace  \circ \Phi_S^t dt.
\end{equation}
\begin{remark} \label{utilite eq hom NL}
In several proof of KAM theorems, the authors solve the following linear homological equation instead of the nonlinear one
\begin{equation} \nonumber
\lbrace h, S \rbrace + f^T = \hat{h} + R.
\end{equation}
In this case, the new perturbation term is given by 
\begin{equation} \nonumber
f_1= R + (f-f^T)\circ  \Phi_S^1 + \int_0^1 \lbrace (1-t)(\hat{h} + R + tf^T, S \rbrace \circ  \Phi_S^t dt.
\end{equation}
\ref{freq_kam}
Dans les itérations du théorème KAM le terme $(f-f^T)\circ  \Phi_S^1$ est très défavorable. Pour avoir un contrôle sur le jet de la perturbation à chaque étape on est amené à estimer le terme $ \left((f-f^T)\circ  \Phi_S^1 \right)^T$. Ce terme est difficile à contrôler. Pour cela on résout l'équation homologique non linéaire. De plus on remarque que 
\begin{equation} \label{jet pertub NL}
f_1^T= R + \lbrace f^T , S \rbrace^T + \left( \int_0^1 (1-t) \left\lbrace  \left\lbrace  h+f, S \right\rbrace  , S \right\rbrace  \circ \Phi_S^t dt \right)^T.
\end{equation}
\end{remark}
In the following lemma we give an upper bound of the norm of the new perturbation $f_1$ and its jet $f^T_1$.

\begin{lemma} \label{esti-f+}
Let $0 < \kappa < \frac{\delta}{4} \leq \frac{1}{8} \min (c_0, c_1)$, $N\geq 1$, $0<\sigma'< \sigma\leq 1$ and $\displaystyle 0 < \mu' < \frac{\mu}{2}$. Assume that the perturbation verifies $\partial_\rho^j f \in \mathcal{T}^{\alpha,\beta}(\mathcal{D},\sigma,\mu)$ for $j=0,1$. For $\displaystyle \rho \in \mathcal{D'} \subset \mathcal{D}$, we assume that  $R$ satisfies \eqref{estim sol eq hom NL}, that the jet function $S$ verifies $\displaystyle  \partial_\rho^j S \in \mathcal{T}^{\alpha,\beta+}(\mathcal{D}',\sigma',\mu)$ and satisfies also the estimate \eqref{estim reste sol hom NL}. If
\begin{align}\label{hypothese sur S}
\llbracket   f^T \rrbracket^{\alpha,\beta,\kappa} _{\sigma,\mu,\mathcal{D}} \leq \frac{\kappa^3(\sigma-\sigma')^{6n+3} \mu^5 \mu'}{\mu^3 \mu' \kappa^2 + \kappa \mu' \Xi +\Xi^2} ,
\end{align}
then $\displaystyle \partial_\rho^j f_1 \in \mathcal{T}^{\alpha,\beta}(\mathcal{D}',\sigma',\mu')$. Furthermore, we have the following estimates
\begin{equation} \label{nouv-pertur}
\begin{split}
\llbracket f_1 \rrbracket^{\alpha,\beta,\kappa} _{\sigma',\mu',\mathcal{D}'} \leq & C \left( \frac{\Delta}{\Pi} + \frac{\Xi}{\Pi \kappa \mu^2} + \frac{\mu' \Xi}{\Pi \kappa \mu^3 } + \frac{(\Delta+1) X \varepsilon}{\Pi^2 \kappa \mu'^2} + \frac{X \varepsilon \Xi}{\Pi \kappa \mu'^4} \right) X \varepsilon  + \left( \frac{\mu'}{\mu} \right)^3 \Xi,
\end{split}
\end{equation}

\begin{equation} \label{nouv-pertur-jet}
\llbracket f_1^T \rrbracket^{\alpha,\beta,\kappa} _{\sigma',\mu',\mathcal{D}'} \leq C \left( \frac{\Delta}{\Pi} + \frac{\varepsilon}{\Pi \kappa \mu'^2 } + \frac{(\Delta+1) X \varepsilon}{\Pi^2 \kappa \mu'^2} + \frac{X \varepsilon \Xi}{\Pi \kappa \mu'^4} \right) X \varepsilon,
\end{equation}
where 
\begin{align*}
X & = 1 + \frac{ \mu' \llbracket f \rrbracket^{\alpha,\beta,\kappa} _{\sigma,\mu,\mathcal{D}}}{ \kappa \mu^3}  + \frac{(\llbracket f \rrbracket^{\alpha,\beta,\kappa} _{\sigma,\mu,\mathcal{D}})^2}{ \kappa^2 \mu^3 \mu'}, \\
\Delta & = e^{-( \sigma-\sigma')N/10}, \quad \Pi = (\sigma-\sigma')^{6(2n+1)},\\
\varepsilon & = \llbracket f^T \rrbracket^{\alpha,\beta,\kappa} _{\sigma,\mu,\mathcal{D}}, \quad \Xi = \llbracket f \rrbracket^{\alpha,\beta,\kappa} _{\sigma,\mu,\mathcal{D}}.
\end{align*}

The constant $C>0$ depends on $n$, $\sigma$, $\beta$, $ |\omega_0|_{\mathcal{C}^1 \left( \mathcal{D} \right) }$ and $\vert  A_0 \vert_{\beta,{\mathcal{C}^1 \left( \mathcal{D} \right) }}$.
\end{lemma}
\begin{proof}
Consider $\sigma':=\sigma_3<\sigma_2<\sigma_1<\sigma_0=:\sigma$ an arithmetic progression, and $\mu':=\mu_3<\mu_2<\mu_1<\mu_0=:\mu$ a geometric progression. We recall that the new perturbation is given by
\begin{equation} \nonumber
f_1= R+(f-f^T) + \lbrace f,S \rbrace - \lbrace  f-f^T , S \rbrace^T + \int_0^1 (1-t) \left\lbrace  \left\lbrace  h+f, S \right\rbrace  , S \right\rbrace  \circ \Phi_S^t dt.
\end{equation}
We decompose $f_1$ as follows:
\begin{equation} \nonumber
f_1^1= R, \quad f_1^2 = f-f^T, \quad f_1^3 = \lbrace f,S \rbrace , \quad f_1^4= \left\lbrace  f-f^T , S \right\rbrace , \quad f_1^5= \int_0^1 (1-t) \left\lbrace  \left\lbrace  h+f, S \right\rbrace  , S \right\rbrace  \circ \Phi_S^t dt.
\end{equation}
We denote $ \displaystyle \varepsilon := \llbracket f^T \rrbracket^{\alpha,\beta,\kappa} _{\sigma,\mu\mathcal{D}}$ and $ \displaystyle \Xi = \llbracket f \rrbracket^{\alpha,\beta,\kappa} _{\sigma,\mu,\mathcal{D}}$. We will give an upper bound of the norm of each term of $f_1$.
\begin{itemize}
\item Let us start with $\llbracket f_1^1 \rrbracket^{\alpha,\beta,\kappa} _{\sigma',\mu',\mathcal{D}'}$. By Proposition~\ref{proposition homo NL} and estimate  \eqref{estim sol eq hom NL}, we have
\begin{equation} \nonumber
\llbracket  f_1^1 \rrbracket^{\alpha,\beta,\kappa} _{\sigma',\mu',\mathcal{D}'} \leq \frac{C \Delta}{(\sigma-\sigma')^{5n+2}}\left(1 + \frac{ \mu' \Xi}{ \kappa \mu^3}  + \frac{\Xi^2}{ \kappa^2 \mu^3 \mu'} \right) \varepsilon.
\end{equation}
The constant $C$ depends on $n$, $\beta$, $ |\omega_0|_{\mathcal{C}^1 \left( \mathcal{D} \right) }$ and $\vert  A_0 \vert_{\beta,{\mathcal{C}^1 \left( \mathcal{D} \right) }}$.
\item For the second term $\llbracket f_1^2 \rrbracket^{\alpha,\beta,\kappa} _{\sigma',\mu',\mathcal{D}'}$, using estimate \eqref{estimation difference jet} from Lemma~\ref{estimation jet fonction}, we have
\begin{equation} \nonumber
\llbracket f_2^2 \rrbracket^{\alpha,\beta,\kappa} _{\sigma',\mu',\mathcal{D}'} \leq 2 \left(  \frac{\mu'}{\mu}  \right)^3 \Xi.
\end{equation} 
\item  For the third term $\llbracket f_1^3 \rrbracket^{\alpha,\beta,\kappa} _{\sigma',\mu',\mathcal{D}'}$, by Lemma~\ref{estim crochet de poisson} we have
\begin{align*}
\llbracket f_1^3 \rrbracket^{\alpha,\beta,\kappa} _{\sigma',\mu',\mathcal{D}'} & \leq \frac{C}{(\sigma- \sigma')\mu^2} \Xi \llbracket S \rrbracket^{\alpha,\beta+,\kappa} _{\sigma',\mu',\mathcal{D}'}  \\
& \leq \frac{C}{(\sigma-\sigma')^{6n+3}\kappa \mu^2} \left(1 + \frac{ \mu' \Xi}{ \kappa \mu^3}  + \frac{\Xi^2}{ \kappa^2 \mu^3 \mu'} \right) \varepsilon \Xi . 
\end{align*}
\item Consider now $ \displaystyle \llbracket f_1^4 \rrbracket^{\alpha,\beta,\kappa} _{\sigma',\mu',\mathcal{D}'}$. Using estimate \eqref{estimation difference jet} from Lemma~\ref{estimation jet fonction}, we obtain that
\begin{equation} \nonumber
\llbracket f_1^4 \rrbracket^{\alpha,\beta,\kappa} _{\sigma',\mu',\mathcal{D}'} \leq 3 \llbracket \lbrace f-f^T , S \rbrace \rrbracket^{\alpha,\beta,\kappa} _{\sigma',\mu',\mathcal{D}'}. 
\end{equation}
Then by Lemma~\ref{estim crochet de poisson} and estimate \eqref{estimation difference jet} from Lemma~\ref{estimation jet fonction}, we get
\begin{align*}
\llbracket f_1^4 \rrbracket^{\alpha,\beta,\kappa} _{\sigma',\mu',\mathcal{D}'} & \leq \frac{C}{(\sigma-\sigma') \mu'^2} \llbracket f-f^T \rrbracket^{\alpha,\beta,\kappa} _{\sigma_1,\mu_1,\mathcal{D}'}  \llbracket S \rrbracket^{\alpha,\beta+,\kappa}_{\sigma_1,\mu_1,\mathcal{D}'} \\
& \leq  \frac{C \mu'}{(\sigma-\sigma')^{6n+3} \kappa \mu^3}  \left(1 + \frac{ \mu' \Xi}{ \kappa \mu^3}  + \frac{\Xi^2}{ \kappa^2 \mu^3 \mu'} \right)  \varepsilon \Xi .
\end{align*}
\item It remains to study the term $f_1^5$. We will decompose it in two term  $f_1^5=f_1^6+f_1^7$, where
\begin{equation} \nonumber
f_1^6=\int_0^1 (1-t) \lbrace \hat h +R-f^T,S  \rbrace \circ \Phi_S^t dt,
\end{equation}
\begin{equation} \nonumber
f_1^7=\int_0^1 (1-t) \lbrace \lbrace f,S \rbrace - \lbrace f-f^T,S \rbrace^T , S  \rbrace \circ \Phi_S^t dt.
\end{equation}
For $f_1^6$, by Proposition~\ref{proposition homo NL}, we have
\begin{equation} \nonumber
\begin{split}
\llbracket \hat{h}+R-f^T \rrbracket^{\alpha,\beta,\kappa} _{\sigma_2,\mu_2,\mathcal{D}'}  \leq C & \left[ 2 + \frac{\mu'}{(\sigma- \sigma')^{2n+1}\kappa \mu^3} \Xi \right.  + \frac{1}{(\sigma- \sigma')^{4n+2}\kappa \mu^3 \mu'}  \Xi^2 \\
& + \left. \frac{ \Delta}{(\sigma-\sigma')^{5n+2}} \left( 1 + \frac{\Xi}{ \kappa \mu^3}  + \frac{\Xi^2}{ \kappa^2 \mu^3 \mu'} \right) \right] \varepsilon.
\end{split}
\end{equation}
Using Proposition~\ref{composition} with the assumption \eqref{hypothese sur S} and the Lemma~\ref{estim crochet de poisson}, we obtain
\begin{equation} \nonumber
\begin{split}
\llbracket f_1^6 \rrbracket^{\alpha,\beta,\kappa} _{\sigma',\mu',\mathcal{D}'} \leq  & \frac{C}{(\sigma-\sigma')^{6n+3}\kappa\mu'^2} \left[ 2 + \frac{\mu'}{(\sigma- \sigma')^{2n+1}\kappa \mu^3} \Xi \right. + \frac{1}{(\sigma- \sigma')^{4n+2}\kappa \mu^3 \mu'} \Xi^2 \\
& + \left. \frac{ \Delta}{(\sigma-\sigma')^{5n+2}}\left(1 + \frac{ \mu' \Xi}{ \kappa \mu^3}  + \frac{\Xi^2}{ \kappa^2 \mu^3 \mu'} \right) \right] \times \left(1 + \frac{ \mu' \Xi}{ \kappa \mu^3}  + \frac{\Xi^2}{ \kappa^2 \mu^3 \mu'} \right)  \Xi^2.
\end{split}
\end{equation}
Similarly we prove that
\begin{equation} \nonumber
\begin{split}
\llbracket f_1^7 \rrbracket^{\alpha,\beta,\kappa} _{\sigma',\mu',\mathcal{D}'} \leq & \frac{C}{(\sigma- \sigma')^{6n+4} \kappa  \mu'^4} \left(1 + \frac{ \mu' \Xi}{ \kappa \mu^3}  + \frac{\Xi^2}{ \kappa^2 \mu^3 \mu'} \right)^2 \varepsilon^2 \Xi.
\end{split}
\end{equation}
To ease notations, we define
\begin{equation} \nonumber
X  = 1 + \frac{ \mu' \llbracket f \rrbracket^{\alpha,\beta,\kappa} _{\sigma,\mu,\mathcal{D}}}{ \kappa \mu^3}  + \frac{(\llbracket f \rrbracket^{\alpha,\beta,\kappa} _{\sigma,\mu,\mathcal{D}})^2}{ \kappa^2 \mu^3 \mu'}, \quad \Delta  = e^{-( \sigma-\sigma')N/10}, \quad \Pi = (\sigma-\sigma')^{6(2n+1)},
\end{equation}
\begin{align*}
X & = 1 + \frac{ \mu' \llbracket f \rrbracket^{\alpha,\beta,\kappa} _{\sigma,\mu,\mathcal{D}}}{ \kappa \mu^3}  + \frac{(\llbracket f \rrbracket^{\alpha,\beta,\kappa} _{\sigma,\mu,\mathcal{D}})^2}{ \kappa^2 \mu^3 \mu'}, \\
\Delta & = e^{-( \sigma-\sigma')N/10}, \quad \Pi = (\sigma-\sigma')^{6(2n+1)},
\end{align*}
So we get
\begin{align*} 
\llbracket f_1 \rrbracket^{\alpha,\beta,\kappa} _{\sigma',\mu',\mathcal{D}'} \leq & C \left( \frac{\Delta}{\Pi} + \frac{\Xi}{\kappa \mu^2 \Pi} + \frac{\mu' \Xi}{\Pi \mu^3 \kappa} + \frac{(\Delta+1) X \varepsilon}{\kappa \Pi^2\mu'^2} + \frac{X \varepsilon \Xi}{\kappa\Pi \mu'^4} \right) X \varepsilon \\& + \left( \frac{\mu'}{\mu} \right)^3 \Xi.
\end{align*}
Recall that the jet of the new perturbation is given by
\begin{equation} \nonumber
f_1^T= R + \lbrace f^T , S \rbrace^T + \left( \int_0^1 (1-t) \left\lbrace  \left\lbrace  h+f, S \right\rbrace  , S \right\rbrace  \circ \Phi_S^t dt \right)^T.
\end{equation}
With the same argument used to obtain the upper bound for $ \displaystyle \llbracket f_1 \rrbracket^{\alpha,\beta,\kappa} _{\sigma',\mu',\mathcal{D}'}$, we prove that 
\begin{equation} \nonumber
\llbracket f_1^T \rrbracket^{\alpha,\beta,\kappa} _{\sigma',\mu',\mathcal{D}'} \leq C \left( \frac{\Delta}{\Pi} + \frac{\varepsilon}{\kappa \mu'^2 \Pi} + \frac{(\Delta+1) X \varepsilon}{\kappa \Pi^2\mu'^2} + \frac{X \varepsilon \Xi}{\kappa\Pi \mu'^4} \right) X \varepsilon.
\end{equation}

\end{itemize}
\end{proof}
\subsection{Initialization of KAM procedure}
In this section we will describe the first KAM step. 

To prove the KAM Theorem~\ref{theoreme kam}, we will construct an analytic symplectic change of  variables as follows
\begin{equation} \nonumber
\Phi : \mathcal{O}^\alpha(\frac{\sigma}{2}, \frac{\mu}{2}) \rightarrow \mathcal{O}^\alpha(\sigma,\mu),
\end{equation}
satisfying:
\begin{equation} \nonumber
(h + f ) \circ \Phi = \tilde{h} + \tilde{f},
\end{equation}
where $\tilde{h}$ is a normal form close to $h$ and $\tilde{f}$ is the new perturbation with a zero-jet. We will construct $ \Phi$ iteratively
\begin{equation} \nonumber
\Phi = \lim_{k \rightarrow \infty} \Phi_1 \circ \Phi_2 \ldots \circ \Phi_k.
\end{equation}
Each diffeomorphism $\Phi_k$ will be the time-one flow of a Hamiltonian $S_k$. Each $S_k$ will be a jet function that satisfies the nonlinear homological equation at step $k$. Let us focus on the first KAM step. Assume that $ f^T=O(\varepsilon)$ for $\varepsilon$ small. The first change variable $\Phi_1$ satisfy
\begin{equation} \nonumber
(h+f)\circ \Phi_1 = h_1 +f_1,
\end{equation}
and $f_1^T = O(\varepsilon^\gamma)$ for $\gamma > 1$. From Proposition~\ref{proposition homo NL}, for $\rho \in \mathcal{D}' \subset \mathcal{D}$, the new normal form is given by $h_1 = h + \hat{h}$, where
\begin{equation} \nonumber
\hat{h}(\rho) = \hat\omega(\rho) \cdot r +\frac{1}{2} \langle \zeta,\hat K(\rho)\zeta\rangle.
\end{equation}
The new frequency $\hat \omega$ satisfies \eqref{fréquence eq homo NL} and the matrix $\hat{K}$ satisfies \eqref{partie quad eq hom NL}. After the first step, the new normal form $h_1$ is given by
\begin{align*}
h_1( \rho) & = \left(  \omega(\rho) + \hat \omega(\rho) \right) \cdot r  +\frac{1}{2} \langle \zeta, ( A(\rho)+\hat K(\rho )) \zeta\rangle\\
& = \omega_1(\rho) \cdot r +\frac{1}{2} \langle \zeta,A_1(\rho)\zeta\rangle.
\end{align*}
The new perturbation is given by \eqref{perturbation-first-step} and its jet is given by \eqref{jet pertub NL}. To prove that $f_1^T = O(\varepsilon^\gamma)$ for $\gamma > 1$, we make the following choice of parameters
\begin{align*}
& \varepsilon  = \lc f^T \rc^{\alpha,\beta,\kappa}_{\sigma,\mu,\mathcal{D}},  \quad \varepsilon_1 = \lc f_1^T \rc^{\alpha,\beta,\kappa}_{\sigma',\mu',\mathcal{D}'} ,\\
& \Xi  = \lc f \rc^{\alpha,\beta,\kappa} _{\sigma,\mu,\mathcal{D}}= O(\varepsilon^\tau) \quad \tau \in  [\frac{1}{2} , 1], \\
& \Xi_1  = \lc f_1 \rc^{\alpha,\beta,\kappa} _{\sigma,\mu,\mathcal{D}} \\
& \sigma_1  = \frac{3}{4} \sigma,\\
& \mu_1= \frac{3}{4} \mu, \\
& N  = 10 (\sigma-\sigma_1)^{-1} \ln(\varepsilon^{-1}), \\
& \kappa  = \varepsilon^{1/20}  .
\end{align*}
\begin{lemma} \label{premiere etape}
There exists a closed set $\mathcal{D}' \subset \mathcal{D}$ and $\gamma>0$ such that
\begin{equation} \nonumber
\operatorname{mes} \left( \mathcal{D} \setminus \mathcal{D}' \right) \leq \varepsilon^\gamma.
\end{equation}
For $ \rho \in \mathcal{D}'$,  there exists a real analytic symplectomorphism
\begin{equation} \nonumber
\Phi : \mathcal{O}^\alpha(\sigma', \mu') \to \mathcal{O}^\alpha(\sigma,\mu),
\end{equation}
such that
\begin{equation} \nonumber
(h+f) \circ \Phi_1 = h_1 + f_1.
\end{equation}
For $ \rho \in \mathcal{D}'$ and $j=0,1$ we have:
\begin{equation} \label{hyp B}
\vert  \partial^j_\rho \left( A_{1}(\rho) - A(\rho) \right)  \vert_\beta = \vert \partial^j_\rho \hat K(\rho) \vert_\beta  \leq C \varepsilon,
\end{equation}
\begin{equation}  \label{hyp freq}
\vert \partial^j_\rho \left( \omega_{1}(\rho) - \omega(\rho) \right)   \vert  \leq  C \varepsilon,
\end{equation}
\begin{equation} \label{proche id1}
\Vert \Phi(x,\rho) - x \Vert_\alpha \leq  \varepsilon^{9/10} \mbox{ for } x \in  \mathcal{O}^\alpha(\sigma,\mu).
\end{equation}
The constant $C$ depends on $n$, $\beta$, $\sigma$, $\mu$, $ |\omega_0|_{\mathcal{C}^1 \left( \mathcal{D} \right) }$ and $\vert  A_0 \vert_{\beta,{\mathcal{C}^1 \left( \mathcal{D} \right) }}$.

Moreover the new jet satisfy
\begin{equation} \label{prem Xi}
\Xi_1 \leq \varepsilon^{7/5} + \Xi,
\end{equation}
and its jet verifies
\begin{equation} \label{prem eps}
\varepsilon_1 \leq \varepsilon^{8/5}.
\end{equation}
\end{lemma}
\begin{proof}
The existence of the the closed set $\mathcal{D}'$ and the map $ \Phi$ is given by the Proposition~\ref{proposition homo NL}. According to the choice of parameters, the Lebesgue measure of the closed set satisfy
\begin{equation} \nonumber
\operatorname{mes} \mathcal{D} \setminus \mathcal{D}' \leq Cd (\kappa \delta^{-1} )^\iota N^\upsilon \leq \varepsilon^\gamma.
\end{equation}
For $\rho \in \mathcal{D}'$, by estimate \eqref{partie quad eq hom NL} and the parameters choice, we have
\begin{equation} \nonumber
\begin{split}
\vert \partial^j_\rho ( A_{1}(\rho) - A(\rho)) \vert_\beta \leq \vert \partial_\rho^j \hat K (\rho) \vert_\beta \leq C \varepsilon, \quad j=0,1, 
\end{split}
\end{equation}
for $ \varepsilon$ sufficiently small. For the frequency, according to estimate \eqref{fréquence eq homo NL} and the parameters choice, we have
\begin{equation}\nonumber
\begin{split}
\vert \partial^j_\rho \left( \omega_{1} (\rho) - \omega(\rho) \right)  \vert = & \vert \partial^j_\rho \hat{\omega} (\rho) \vert \leq C \varepsilon, \quad j=0,1,
\end{split}
\end{equation}
for $ \varepsilon$ sufficiently small.
For $x=(r,\theta,\zeta) \in \mathcal{O}^\alpha(\sigma,\mu)$ we recall that $\Vert (r,\theta,\zeta) \Vert_\alpha=\max(|r|,|\theta|,\Vert \zeta \Vert_\alpha)$. By Proposition~\ref{estim-phi}, estimate \eqref{estim reste sol hom NL} and the parameters choice, we have
\begin{equation} \nonumber
\begin{split}
\Vert \Phi(x) - x \Vert_\alpha & \leq \frac{\llbracket S \rrbracket^{\alpha,\beta+,\kappa} _{\sigma',\mu,\mathcal{D}_1}}{(\sigma-\sigma')\mu^2} \leq \frac{C}{\Pi(\sigma - \sigma_1) \mu^2 \kappa}\left(1 + \frac{ \mu' }{ \kappa \mu^3} \Xi  + \frac{1}{ \kappa^2 \mu^3 \mu'} \Xi^2 \right) \varepsilon, \\
& \leq \tilde{C} \varepsilon^{19/20} \leq \varepsilon^{9/10}.
\end{split}
\end{equation}
for $ \varepsilon$ sufficiently small. It remains to prove estimates \eqref{prem Xi} and \eqref{prem eps}. According to the parameters choice, the hypothesis \eqref{hypothese sur S} is verified for $ \varepsilon$ sufficiently small.
So we can apply the Lemma~\ref{esti-f+}. From parameters choice, we have
\begin{align*}
X & =1 + \frac{3}{4\mu^2} \varepsilon^{-1/20} O(\varepsilon^\tau) + \frac{4}{3\mu^4} \varepsilon^{-1/10} O(\varepsilon^{2\tau}), \quad \tau \in [ \frac{1}{2},1], \\
& \leq 1 + \frac{3}{4\mu^2} \varepsilon^{9/20} + \frac{4}{3\mu^4} \varepsilon^{8/10} \leq 3 ,\\
  \Delta &= e^{-(\sigma-\sigma_1)N/10}=\varepsilon.
\end{align*} 
So by estimate \ref{nouv-pertur}, we have
\begin{align*}
\Xi & \leq \frac{C}{\Pi} \left( \varepsilon + \frac{1}{\mu^2} \varepsilon^{9/20} + \frac{4}{3 \mu^2} \varepsilon^{9/20} + \frac{48}{9\Pi \mu^2} (\varepsilon+1) \varepsilon^{19/20} + \frac{256}{27 \mu^4} \varepsilon^{29/20} \right)  \varepsilon + \left( \frac{3}{4} \right)^3 \Xi \\
& \leq \tilde{C} \varepsilon^{29/20} + \Xi \leq \varepsilon^{7/5} + \Xi.
\end{align*}
Similarly, by estimate \eqref{nouv-pertur-jet}, the jet of the the new perturbation satisfies:
\begin{align*}
\varepsilon_1 &  \leq  \frac{C}{\Pi} \left( \varepsilon + \frac{1}{\mu^2} \varepsilon^{19/20} + \frac{48}{9\Pi \mu^2} (\varepsilon+1) \varepsilon^{19/20} + \frac{256}{27 \mu^4} \varepsilon^{29/20} \right) 3 \varepsilon \\
& \leq \tilde{C} \varepsilon^{39/20} \leq \varepsilon^{8/5},
\end{align*}
for $ \varepsilon$ sufficiently small. This ends the proof of the first iteration.
\end{proof}
\begin{remark} \label{condition reiterer}
\begin{itemize}
\item To be able to reiterate again, it is necessary that the new frequency satisfies the separation condition A1, the transversality condition A2 and the second Melnikov condition. So it is necessary  that the new frequency is at a distance of $ \delta_0 $ from the starting frequency.
According to \eqref{hyp freq}, we set
\begin{equation} \nonumber
 \varepsilon < \delta_0 \leq \delta.
\end{equation} 
\item Similarly to be able to reiterate again, it is necessary that the matrix $\hat K$ satisfies hypothesis B \eqref{Hypothèse B}. According to estimate \eqref{hyp B}, we set
\begin{equation} \label{Condition KAM}
\varepsilon < \frac\delta8.
\end{equation} 
\end{itemize}
\end{remark}
\subsection{Choice of parameters}
In this section we will make a choice of parameters for any KAM step.

Let $k\geq 1$, According to Proposition~\ref{proposition homo NL}  we assume that we constructed a Hamiltonian normal form $h_k=h_{k-1} + \hat{h}_{k-1}=\omega_k.r+\frac{1}{2} \langle \zeta,A_k(\rho)\zeta \rangle$, a perturbation $f_k$ and a jet function $S_k$ satisfying the nonlinear homological equation at step $k$.
We choose
\begin{align*}
&\varepsilon_0=\lc f^T \rc^{\alpha,\beta,\kappa} _{\sigma,\mu,\mathcal{D}}, \quad \Xi_0=\lc f \rc^{\alpha,\beta,\kappa} _{\sigma,\mu,\mathcal{D}}, \quad \sigma_0=\sigma, \quad \mu_0=\mu,\\
& \varepsilon_k=\lc f_k^T \rc^{\alpha,\beta,\kappa_k}_{\sigma_k,\mu_k,\mathcal{D}_k}, \quad \Xi_k = \lc f_k \rc^{\alpha,\beta,\kappa_k}_{\sigma_k,\mu_k,\mathcal{D}_k}  \\
& \Xi_0=O(\varepsilon_0^\tau) \quad \tau \in [  \frac{1}{2} , 1 ], \\
&\sigma_{k} = (\frac{1}{2} + \frac{1}{2^{k+1}}) \sigma, \quad k \geq 0 \\
&\mu_{k} = (\frac{1}{2} + \frac{1}{2^{k+1}}) \mu,, \quad k\geq0,\\
& N_k=10(\sigma_{k} - \sigma_{k+1})^{-1} \ln (\varepsilon_k^{-1}), \quad k\geq 0,\\
& \kappa_k =\varepsilon_k^{1/20} \quad k\geq 0 ,\\
& \mathcal{O}(k)= \mathcal{O}^\alpha(\sigma_k,\mu_k), \quad k\geq 0,
\end{align*}
Recall that the perturbation at step $k+1$ is given by
\begin{align*}
f_{k+1} =& R_k+(f_k-f_k^T) + \lbrace f_k,S_k \rbrace - \left\lbrace  f_k-f_k^T , S_k \right\rbrace + \int_0^1 (1-t) \left\lbrace  \left\lbrace  h_k+f_k, S_k \right\rbrace  , S_k \right\rbrace  \circ \Phi_{S_k}^t dt.
\end{align*} 
By estimate \eqref{nouv-pertur}, we have
\begin{align*} \nonumber
\Xi_{k+1} \leq & C \left( \frac{\Delta_k}{\Pi_{k}} + \frac{\Xi_{k}}{\Pi_{k} \kappa_{k} \mu_{k}^2} + \frac{\mu_{k+1} \Xi_{k}}{\Pi_{k} \kappa_{k} \mu_{k}^3 } + \frac{(\Delta_k+1) X_{k} \varepsilon_k}{\Pi_{k}^2 \kappa_{k} \mu_{k+1}^2} + \frac{X_{k} \varepsilon_{k} \Xi_{k}}{\Pi_{k} \kappa_{k} \mu_{k+1}^4} \right) X_{k} \varepsilon_{k} + \left( \frac{\mu_{k+1}}{\mu_{k}} \right)^3 \Xi_{k}.
\end{align*}
where
\begin{align*}
X_{k} & = 1 + \frac{ \mu_{k+1} \Xi_k}{ \kappa_k \mu^3_k}  + \frac{\Xi_k^2}{ \kappa_k^2 \mu_k^3 \mu_{k+1}}, \\
\Delta_k & = e^{-( \sigma_k-\sigma_{k+1})N/10}, \quad \Pi_k = (\sigma_k-\sigma_{k+1})^{6(2n+1)}.
\end{align*}
The jet of the perturbation at step $k+1$ is given by
\begin{equation} \nonumber 
f_{k+1}^T= R_k + \lbrace f_k^T , S_k \rbrace^T + \left( \int_0^1 (1-t) \left\lbrace  \left\lbrace  h_k+f_k, S_k \right\rbrace  , S_k \right\rbrace  \circ \Phi_{S_k}^t dt \right)^T,
\end{equation}
and according to estimate \eqref{nouv-pertur-jet} satisfies
\begin{equation} \nonumber
\varepsilon_{k+1} \leq C \left( \frac{\Delta_k}{\Pi_k} + \frac{\varepsilon_k}{\Pi_k \kappa_k \mu_{k+1}^2 } + \frac{(\Delta_k+1) X_k \varepsilon_k}{\Pi^2_k \kappa_k \mu_{k+1}^2} + \frac{X_k \varepsilon_k \Xi_k}{\Pi_k \kappa_k \mu_{k+1}^4} \right) X_k \varepsilon_k.
\end{equation}
We can remark that $\varepsilon_0$ depends on $n,\alpha,\beta,\sigma$, $\mu$ $, |\omega_0|_{\mathcal{C}^1 \left( \mathcal{D} \right) }$ et $\vert  A_0 \vert_{\beta,{\mathcal{C}^1 \left( \mathcal{D} \right) }}$.
\begin{lemma} \label{perturabation-jet-etape-k}
For the previous choice of parameters we have  
\begin{equation} \label{perturbation k}
\Xi_{k+1} \leq \varepsilon_k^{4/5} + \Xi_k,
\end{equation}
\begin{equation} \label{jet perturbation k}
\varepsilon_{k+1} \leq  \varepsilon_k^{8/5},
\end{equation}
for $k \in \mathbb{N}$ and $\varepsilon_0$ sufficiently small.
\end{lemma}
\begin{proof}
We prove the statement by induction. For $k=0$, estimates \eqref{prem Xi} and \eqref{prem eps} are verified. Assume that 
\begin{align*}
\Xi_{k} & \leq \varepsilon_{k-1}^{4/5} + \Xi_{k-1},\\
\varepsilon_{k} & \leq  \varepsilon_{k-1}^{8/5}.
\end{align*}
Using the induction hypothesis, we have
\begin{align*}
\Xi_k \leq \sum_{1\leq j \leq k-1} \varepsilon_j^{4/5} + \varepsilon^\tau_0 \leq \sum_{1\leq j \leq k-1}\varepsilon_0^{\frac{4}{5} \left( \frac{8}{5} \right)^j} + \varepsilon^\tau_0 \leq 2 \varepsilon_0^\tau.
\end{align*}
The parameters choice gives
\begin{equation} \nonumber
\Delta_k  = e^{-( \sigma_k-\sigma_{k+1})N/10}=\varepsilon_k, \quad \Pi_k = (\sigma_k-\sigma_{k+1})^{6(2n+1)}= \left( \frac{\sigma}{2^{k+2}} \right)^{6(2(n+1)}.
\end{equation}
By estimate \eqref{nouv-pertur}, we have
\begin{align*} \nonumber
\Xi_{k+1} \leq & \frac{C}{\Pi_k}X_k \varepsilon_k^2 + \frac{C}{\Pi_k} \left( \frac{\Xi_{k}}{ \mu_{k}^2} + \frac{\mu_{k+1} \Xi_{k}}{\mu_{k}^3 } + \frac{(\varepsilon_k+1) X_{k} \varepsilon_k}{\Pi_{k} \mu_{k+1}^2} + \frac{X_{k} \varepsilon_{k} \Xi_{k}} {\mu_{k+1}^4} \right) X_{k} \varepsilon_{k}^{\frac{19}{20}} + \left( \frac{\mu_{k+1}}{\mu_{k}} \right)^3 \Xi_{k}.
\end{align*}
Remarque that 
\begin{align*}
X_k \varepsilon_k & = \varepsilon_k  + \frac{\mu_{k+1}}{\mu_k^3} \varepsilon_k^{19/20} \Xi_k + \frac{1}{\mu^3_k \mu_{k+1}} \varepsilon^{9/10} \Xi_k  \leq \tilde{C} \varepsilon_k^{9/10}
\end{align*}
This lead to
\begin{equation} \nonumber
\Xi_{k+1} \leq C \varepsilon_k^{17/20} + \Xi_k \leq \varepsilon_k^{4/5} + \Xi_k.
\end{equation}
It remains to prove $\varepsilon_{k+1} \leq  \varepsilon_k^{8/5}$. By \eqref{nouv-pertur-jet} and the parameters choice, we have
\begin{align*}
\varepsilon_{k+1} & \leq \frac{C}{\Pi_k} X_k \varepsilon_k^2 + \frac{C}{\Pi_k} \left( \frac{1}{\mu_{k+1}^2} \varepsilon_k^{\frac{19}{20}} + \frac{\varepsilon_k+1 }{\Pi_k \mu_{k+1}^2} X_k \varepsilon_k^{\frac{19}{20}} + \frac{1}{\mu_{k+1}^4} X_k \varepsilon_k^{\frac{19}{20}} \right) X_k \varepsilon_k \\
& \leq C \varepsilon_k^{17/20} \varepsilon^{9/10}_k \leq \varepsilon_k^{8/5}.
\end{align*}
This conclude the proof of the lemma.
\end{proof}
\subsection{Iterative lemma}
In this section we describe a general step of the KAM procedure. We set $h_{0}(\rho)=\omega_{0}(\rho).r+\frac{1}{2} \langle \zeta,A_{0}(\rho)\zeta \rangle$, $\mathcal{D}_0=\mathcal{D}$, $f_0=f$ where $\partial_\rho^j f_0 \in \mathcal{T}^{\alpha,\beta} (\sigma_0,\mu_0, \mathcal{D}_0)$  for $j=0,1$, $\lc f_0^T \rc^{\alpha,\beta,\kappa_0} _{\sigma_0,\mu_0,\mathcal{D}_0} = \varepsilon_0$, $\Xi_0 = \lc f_0 \rc^{\alpha,\beta,\kappa_0} _{\sigma_0,\mu_0,\mathcal{D}_0}$, and $\varepsilon= \varepsilon_0$.
\begin{lemma} \label{lem-eps}
Assume that there exists a positive constant $\varepsilon$ depending on $n,\alpha,\beta,\sigma$, $\mu$ $, |\omega_0|_{\mathcal{C}^1 \left( \mathcal{D} \right) }$ and $\vert  A_0 \vert_{\beta,{\mathcal{C}^1 \left( \mathcal{D} \right) }}$ that verifies
\begin{equation} \label{Condit itéra}
\varepsilon \leq  \frac{1}{8} \delta.
\end{equation}
Assume that
\begin{equation} \nonumber
\Xi_0 = O(\varepsilon_0^\tau), \mbox{ whith } \quad \frac{1}{2} \leq \tau \leq 1,
\end{equation}
then, for $k \geq 1$, there exists a closed subset $\mathcal{D}_k \subset \mathcal{D}_{k-1}$, a real jet-function $S_{k-1}$ such that $\partial^j_\rho S_{k-1} \in \mathcal{T}^{\alpha,\beta+} (\sigma_k,\mu_k, \mathcal{D}_k)$ for $j=0,1$, a normal form $h_k(\rho)=\omega_{k}\cdot r+\frac{1}{2} \langle \zeta,A_{k}(\rho)\zeta \rangle$ where $\rho \in \mathcal{D}_k$ and a perturbation  $f_{k}$ that satisfies $\partial^j_\rho f_{k} \in \mathcal{T}^{\alpha,\beta} (\sigma_k,\mu_k, \mathcal{D}_k)$ such that
\begin{equation} \nonumber
\Phi_{k}=\Phi_{S_{k-1}}^1(.,\rho) : \mathcal{O}(k) \longrightarrow \mathcal{O}(k-1), \quad  \rho \in \mathcal{D}_{k},\
\end{equation}
is a real analytic symplectomorphism linking the Hamiltonian at step $k-1$ and the Hamiltonian at the step $k$, i.e.
\begin{equation} \nonumber
(h_{k-1}+f_{k-1}) \circ \Phi_k = h_{k} + f_{k}.
\end{equation}
Moreover, we have
\begin{align*}
\operatorname{mes} (\mathcal{D}_{k-1} \setminus \mathcal{D}_k) &\leq \varepsilon_{k-1}^{\gamma}, \quad \mbox{ with } \gamma >0,\\
\lc f_k^T \rc^{\alpha,\beta,\kappa_k} _{\sigma_k,\mu_k,\mathcal{D}_k} & \leq  \varepsilon_k,\\
\vert \partial_\rho^j (A_{k} - A_{k-1}) \vert_\beta = \vert \hat \partial_\rho^j K_{k-1} \vert_\beta & \leq  C \varepsilon_{k-1}^{9/10}, \quad j=0,1,\\
\vert \partial_\rho^j (\omega_{k} - \omega_{k-1} )\vert & \leq  C \varepsilon_{k-1}^{9/10}, \quad j=0,1,\\
\Vert \Phi_k(x,\rho) - x \Vert_\alpha & \leq   \varepsilon_{k-1}^{4/5} \: \mbox{ for }\: x \in  \mathcal{O}(k),\:\rho\in \mathcal{D}_k,
\end{align*}
$C>0$ and depends on $n$, $\beta$, $\sigma$, $\mu$, $ |\omega_0|_{\mathcal{C}^1 \left( \mathcal{D} \right) }$ and $\vert  A_0 \vert_{\beta,{\mathcal{C}^1 \left( \mathcal{D} \right) }}$.
\end{lemma}
\begin{proof}
At the first step, by Lemma~\ref{premiere etape}, there exists a closed set $\mathcal{D}_1 \subset\mathcal{D}_0$ that satisfies:
\begin{equation} \nonumber
\operatorname{mes} (\mathcal{D}_0 \setminus \mathcal{D}_1) \leq \varepsilon_0^\gamma.
\end{equation}
For $\rho \in \mathcal{D}_1$, there exists an anlytic symplectomorphism
\begin{equation} \nonumber
\Phi_1 = \Phi_{S_{0}}^{t=1} : \mathcal{O}(1) \to \mathcal{O}(0),
\end{equation}
linking the initial Hamiltonian and the Hamiltonian at the first step
\begin{equation} \nonumber
(h+f) \circ \Phi_1 = h_1 + f_1,
\end{equation} 
where $h_{1}(\rho)=\omega_{1}.r+\frac{1}{2} \langle \zeta,A_{1}(\rho)\zeta \rangle$. By estimate \eqref{hyp B}-\eqref{proche id1}, we have:
\begin{equation} \nonumber
\vert  \partial^j_\rho \left( A_{1}(\rho) - A_0(\rho) \right)  \vert_\beta   \leq C \varepsilon_0, \quad \vert \partial^j_\rho \left( \omega_{1}(\rho) - \omega_0(\rho) \right)   \vert  \leq  C \varepsilon_0,
\end{equation}
\begin{equation} \nonumber
\Vert \Phi_1(x,\rho) - x \Vert_\alpha \leq \varepsilon^{9/10} \mbox{ for } x \in  \mathcal{O}^\alpha(\sigma,\mu),
\end{equation} 
This achieve the first step. Thanks to conditions \eqref{Condit itéra} we are able to reiterate again. Now assume that we have completed the iteration up to step $k-1$. We want to perform the step $k$.  By construction, the matrix $A_{k}$ satisfies
\begin{equation}  \nonumber
A_{k}=A_{k-1}+\hat{K}_{k-1} = A_0 + \hat K_0+ \hat K_1 + \ldots+ \hat K_{k-1}.
\end{equation} 
According to \eqref{partie quad eq hom NL} and Lemma~\ref{perturabation-jet-etape-k}, we have
\begin{equation} \nonumber
\vert \partial_\rho^j (A_{k}(\rho)-A_{k-1}(\rho)) \vert_\beta \leq \vert \partial_\rho^j \hat K_{k-1}(\rho) \vert_\beta \leq C \varepsilon^{9/10}_{k-1} \leq \frac{\delta}{8}.
\end{equation}
The frequency $\omega_{k}$ satisfies
\begin{equation} \nonumber
\omega_{k} = \omega_0 +  \underset{0 \leq j \leq k-1}{ \sum} \hat{\omega}_j.
\end{equation}
By estimate \eqref{fréquence eq homo NL} and Lemma~\ref{perturabation-jet-etape-k}, $\omega_k$ is close to $\omega_0$, i.e. 
\begin{equation} \nonumber
\vert \omega_{k} - \omega_0 \vert \leq C \underset{0 \leq j \leq k-1}{ \sum} \varepsilon_j^{9/10} \leq \delta_0.
\end{equation}
So we can apply the Proposition~\ref{proposition homo NL}: there exists a closed subset $\mathcal{D}_k \subset \mathcal{D}_{k-1}$ that Lebesgue measure satisfies
\begin{equation} \nonumber
\operatorname{mes} (\mathcal{D}_{k-1} \setminus \mathcal{D}_k) \leq C (\kappa_k \delta^{-1})^\iota N_k^\upsilon \leq \varepsilon_k^{\gamma}, \quad \mbox{ for } \gamma>0.
\end{equation} 
For $ \rho \in \mathcal{D}_k$, there exists a jet-function $S_{k-1}$ such that $\partial^j_\rho S_{k-1} \in \mathcal{T}^{\alpha,\beta+} (\sigma_k,\mu_k, \mathcal{D}_k)$ for $j=0,1$, and we have
\begin{equation} \nonumber
\llbracket S_{k-1} \rrbracket^{\alpha,\beta+,\kappa_{k-1}} _{\sigma_{k},\mu_{k},\mathcal{D}_{k}} \leq \frac{C}{\Pi_{k-1}\kappa_{k-1}}\left(1 + \frac{ \mu_{k} }{ \kappa_{k-1} \mu^3_{k-1}} \Xi_{k-1}  + \frac{1}{ \kappa^2_{k-1} \mu^3_{k-1} \mu_{k}} \Xi^2_{k-1} \right) \varepsilon_{k-1}.
\end{equation}
The symplectomorphism associate to $S_{k}$ is analytic and we have
\begin{equation} \nonumber
\Phi_{k}=\Phi_{S_{k-1}}^1(.,\rho) : \mathcal{O}(k) \longrightarrow \mathcal{O}(k-1), \quad  \rho \in \mathcal{D}_{k}.
\end{equation}
Thins transformation link the Hamiltonian at step $k-1$ and the Hamiltonian at the step $k$
\begin{equation} \nonumber
(h_{k-1}+f_{k-1}) \circ \Phi_k = h_{k} + f_{k}.
\end{equation}
By construction we have $h_k(\rho)=\omega_{k}\cdot r+\frac{1}{2} \langle \zeta,A_{k}(\rho)\zeta \rangle$. According to \eqref{partie quad eq hom NL}, Lemma~\ref{perturabation-jet-etape-k} and the choice of parameters, we have for $j=0,1$:
\begin{align*} \nonumber
\vert \partial_\rho^j (A_{k}(\rho)-A_{k-1}(\rho)) \vert_\beta & \leq \vert \partial_\rho^j \hat K_{k-1}(\rho) \vert_\beta  \\
& \leq  C \left( 1 + \frac{\mu_k\Xi_{k-1}}{\Pi_{k-1} \kappa_{k-1} \mu^3_{k-1}}  + \frac{\mu_k^2\Xi_{k-1}^2}{\Pi_{k-1} \mu_{k-1}^6 \kappa_{k-1}^2}  \right) \frac{\varepsilon_{k-1}}{\mu^2_{k-1}}\\
& \leq C \varepsilon_{k-1}^{9/10}.
\end{align*}
The new frequency $\omega_{k}$ are given by
\begin{equation} \nonumber
\omega_{k} = \omega_{k-1} +  \hat{\omega}_{k-1} ,
\end{equation}
and by estimate \eqref{fréquence eq homo NL}, Lemma~\ref{perturabation-jet-etape-k} and the choice of parameters satisfies
\begin{align*}
\vert \partial_\rho^j (\omega_{k} - \omega_{k-1}) \vert &=  \leq \Bigg[  1+ \frac{\Xi_{k-1}}{\Pi_{k-1} \kappa_{k-1} (\mu_{k-1}^2-\mu_k^2)} \Bigg( 1 + \frac{\mu_k \Xi_{k-1}}{\kappa_{k-1} \mu_{k-1}^3}  + \frac{\mu_k^2\Xi_{k-1}}{\kappa_{k-1} \mu_{k-1}^4}  \Bigg) \Bigg] \frac{\varepsilon_{k-1}}{\mu_{k-1}^2} \\
& \leq C \varepsilon_{k-1}^{9/10}.
\end{align*}
According to Proposition~\ref{estim-phi}, estimate \eqref{estim reste sol hom NL} and the choice of parameters, we have
\begin{align*} 
\Vert \Phi_k(x) - x \Vert_\alpha & \leq \frac{\llbracket S_{k-1} \rrbracket^{\alpha,\beta+,\kappa_{k-1}} _{\sigma_k,\mu_{k-1},\mathcal{D}_k}}{\mu^2_{k-1}(\sigma_{k-1} - \sigma_{k})}\\
&  \leq \frac{C}{\Pi_{k-1} (\sigma_{k-1} - \sigma_{k}) \kappa_{k-1} \mu^2_{k-1}}\Bigg(1 + \frac{ \mu_{k} }{ \kappa_{k-1} \mu^3_{k-1}} \Xi_{k-1} + \frac{1}{ \kappa^2_{k-1} \mu^3_{k-1} \mu_{k}} \Xi^2_{k-1} \Bigg) \varepsilon_{k-1} \\
& \leq C \varepsilon^{17/20}_{k-1} \leq \varepsilon^{4/5}_{k-1}.
\end{align*}
for $\varepsilon_0$ small enough.
\end{proof}
\subsection{Transition to limit and proof of theorem \ref{theoreme kam}}
Consider $\mathcal{D'}:= \underset{k\geq 1}{\cap}\mathcal{D}_k$. The Lebesgue measure of  $\mathcal{D'}$ satisfies
\begin{equation} \nonumber
\operatorname{mes} (\mathcal{D} \setminus \mathcal{D'}) \leq  c \varepsilon ^{\gamma},
\end{equation} 
where $c$ depends on $\delta$ and $\sigma$.
Consider $1 \leq j \leq M$, we define $\Phi_M^j = \Phi_{j} \circ \Phi_{j+1} \circ \ldots \circ \Phi_M$ a symplectomorphism that maps $\mathcal{O}(M) \times \mathcal{D}'$ to $\mathcal{O}(j)\times \mathcal{D}'$. Moreover, for $1\leq j \leq M$, we have:
$$ \Vert \Phi_M^j- id \Vert_\alpha \leq \sum_{k=j}^M \varepsilon_{k-1}^{4/5} \leq C \varepsilon_{j-1}^{4/5}, $$
For $P>M$, we have
$$\Vert \Phi_P^j- \Phi_M^j \Vert_\alpha \leq C \varepsilon^{4/5}_M.$$
Consequently, $(\Phi_M^j)_M$ is a Cauchy sequence that converge when $M \to \infty$ to a real analytic symplectomorphism $\Phi_\infty^j$ mapping $\mathcal{O}^\alpha(\frac{\sigma}{2},\frac{\mu}{2})$ to $\mathcal{O}(j)$. Moreover, for $\rho \in \mathcal{D}'$ we have
\begin{equation} \label{Phi infini}
\Vert \Phi_\infty^j - id \Vert_{\alpha} \leq \sum_{k\geq j} \varepsilon_{k-1}^{4/5} \leq  C \varepsilon_{j-1}^{4/5}.
\end{equation}
By Cauchy estimate, for $j\geq 1$, we have:
\begin{equation}
\Vert D \Phi_\infty^j - id \Vert_{\mathcal{L}\left( Y_\alpha,Y_\alpha \right) } \leq C \varepsilon_{j-1}^{4/5}.
\end{equation}
By construction, the map $\Phi_M^1$ transforms  the Hamiltonian
\begin{equation} \nonumber
H_1=\omega.r+\frac{1}{2} \langle \zeta,A(\rho)\zeta \rangle + f
\end{equation}
into
\begin{equation} \nonumber
H_M=\omega_M.r+\frac{1}{2} \langle \zeta,A_M(\rho)\zeta \rangle + f_M.
\end{equation}
Clearly $\omega_M$ converge to $\omega_\infty$ and $A_M$ converge to $A_\infty$. In addition, we have
\begin{align*}
\vert \omega_\infty - \omega \vert & \leq C \varepsilon_0 + C \underset{j \geq 1}{ \sum} \varepsilon_j^{9/10} \leq C \varepsilon_0 + C \sum_{j \geq 1} \varepsilon_0^{\frac{9}{10} (\frac{8}{5})^j} \leq C \varepsilon_0 + C \sum_{j \geq 0} \varepsilon_0^{\frac{36}{25} (\frac{8}{5})^j} \leq C \varepsilon_0, \\
\end{align*}
and we obtain the same estimate for $\vert A_\infty - A \vert_\beta$. Similarly, $\partial_\rho \omega_\infty$ and $\partial_\rho A_\infty$ satisfies the same estimates.

We define $H_\infty = H_1 \circ \Phi_\infty^1$, with
\begin{equation} \nonumber
H_\infty = \omega_\infty.r+\frac{1}{2} \langle \zeta,A_\infty(\rho)\zeta \rangle + f_\infty.
\end{equation}
Consider $x=(\theta,0,0) $ and $h=(\theta,r,\zeta)$, by the chain rule we have
\begin{equation} \nonumber
\langle \nabla H_\infty (x) , h \rangle = \langle \nabla H_k (\Phi_\infty^k(x)), D \Phi_\infty^k(x)h\rangle.
\end{equation}
We recall that $\lc f_k^T \rc^{\alpha,\beta,\kappa_k} _{\sigma_k,\mu_k,\mathcal{D}'}  \leq \varepsilon_k$ for $ k \geq 1$, then $\nabla H_k (\Phi_\infty^k(x))= ^t(0,\omega_k,0)+O(\varepsilon_k^{4/5})$. Recall also, for $j\geq 1$, that $\Vert \Phi_\infty^j - id \Vert_\alpha \leq  C \varepsilon_{j-1}^{4/5}$. So 
$$\nabla H_\infty=^t(0,\omega_\infty,0).$$ This allows us to deduce that
\begin{equation} \nonumber
\partial_r f_\infty(\theta,0,0)= \partial_\zeta f_\infty(\theta,0,0)=0.
\end{equation}
Consider now the matrix $\partial^2_{\zeta_i\zeta_j} H_\infty(x)$. We have
\begin{equation} \nonumber
\partial^2_{\zeta_i\zeta_j} H_k(x)=(A_k)_{i,j} + O(\varepsilon_k^{4/5}).
\end{equation}
This leads to $\partial^2_{\zeta_i\zeta_j} H_\infty(x)=(A_\infty)_{i,j}$ and to deduce that
\begin{equation} \nonumber
\partial^2_{\zeta\zeta} f_\infty(\theta,0,0)=0
\end{equation}

This completes the proof of Theorem~\ref{theoreme kam}.

\section{Wave equation with a convolutive potential}
We consider the convolutive wave equation on the circle:
\begin{equation} \label{eq onde}
u_{tt} - u_{xx} + V \star u + \varepsilon g(x,u) = 0, \quad t \in \mathbb{R},\: x \in \mathbb{S}^1,
\end{equation}
where $g$ is a real holomorphic function on $\mathbb{S}^1\times J$, for $J$ some neighborhood of the origin of $\mathbb{R}$. The convolution potential $V : \mathbb{S}^1 \to \mathbb{R}$ is supposed to be holomorphic with real Fourier coefficients $\hat{V}(a), a \in \mathbb{Z} $, satisfying
\begin{equation} \label{wave_conv_assumption1}
a^2 +  \hat{V}(a)>0 ,  \quad \forall \: a \in \mathbb{Z}.
\end{equation}  
Consider $ \mathcal{A}$ a finite set of $\mathbb{Z}$ of cardinality $n$. We define the set $ \mathcal{L} := \mathbb{Z} \setminus \mathcal{A}$ and the parameter of the equation $\displaystyle \rho := \left( \hat V (a) \right)_{a \in \mathcal{A}}$. We assume that the parameter $\rho=(\rho_{a_1}, \ldots, \rho_{a_n})$ belongs to the set $\displaystyle \mathcal{D}= \left[ b_{a_1}, c_{a_1} \right]  \times \ldots \times \left[ b_{a_n}, c_{a_n} \right]$ and all other Fourier coefficients are fixed. We denote $\omega(\rho)=\left( \omega_a (\rho) \right)_{a \in \mathcal{A}} =\left( \sqrt{a^2+\rho_a} \right)_{a \in \mathcal{A}}$ and $\lambda_s = \sqrt{s^2+ \hat{V}(s)}$ for $s \in \mathcal{L}$. We also suppose that
\begin{equation} \label{wave_conv_assumption2}
\lambda_s \neq \lambda_{s'}, \quad \forall \: s,s' \in \mathcal{L}, \quad s \neq \pm s'.
\end{equation}
Introducing $v=\dot u$, the equation \eqref{eq onde} becomes:
\begin{equation*}
\dot{u}= v, \quad \dot{v} = - (\Lambda^2 u+\varepsilon g(x,u)),
\end{equation*}
where $\Lambda:=(\sqrt{-\partial_{xx}+ V \star})$.
Defining $\psi:=\frac{1}{\sqrt{2}} ( \Lambda^{\frac{1}{2}}u-i\Lambda^{-\frac{1}{2}}v)$, we get the following equation for $\dot \psi$:
\begin{equation} \nonumber
\frac{1}{i} \dot \psi = \Lambda \psi + \varepsilon \frac{1}{\sqrt{2}} \Lambda^{-1/2}g\left( x,\Lambda^{-1/2}\left( \frac{\psi+\bar{\psi}}{\sqrt{2}} \right)\right) .
\end{equation}
Let us endow $L^2(\mathbb{S}^1,\mathbb{C})$ with the classical real symplectic form $-i d \psi \wedge d \bar{\psi}=-du \wedge dv$ and consider the following Hamiltonian:
\begin{equation} \nonumber
H(\psi,\bar{\psi}) = \int_{\mathbb{S}^1} (\Lambda \psi)\bar{\psi} dx + \varepsilon \int_{\mathbb{S}^1} G \left( x,\Lambda^{-1/2} \left( \frac{\psi+\bar{\psi}}{\sqrt{2}} \right)  \right) dx,
\end{equation}
where $ G $ is a primitive of $ g $ with respect to $ u $: $ g=\partial_uG $. Then, \eqref{eq onde} becomes a Hamiltonian system:
\begin{equation} \nonumber
\dot{\psi} = i \frac{\partial H}{\partial \bar{\psi}}.
\end{equation}

Consider now the complex Fourier orthonormal basis given by $\lbrace \varphi_a(x)=\frac{e^{iax}}{\sqrt{2\pi}} ,\:a\in \mathbb{Z}\rbrace$. In this base, the operator $\Lambda$ is diagonal, and we have:
\begin{equation} \nonumber
\Lambda \varphi_a= \sqrt{a^2 + \hat{V}(a)} \: \varphi_a.
\end{equation}
Let us decompose $\psi$ and $\bar{\psi}$ in this basis: $ \psi=\underset{s \in \mathbb{Z}}{ \sum}\xi_s\varphi_s$ and $\bar{\psi}= \underset{s \in \mathbb{Z}}{ \sum} \eta_s \varphi_{-s}$.
By injecting this decomposition into the expression of $H$, we obtain:
\begin{equation} \label{Hamiltonien_0}
H= \sum_{a \in \mathcal{A}} \omega_a(\rho) \xi_a \eta_a + \sum_{s \in \mathcal{L}} \lambda_s \xi_s \eta_s + \varepsilon \int_{\mathbb{S}^1} G \left( x , \underset{s \in \mathbb{Z}}{ \sum} \frac{\xi_s\varphi_s+\eta_s \varphi_{-s}}{\sqrt{2\lambda_s}} \right) dx.
\end{equation}
Let $\mathcal{P}_{\mathbb{C}}:= \ell^2(\mathbb{Z},\mathbb{C}) \times \ell^2(\mathbb{Z},\mathbb{C})$ that we endow with the complex symplectic form $-i{ \sum_{s \in \mathbb{Z}}} d\xi_s \wedge d\eta_s$. We define the subspace $ \mathcal{P}_{\mathbb{R}}:= \lbrace (\xi,\eta) \in \mathcal{P}_{\mathbb{C}} | \eta_s=\bar{\xi}_s \rbrace$.
Then, equation \eqref{eq onde} is equivalent to the following Hamiltonian system on $\mathcal{P}_{\mathbb{R}}$:
\begin{equation} \label{onde2}
\dot{\xi}_s = i \frac{\partial H}{\partial \eta_s}, \quad \dot{\eta}_s = -i \frac{\partial H}{\partial \xi_s}, \quad s \in \mathbb{Z}
\end{equation}
From now, we write $H=h+ \varepsilon f $, where
\begin{equation} \label{hamiltonien_new}
h= \sum_{a \in \mathcal{A}} \omega_a(\rho) \xi_a \eta_a + \sum_{s \in \mathcal{L}} \lambda_s \xi_s \eta_s, \quad \mbox{and} \quad  f=\int_{\mathbb{S}^1} G \left( x , \underset{s \in \mathbb{Z}}{ \sum} \frac{\xi_s\varphi_s+\eta_s \varphi_{-s}}{\sqrt{2\lambda_s}} \right)dx.
\end{equation}
Let us fix a vector $I=(I_a)_{a\in\mathcal{A}}$ with positive components (i.e. $I_a > 0$ for all $ a \in \mathcal{A}$).
Let $T_I^n$ be the real torus of dimension $n$ defined by
\begin{equation} \nonumber
T_I^n=\left\{
\begin{array}{l c l}
\xi_a=\bar{\eta}_a, \: |\xi_a|^2=I_a \quad  &\text{if}&\: a \in \mathcal{A},\\
\xi_s=\eta_s=0 \quad & \text{if} & \: s \in \mathcal{L}=\mathbb{Z} \setminus  \mathcal{A}.
\end{array}
\right.
\end{equation}

This torus is invariant by the Hamiltonian flow when the perturbation $ f $ is zero and it is linearly stable. We can even give the analytic expression of the solution of the linear equation.

Our purpose is to prove the persistence of the torus $T_I^n$ when the perturbation $ f $ is no longer zero.

In a neighborhood of the invariant torus $T_I^n$ in $\mathbb{C}^{2n}$, we define the action-angle variables $ (r_a, \theta_a) _ {\mathcal {A}} $ by:
\begin{equation*}
\left\{
\begin{array}{l l l}
\xi_a=\sqrt{(I_a+r_a)} e^{i\theta_a},\\
\eta_a=\bar{\xi}_a.
\end{array}
\right.
\end{equation*}
In these new variables, the Hamiltonian becomes, up to a constant,
\begin{equation} \nonumber
H= \underset{a \in \mathcal{A}}{ \sum} \omega_a( \rho) r_a +   \underset{s \in \mathcal{L}}{ \sum} \lambda_s \xi_s\eta_s + \varepsilon \int_{\mathbb{S}^1} G(x,\hat{u}_{I,V}(r,\theta,\xi,\eta))dx,
\end{equation}
with
\begin{align*} \nonumber
\hat{u}_{I,V}(r,\theta,\xi,\eta)& =\underset{a \in \mathcal{A}}{ \sum} \sqrt{(I_a+r_a)} \frac{e^{-i\theta_a}\varphi_a(x)+e^{i\theta_a}\varphi_{-a}(x)}{\sqrt{2 \omega_a}}  + \underset{s \in \mathcal{L}}{ \sum}  \frac{\xi_s \phi_s(x)+\eta_{-s}\phi_s(x)}{\sqrt{2\lambda_s}}.
\end{align*}
We set $u_{I,V}(\theta,x)=\hat{u}_{I,V}(0,\theta,0,0)$. Then, for any $I \in \mathbb{R}_+^\mathcal{A}$ and $\theta_0 \in \mathbb{S}^1$, the function $(t,x) \mapsto u_{I,V}(\theta_0+t\omega,x)$ is solution of the linear wave equation. In this case, the torus $T_I^n$ is invariant and linearly stable.
Our goal is to state a similar result when the perturbation is not zero (in the nonlinear case) by applying the Theorem~\ref{theoreme kam}. For this, we have to verify the assumptions A1, A2, A3 and that the nonlinearity $f$ belongs to the right space.

Due to assumption \eqref{wave_conv_assumption1} and \eqref{wave_conv_assumption2}, the hypothesis A1 holds trivially. The hypothesis A2 also holds since in each case the second alternative is fulfilled. More precisely, for $s \in \mathcal{L}$, the frequency $\lambda_s$ does not depend on the parameter $ \rho$. So it's enough to show that there exists a unit vector $z_k \in \mathbb{R}^n$ such that
\begin{equation} \nonumber
\langle \partial_\rho ( k\cdot\omega(\rho)),z_k \rangle \geq \delta \quad \forall \rho \in \mathcal{D};
\end{equation}
for $k \neq 0$ and suitable $\delta$. This hypothesis is fulfilled for $ z_k = k/ \vert k \vert  $. Let us prove now that the hypothesis A3 holds. Consider $N>0$, $0 < \kappa < \delta$ and the following set
\begin{equation} \nonumber
J(k,s,s') = \lbrace \rho  \in \mathcal{D} \mid \left| k \cdot \omega(\rho) + \lambda_s-\lambda_{s'}  \right| < \kappa \rbrace.
\end{equation}
By hypothesis A2, the Lebesgue measure of $J(k,s,s')$ is bounded by $C \kappa \delta^{-1}$, where $C$ depends on $\mathcal{D}$. For $p \in \mathbb{Z}$ and $k \in \mathbb{Z}^n$, we define the set $ W(k,p) = \lbrace \rho \in \mathcal{D} \mid \left| k \cdot \omega( \rho) + p \right| < 2 \kappa \rbrace$.
By A2, its Lebesgue measure is bounded by $ C \kappa \delta^{-1}$. Let $W = \lbrace \rho \in \mathcal{D} \mid \left| k \cdot \omega (\rho) \cdot k + p \right| < 2\kappa \rbrace$, then
\begin{equation} \nonumber
\operatorname{mes} \left(  W \right) \leq \sum \limits_{\underset{\vert k \vert \leq  N}{k \in \mathbb{Z}^n}} \sum \limits_{\underset{\vert p \vert < CN}{p \in \mathbb{Z}}}  W(k,p) \leq C N^{n+1} \kappa \delta^{-1}.
\end{equation}
For $s \in  \mathcal{L}$, we note that $\left|  \lambda_s - \vert s \vert  \right|  \leq  \frac{\tilde C}{\vert s \vert}$, where $ \tilde{C}$ depends on the potential $V$. If $ \vert s \vert > \vert s' \vert > 2 \tilde{C} k^{-1}$, then $\left| \lambda_s - \lambda_{s'} - \left( \vert s  \vert - \vert s' \vert  \right)    \right| \leq  \kappa.$ So, if $\rho \in \mathcal{D} \setminus W$ and $\vert s  \vert > \vert s' \vert >2 \tilde{C} k^{-1}$, we obtain that
\begin{align*}
\left| \omega(\rho)+\lambda_{s} - \lambda_{s'} \right| & \geq \kappa.
\end{align*}
It remains to look at the cases where $\min (\vert s \vert, \vert s' \vert) < 2 \tilde{C} k^{-1}$ and there is $ k \in \mathbb{Z}^n$ such that
\begin{equation} \nonumber
\left| \omega(\rho)+\lambda_{s} - \lambda_{s'} \right| < 1,
\end{equation}
for $ \vert k \vert < N$. We remark in those cases that $\left| \vert s \vert - \vert s' \vert \right| \leq CN$.  Consider the set:
\begin{equation} \nonumber
\mathcal{Q} = \left\lbrace (s,s') \in \mathbb{Z}^2 \mid \min( \vert s \vert, \vert s' \vert)  < 2 \tilde{C} k^{-1} \:\text{and} \: \left| \vert s \vert - \vert s' \vert \right| \leq CN   \right\rbrace.
\end{equation}
We remark that $\operatorname{Card} \left(  \mathcal{Q} \right)  \leq C N \kappa^{-2}$.
So if we restrict $ \rho $ to
\begin{equation} \nonumber
\mathcal{D}' = \mathcal{D} \setminus (  W \bigcup_{\stackrel{\vert k \vert \leq  N}{(s,s') \in \mathcal{Q}}} \left( J (k,s,s')\right)) 
\end{equation}
we get
\begin{equation} \nonumber
\left| k\cdot\omega(\rho)+\lambda_{s} - \lambda_{s'} \right| \geq \kappa.
\end{equation}
Moreover,
\begin{equation*}
\operatorname{mes} \left(  \mathcal{D} \setminus \mathcal{D}' \right)  \leq \operatorname{mes}  \left(  W \right)  +  \sum \limits_{\underset{\vert k \vert \leq  N}{k \in \mathbb{Z}^n}} \sum_{(s,s') \in \mathcal{Q}} \operatorname{mes} J(k,s,s') \leq CN^{n+1} \kappa \delta^{-1}.
\end{equation*}
Using this, we prove easily that \ref{Melnikov-measure} and \ref{Melnikov-cond} are fulfilled for suitable positive exponent $ \tau$ and $ \iota $ .

It remains to prove that the nonlinearity $f$ belongs to the right space. We denote by $\mathcal{T}^{\alpha,\beta}(\mu)$ the set of functions of $\mathcal{T}^{\alpha,\beta}(\mathcal{D},\sigma,\mu)$ that do not depend on $r,\theta$ and $\rho$. It remains to verify that $f \in \mathcal{T}^{\alpha,\beta}(\mu)$ for some choice of $ \alpha, \beta$ and $\mu$. We will prove that $f \in \mathcal{T}^{\alpha,1/2}(\mu)$ for $\alpha > 0$.  It suffices to show that
\begin{equation} \nonumber
\nabla f \in Y^\alpha \cap L_{1/2}\quad \text{ and } \quad \nabla^2f \in \mathcal{M}_{1/2}.
\end{equation}
Recall that for $x \in \mathbb{S}^1$, we have:
\begin{equation} \nonumber
u(x) = \sum_{s \in \mathbb{Z}}\frac{\xi_s\varphi_s(x)+\eta_s \varphi_{-s}(x)}{\sqrt{2\lambda_s}}=u(\zeta)(x),
\end{equation}
where $\zeta= (\xi_s, \eta_s)_{s \in \mathbb{Z}}$. By Cauchy-Schwarz inequality, there exists a constant $C_\alpha$ depending on $\alpha$, such that for $\zeta  \in \mathcal{O}_\mu(Y_\alpha)$ we have
\begin{equation} \nonumber
\vert u(\zeta)(x) \vert \leq C_\alpha \Vert \zeta \Vert_\alpha \leq C_\alpha \mu .
\end{equation}
For $\alpha\geq 0$, we define the following space:
\begin{equation} \nonumber
Z_\alpha= \left\lbrace v=\left( v_s \in \mathbb{C}, \: s \in \mathbb{Z} \right) \:| \: \left( \vert v_s \vert \langle s\rangle^{\alpha}\right)_s \in \ell^2\left( \mathbb{Z} \right)  \right\rbrace .
\end{equation}
For $v \in Z_\alpha$, we define the Fourier transform $\mathcal{F}(v)$ of $v$ by  $u(x)=\mathcal{F}(v):=~ \sum v_s e^{isx}$. We also define the discrete Sobolev space by
\begin{equation} \nonumber
H^\alpha (\mathbb{S}^1)=\left\lbrace u \,|\, u(x)=\sum_{s\in \mathbb{Z}} \hat{u}(s)e^{isx} | \: \left( \vert\hat{u}(s) \vert \langle s\rangle^{\alpha}\right)_s \in \ell^2\left( \mathbb{Z} \right) \right\rbrace .
\end{equation}
If $\alpha \in \mathbb{N}$, then
\begin{equation} \nonumber
H^\alpha (\mathbb{S}^1)=\left\lbrace u \, | \, u(x)=\sum_{s\in \mathbb{Z}} \hat{u}(s)e^{isx} | \: \left( \widehat{ \partial^\alpha u}(s)\right) _s \in \ell^2\left( \mathbb{Z} \right) \right\rbrace .
\end{equation}
So we have the following equivalence:
\begin{equation} \label{eqiv sob}
u \in H^\alpha(\mathbb{S}^1) \Longleftrightarrow \left( \hat{u}(s)\right)_s \in Z_\alpha.
\end{equation}
\begin{itemize}
\item To prove that $\nabla_\zeta f \in Y_\alpha$, it is sufficient to prove, for example, that $\frac{\partial f}{\partial \xi} \in Z_\alpha$.
We have
\begin{align*}
\frac{\partial f}{ \partial \xi_s}  (\zeta) & = \frac{1}{\sqrt{2 \lambda_s}} \int_{\mathbb{S}^1} \partial_u G \left( x, u(\zeta)(x) \right) \phi_s(x)dx.
\end{align*}
The map  $(x,u) \mapsto g(x,u)$ is real holomorphic on a neighborhood of  $\: \mathbb{S}^1 \times J$, so $x \mapsto  \partial_u f \left( x, u(\zeta)(x) \right) \in H^\alpha (\mathbb{S}^1)$. We deduce from equivalence \eqref{eqiv sob} that $\frac{\partial f}{\partial \xi} \in Z_\alpha$.
\item Let us prove now that $\nabla^2 f \in \mathcal{M}_{1/2}$. Recall that:
\begin{equation} \nonumber
\vert  \nabla^2 f \vert_{1/2} = \sup_{s,s' \in \mathbb{Z}} \langle s \rangle^{1/2}\langle s' \rangle^{1/2} \left\Vert \frac{\partial^2 f}{\partial\zeta_s \partial \zeta_{s'}} \right\Vert_\infty .
\end{equation}
We have
\begin{equation} \nonumber
\frac{\partial^2 f}{\partial \xi_s \xi_{s'}}= \frac{1}{2 \lambda_s^{1/2} \lambda_{s'}^{1/2}} \int_{\mathbb{S}^1} \partial_u^2 G (x,u(\zeta)(x)) \varphi_s(x) \phi_{s'}(x)dx.
\end{equation}
Then
\begin{equation} \nonumber
\frac{\partial^2 f}{\partial \zeta_s \zeta_{s'}}=  \frac{1}{2 \lambda_s^{1/2} \lambda_{s'}^{1/2}} \begin{pmatrix}
\widehat{\partial^2_u G} (s+s')& \widehat{\partial^2_u G} (s-s') \\
\widehat{\partial^2_u G} (-s+s')& \widehat{\partial^2_u G} (-s-s')
\end{pmatrix},
\end{equation}
which leads to
\begin{equation} \nonumber
\vert  \nabla^2 f \vert_{1/2} = \sup_{s \in \mathbb{Z}} \left\vert \widehat{\partial^2_u G} (s) \right\vert  < \infty.
\end{equation}
\item To conclude, we have to show that $\nabla f \in L_{1/2}$. Recall that for  $\beta \leq \alpha$, we have $Y_\alpha \subset L_\beta$. So $\nabla f \in Y_1 \subset L_{1/2}$. This achieve the verification of the assumptions of Theorem~\ref{theoreme kam} and gives the following result:
\end{itemize}
\begin{theorem} \label{théorème onde}
Let $ \alpha > 0$. There exist $ \varepsilon_0$, $\gamma$, $C>0$ such that for $0 < \varepsilon \leq \varepsilon_0$ there exists a Borel set $ \mathcal{D}' \subset \mathcal{D}$ asymptotically of full Lebesgue measure, i.e.
\begin{equation} \nonumber
\operatorname{mes}\left( \mathcal{D} \setminus  \mathcal{D}'  \right) \leq C \varepsilon^\gamma,
\end{equation}
where $\gamma$ depends on $n$. For $ \rho \in \mathcal{D}'$, there exists:
\begin{enumerate}
\item a function $u(\theta,x)$ analytic in $\theta \in \mathbb{T}^n_{\sigma/2}$ and of class $H^\alpha$ in $x \in \mathbb{S}^1$ such that:
\begin{equation} \nonumber
\underset{\theta \in \mathbb{R}}{ \sup} \parallel u(\theta,.)-u_{I,V}(\theta,.) \parallel_{H^\alpha} \leq \tilde{C} \varepsilon^{4/5},
\end{equation}
with $\tilde{C}$ an absolute constant.
\item a mapping $\omega': \mathcal{D}' \rightarrow \mathbb{R}^n $ verifying:
\begin{equation} \nonumber
\vert \omega' (\rho) - \omega (\rho) \vert_{\mathcal{C}^1(\mathcal{D}')} \leq \bar C \varepsilon
\end{equation}
such that for $\rho \in \mathcal{D}'$ the function
\begin{equation}  \nonumber
t\mapsto u(\theta+t\omega',x)
\end{equation}
is solution of the wave equation \eqref{eq onde}. This solution is linear stable.
\end{enumerate}

\end{theorem}

\appendix
\section*{Appendix}
\begin{proof}[Proof of Remark \ref{melnikov}]
We prove in the same way the three estimates on the small divisors. We will choose to prove the last one. Let $ 0 <  \kappa < \frac{1}{2} \delta$, $N> 1$ and $s , s' \in \mathcal{L}$.
There are two possible cases, we have either
\begin{equation} \nonumber
\vert k \cdot \omega' (\rho) + \lambda_s +\lambda_{s'}\vert  \geq \delta (\langle s\rangle + \langle s\rangle) \geq \kappa (\langle s\rangle+\langle s\rangle),\quad \forall \rho \in \mathcal{D},
\end{equation}
or there exists a unit vector $z_k \in \mathbb{R}^p$ such that:
\begin{equation} \nonumber
\langle \partial_\rho ( k \cdot \omega'(\rho) + \lambda_s + \lambda_{s'}),z_k \rangle \geq \delta \quad \forall \rho \in \mathcal{D}.
\end{equation}
Let us consider the second case and assume that $\max( \vert s \vert , \vert s' \vert )\leq 2 Cc_0^{-1} N \delta^{-1}$ where $ C= \vert \omega \vert_{\mathcal{C}^1 \left( \mathcal{D} \right)}$. Consider the following set
\begin{equation} \nonumber
J(k,s,s') = \lbrace \rho \in \mathcal{D} \mid \vert k \cdot \omega'(\rho) + \lambda_s + \lambda_{s'} \vert < \kappa ( \langle s\rangle + \langle 's\rangle)\rbrace.
\end{equation}
The Lebesgue measure of that set satisfies
\begin{equation} \nonumber
\operatorname{mes} (J(k,s,s') ) \leq \kappa \delta^{-1}(\langle s\rangle + \langle s' \rangle)  \leq 2 C c_0^{-1} \kappa \delta^{-1}N.
\end{equation}
We define the set  $\mathcal{B}_2 := \lbrace (k,s,s') \in \mathbb{Z}^n \times \mathcal{L} \times \mathcal{L} \mid \vert k \vert \leq N ,\; \max( \vert s \vert , \vert s' \vert )\leq 2 C c_0^{-1} N  \rbrace$. This set contains at most $16C c_0^{-2} N^{2n+1} \delta^{-2}$ points. For $\rho \in  \displaystyle \mathcal{D}_1 := \mathcal{D} \setminus \bigcup_{(k,s,s') \in  \mathcal{B}_2} J(k,s,s')$, we have:
\begin{equation} \nonumber
\vert k \cdot \omega' (\rho) + \lambda_s + \lambda_s \vert  \geq \kappa ( \langle s\rangle + \langle s'\rangle),\quad \forall \rho \in \mathcal{D}.
\end{equation}
Moreover 
\begin{equation} \nonumber
\operatorname{mes} ( \mathcal{D} \setminus \mathcal{D}_1^3 ) \leq \tilde C \kappa \delta^{-1} N^{2(n+1)}.
\end{equation}
Assume now that $\max( \vert s \vert , \vert s' \vert ) > 2 C c_0^{-1} N\delta^{-1}$, by the first separation condition, we have
\begin{align*}
\vert k \cdot \omega' (\rho) + \lambda_s + \lambda_{s'} \vert  & \geq \lambda_s + \lambda_{s'}- \vert \omega' \cdot k \vert  \geq c_0 (\langle s\rangle + \langle s' \rangle)  - \frac{1}{2} c_0^{-1} ( \langle s\rangle + \langle s' \rangle)\\ 
& \geq \kappa ( \langle s\rangle + \langle s' \rangle).
\end{align*} 
\end{proof}

\begin{lemma} \label{série cvg}
Let $s \in \mathbb{N}$ and $\gamma>0$, then $$\sum_{k\in \mathcal{L}} \frac{1}{\langle k \rangle^{2\gamma}(1+|\: | k|-|s| \:|)} \leq C,$$
where $C$ is  a positive constant and depends on $\gamma$ and does not depend on $s$.
\end{lemma}
\begin{proof}\begin{itemize}
\item If $\gamma >1/2$, then
$$\sum_{k\in \mathcal{L}} \frac{1}{\langle k \rangle^{2\gamma}(1+|\: | k|-|s| \:|)} \leq\sum_{k\in \mathcal{L}} \frac{1}{\langle k \rangle^{2\gamma}} \leq C.$$
\item If $0<\gamma \leq 1/2$, then there exists $p>\frac{1}{2\gamma}\geq 1$. Let $q=1+\frac{1}{p-1}$, then $p>1$, $q>1$ and $\frac{1}{p}+\frac{1}{q}=1$. By Young's inequality for products, we have
\begin{eqnarray*}
\sum_{k\in \mathcal{L}} \frac{1}{\langle k \rangle^{2\gamma}(1+|\: | k|-|s| \:|)} & \leq & \frac{1}{p} \sum_{k\in \mathcal{L}} \frac{1}{\langle k \rangle^{2\gamma p}} + \frac{1}{q} \sum_{k\in \mathcal{L}} \frac{1}{(1+|\: | k|-|s| \:|)^{q}} \\
& \leq & \frac{1}{p} \sum_{k\in \mathcal{L}} \frac{1}{\langle k \rangle^{2\gamma p}} +  \frac{1}{q} \sum_{k \in \mathbb{Z} } \frac{1}{(1+|k|)^{q}} \leq C.
\end{eqnarray*} 
\end{itemize}
\end{proof}

\begin{proof}[Proof of Lemma \ref{norme}]
[1.] For $s$, $s'\in \mathcal{L}$, we have
\begin{equation*}
\Vert (AB)_s^{s'} \Vert_\infty  \leq  \sum_{k \in \mathcal{L}} \Vert A_s^k \Vert_\infty \Vert B_k^{s'} \Vert_\infty \leq  \frac{\vert A \vert_{\beta+} \vert B \vert_\beta}{\langle s \rangle^\beta\langle s' \rangle^\beta} \sum_{k\in \mathcal{L}} \frac{1}{\langle k \rangle^{2\beta}(1+|\: | k|-|s| \:|)}.
\end{equation*}
We conclude by Lemma~\ref{série cvg}.
\\$\left[ 2.\right] $ For $s\in \mathcal{L}$ we have
\begin{equation*}
\vert (A \zeta )_s \vert  \leq 2 \sum_{k \in \mathcal{L}} \Vert A_s^k \Vert_\infty \vert \zeta_k \vert  \leq  \frac{\vert A \vert_{\beta+} \vert \zeta \vert_{\beta}}{\langle s \rangle^\beta} \sum_{k\in \mathcal{L}} \frac{2}{\langle k \rangle^{2\beta}(1+|\:|k|-|s|\:|)}.
\end{equation*}
Similarly, we conclude by Lemma~\ref{série cvg}.
\\$\left[ 3.\right] $ For $s\in \mathcal{L}$ we have
\begin{equation*}
\vert (A \zeta )_s \vert  \leq  2 \sum_{k \in \mathcal{L}} \Vert A_s^k \Vert_\infty \vert \zeta_k \vert  \leq  \frac{\vert A \vert_{\beta} \vert \zeta \vert_{\beta+}}{\langle s \rangle^\beta} \sum_{k\in \mathcal{L}} \frac{2}{\langle k \rangle^{2\beta+1}} \leq  C  \vert A \vert_{\beta} \vert \zeta \vert_{\beta+}.
\end{equation*}
[4.] For $s\in \mathcal{L}$, we have
\begin{equation*}
\vert (A \zeta )_s \vert  \leq  2 \sum_{k \in \mathcal{L}} \Vert A_s^k \Vert_\infty \vert \zeta_k \vert  \leq  \frac{\vert A \vert_{\beta+} \vert \zeta \vert_{\beta+}}{\langle s \rangle^\beta} \sum_{k\in \mathcal{L}} \frac{2 \langle s \rangle}{\langle k \rangle^{2\beta+1}(1+|\:|k|-|s|\:|)} .
\end{equation*}
We note that
\begin{align*}
\sum_{k\in \mathcal{L}} \frac{\langle s \rangle}{\langle k \rangle^{2\beta+1}(1+|\:|k|-|s|\:|)} & \leq \sum_{|\:|k|-|s|\:|\leq |s|/2} \frac{ | s |}{\langle k \rangle^{2\beta+1}(1+|\:|k|-|s|\:|)} \\
& +\sum_{|\:|k|-|s|\:|> |s|/2} \frac{ | s |}{\langle k \rangle^{2\beta+1}(1+|\:|k|-|s|\:|)}.
\end{align*}
The second series is bounded by the convergent series $\underset{k \in \mathbb{Z} }{\sum} \frac{2}{\langle k \rangle^{2\beta+1}}$. The first series is bounded by $\underset{k \in \mathbb{Z} }{\sum} \frac{ 2}{\langle k \rangle^{2\beta}(1+|\:|k|-|s|\:|)}$. We conclude by Lemma~\ref{série cvg}.
\\$\left[5.\right]$  For $s$, $s'\in \mathcal{L}$ we have
\begin{equation*}
\Vert (AB)_s^{s'} \Vert_\infty  \leq  2 \vert X_s \vert \vert Y_{s'} \vert  \leq \frac{2}{\langle s \rangle^\beta \langle s' \rangle^\beta} \vert X \vert_\beta \vert Y \vert_\beta.
\end{equation*}
Similarly we prove the last assertion.
\end{proof}

\bibliographystyle{plain}
\bibliography{biblio2}
\end{document}